\documentclass[11pt,reqno]{amsart}
\usepackage{hyperref}
\usepackage{amsmath}
\usepackage{amsthm}
\usepackage{amsfonts}
\usepackage{amssymb}

\usepackage{latexsym,amsmath}
\usepackage{geometry}
\usepackage{fixmath}

\usepackage[dvips]{graphicx}
\usepackage{graphicx}
\usepackage{fullpage}
\usepackage{appendix}

\usepackage{color} 
\usepackage{times}

\usepackage{leftidx}
\usepackage{scrextend}
\usepackage{float}
\usepackage{graphicx}
\usepackage[symbol]{footmisc}
\usepackage[shortlabels]{enumitem}

\newtheorem{theorem}{Theorem}[section]
\newtheorem{lemma}[theorem]{Lemma}

\newtheorem{proposition}[theorem]{Proposition}
\newtheorem{claim}[theorem]{Claim}
\newtheorem{corollary}[theorem]{Corollary}

\newtheorem{definition}[theorem]{Definition}

\newcommand{\cE}{{\mathcal{E}}}

\newcommand{\cH}{{\mathcal{H}}}
\newcommand{\cI}{{\mathcal{I}}}
\newcommand{\cJ}{{\mathcal{J}}}

\newcommand{\cM}{{\mathcal{M}}}
\newcommand{\cP}{{\mathcal{P}}}
\newcommand{\cQ}{{\mathcal{Q}}}

\newcommand{\cS}{{\mathcal{S}}}

\newcommand{\bA}{{{\mathbold A}}}
\newcommand{\bB}{{\mathbold{B}}}
\newcommand{\bC}{{\mathbold{C}}}
\newcommand{\bD}{{\mathbold{D}}}
\newcommand{\bF}{{\mathbold{F}}}

\newcommand{\bH}{\mathbold{H}}
\newcommand{\bI}{{\mathbold{I}}}

\newcommand{\bL}{\mathbold{L}}
\newcommand{\bM}{{\mathbold{M}}}

\newcommand{\bS}{\mathbold{S}}

\newcommand{\bW}{{\mathbold{W}}}
\newcommand{\bX}{{\mathbold{X}}}

\newcommand{\bq}{\mathbold{q}}

\newcommand{\bgamma}{\mathbold{\gamma}}
\newcommand{\bxi}{\mathbold{\xi}}

\newcommand{\bpsi}{{\mathbold{\psi}}}

\newcommand{\walkT}{{T}} 
\newcommand{\KWalks}{\textrm{MAX-$k$}}
\newcommand{\NWalks}{\textrm{MAX-$n$}}
\newcommand{\Tall}{ {\rm T}_{\textrm{MAX-$k$} } } 

\newcommand{\Tsaw}{{T_{\rm SAW}}}
\newcommand{\cp}{{\tt A}}

\newcommand{\maxDeg}{{\Delta}}
\newcommand{\dcritical}{{\Delta_c}}
\newcommand{\lcritical}{{\lambda_c}}

\newcommand{\Pweight}{{\bW}}

\newcommand{\simpleadj}{{\bf A}}
\newcommand{\powadj}{{\mathbold{B}}}
\newcommand{\infmatrix}{{\cI}}
\newcommand{\NBMatrix}{\bH}
\newcommand{\OrntEdges}{\overrightarrow{E}}
\newcommand{\vtooedge}{{\bf K}}
\newcommand{\oedgetov}{{\bf C}}
\newcommand{\eigenval}{\theta}
\newcommand{\spradius}{{\rho}}
\newcommand{\hspradius}{\nu}

\newcommand{\maxeigenv}{{{\bf f}_1}}
\newcommand{\eigenv}{{\bf f}}
\newcommand{\DBounded}{\bq}

\newcommand{\lnorm}{\left | \left |}
\newcommand{\rnorm}{\right| \right|}

\newcommand{\Glauber}{{\{X_t\}_{t\geq 0}}}
\newcommand{\UnIsing}{{\mathbb{M}_{\rm Ising}}}

\newcommand{\subcont}{{\cJ}}

\date{\today}

\begin{document}

\title{Spectral Independence Beyond Uniqueness \\ using  the topological method. }
\author{Charilaos Efthymiou$^*$}
\thanks{ 
University of Warwick,  Coventry,  CV4 7AL, UK.  Email: {\tt charilaos.efthymiou@warwick.ac.uk}\\ 
$^*$Research supported by EPSRC New Investigator Award, grant EP/V050842/1,  and 
Centre of Discrete Mathematics and Applications (DIMAP), University of Warwick, UK. \\ 
}

\address{Charilaos Efthymiou, {\tt charilaos.efthymiou@warwick.ac.uk}, University of Warwick, Coventry, CV4 7AL, UK.}

\maketitle

\thispagestyle{empty}

\noindent
%{\color{red} Remark: At the end, remove all the red comments+newpage commands!}\\
%{\color{red} Remark:  Deterministic Algorithms  Does not imply anything}\\

\begin{abstract}

We present novel  results for fast mixing of Glauber dynamics using the  newly introduced  and  
powerful {\em Spectral Independence}  method from [Anari, Liu, Oveis-Gharan: FOCS 2020].  In our 
results, the parameters of the Gibbs distribution are expressed in terms of the {\em spectral radius} of  
the {\em adjacency matrix} of $G$, or that of the  {\em Hashimoto non-backtracking} matrix.

The analysis relies on new  techniques that we introduce to  bound  the maximum eigenvalue  of the 
pairwise influence 
matrix   $\infmatrix^{\Lambda,\tau}_{G}$ for the two spin Gibbs distribution $\mu$. There is a common 
framework that underlies  these techniques which we call the {\em topological method}.   The idea is to 
systematically  exploit the well-known connections between $\infmatrix^{\Lambda,\tau}_{G}$  and the 
topological  construction called  {\em tree of self-avoiding walks}.

Our approach is novel and  gives new insights to the problem of establishing spectral independence
for Gibbs distributions.  More importantly, it allows us to derive new -improved- rapid mixing bounds 
for Glauber dynamics   on  distributions  such as the Hard-core model and the Ising model
for graphs that  the spectral radius  is smaller than the maximum degree.
\end{abstract}

\newpage
\setcounter{page}{1}

\section{Introduction}

The Markov  Chain Monte Carlo method  (MCMC) is a very simple, yet very powerful method for approximate
sampling from Gibbs distributions on combinatorial  structures. In the standard setting, we are given  a very simple  
to describe, ergodic   Markov chain  and we  need to analyse the speed of  convergence to the equilibrium distribution. 
The challenge  is to show that the chain {\em  mixes fast} when  the parameters  of the equilibrium distribution belong 
to  a certain region of values.

Here our focus is on combinatorial structures that are specified with respect to  an underlying graph $G$, e.g., 
independent sets. For us the graph $G$ is always undirected, connected and finite  and has maximum degree $\maxDeg$.

Recently, a new technique for analysing the speed of convergence of the Markov chain called 
Glauber  dynamics has been introduced in \cite{OptMCMCIS}. This technique is called  the 
{\em Spectral Independence}  method.  The authors in \cite{OptMCMCIS} use the Spectral 
Independence  method (SI) to  prove  a long  standing  conjecture about the mixing time of 
Glauber dynamics  for the so-called {\em Hard-core} model, improving on a series of result
such as \cite{VIS,EHSVY19}.  Since then, it is not an exaggeration 
to claim  that  SI  has revolutionised the study in the field.  Using this method  it has been 
possible to  get positive results for approximate sampling from 2-spin Gibbs distributions  
that  match the hardness ones, e.g.,   \cite{OptMCMCIS,VigodaSpectralInd,Sly10,SS14}.

In this work, our main focus is on  the so-called  {\em pairwise influence matrix}, denoted as 
$\infmatrix^{\Lambda,\tau}_{G}$.  This is  a central concept as the rapid  mixing bounds we derive 
using SI  rely on showing that the {\em maximum  eigenvalue} 
of  this matrix  is bounded. We propose {\em novel}  techniques that  allow us to derive 
more accurate estimations  on the maximum  eigenvalue of this matrix  than what we have been 
getting from the previous  approaches that are  proposed  in \cite{OptMCMCIS,VigodaSpectralInd}. 
In turn,  we get new rapid mixing results   for the Glauber dynamics on two spin Gibbs distribution 
such as the {\em Hard-core} model and the  {\em Ising} model.  Interestingly,  in our  results  the 
parameters of the Gibbs distribution do not depend on the maximum degree $\maxDeg$ of the 
underlying graph $G$.  They rather depend  on  the spectral  radius of the {\em adjacency 
matrix} of $G$,  denoted as $\simpleadj_G$.

For concreteness, consider the Glauber dynamics on the Hard-core model  for  $G$ whose adjacency matrix
has spectral radius $\spradius$ which is a bounded number.  Let $\lcritical(k)$ by  the critical value  for  the 
{\em Gibbs uniqueness}  of the Hard-core model  on the infinite  $k$-ary tree.  We prove that the Glauber dynamics  
mixes   in $O(n\log n)$   steps for any fugacity $0\leq \lambda<\lcritical(\spradius)$. 
For a comparison, recall  that the   max-degree-$\maxDeg$ bound  for the Hard-core model comes from 
\cite{OptMCMCIS,VigodaSpectralIndB} and requires  fugacity $0\leq \lambda<\lcritical(\maxDeg-1)$ in order
to  get  $O(n\log n)$  mixing time. This implies that  our approach gives better bounds when the  spectral 
radius $\spradius$ is  smaller than $\maxDeg-1$.   As a reference,  note that we alway have that  
$\spradius \leq \maxDeg$. On the other hand, the spectral radius  can get much smaller than the  
maximum degree, e.g., for  a {\em planar}  graph  we have that $\spradius \leq \sqrt{8\maxDeg-16}+2\sqrt{3}$,
see further  cases in  Section \ref{sec:Applications}.
We get results of similar flavour for the Ising model, too.

We also consider the case where the spectral radius is unbounded.
Even though we  improve on results in the literature,  our results are not as strong as the 
those  we get  for the bounded case and (most likely)  they admit  improvements. The case of
unbounded spectral radius is considered  in the Appendix, at the end of this paper.

The use of the spectral radius, or   matrix  norms of the adjacency matrix in order to derive rapid 
mixing bounds has also been studied in  \cite{Hayes06,DGJ09}. As opposed to our  approach  here,  that  relies 
on SI,  these results   rely on the {\em path coupling} technique \cite{PathCoupling}.  
Our improvements on these works  reflects  the fact that SI is much stronger 
than  path coupling.  The techniques considered in these two  paper do not seem to apply to 
the setting of analysis we have with SI. In that respect, our approach is orthogonal to the previous ones.

All the above   sed some more light on the well-known hardness transition shown  for approximate 
counting-sampling in  the Hard-core model (and the antiferromagnetic Ising).  Specifically, our results indicate that 
for  graphs  of maximum  degree $\maxDeg$ the hardness transition occurs at $\lcritical(\maxDeg-1)$ only if 
the spectral radius of the adjacency matrix satisfies that $\spradius(\simpleadj_G)\geq \maxDeg-1$.  
On the other hand, for $\spradius(\simpleadj_G) < \maxDeg-1$ we show  rapid mixing for 
Glauber dynamics for the parameters of the Gibbs distribution which  are beyond the  uniqueness region, 
i.e., for $\lambda>\lcritical(\maxDeg-1)$.  Interestingly,  we manage to get access to this region of parameters  
by directly  analysing the influence matrix.

When we deviate from the standard approach, i.e., the one that relies on the maximum degree $\maxDeg$, 
a natural  challenge   that arises is how to deal with the {\em effect of high-degrees}. That is,  how to 
accommodate  in the analysis  the high degree vertices. This  problem is common when the underlying 
structure is a random (hypergraph)  graph \cite{SIGnp,Efth19,Efth22,ConnectiveConst}.   As far as  SI is 
concerned, the work in \cite{SIGnp} deals with a version of  this  problem,  however,  it focuses  only on 
{\em random graphs}  and it takes  advantage of properties that  are special to the typical instances of 
this distribution.  Unfortunately, there does not seem to be a way of exploiting them in the {\em worst-case} 
setting  that we are focusing here. 

Concluding,  one might wonder what are the optimal bounds one can get using the assumption on  the spectral
radius of $\simpleadj_G$, i.e., rather than the maximum degree $\maxDeg$. Note that the regions  of the parameters 
of the Gibbs distribution we get here are 
quite natural and  they are related to the spectral radius  in the same way as the  corresponding regions  for 
Gibbs uniqueness are related to  $\maxDeg-1$.  In the Gibbs uniqueness, the parameters of the distribution need to guarantee 
that the rate  of correlation decay counterbalances  the growth rate of the  underlying graph $G$, which is at most
$\maxDeg-1$,   hence their dependence on  $\maxDeg$.   In an  analogy to uniqueness,  here  the correlation  decay 
counterbalances   the growth rate  of the number walks  between any two  vertices in $G$, which is 
$\spradius(\simpleadj_G)$. 
 Hence, the entries of the influence matrix do not grow too large for distant pairs of vertices. 
In that respect, our {\em conjecture} would be that under the assumption of $\spradius(\simpleadj_G)$ for $G$, 
our results cannot be further improved.

However, we believe that improvements are possible if we consider different matrices, i.e., other than $\simpleadj_G$.
Specifically, we believe that a more natural  matrix to consider for the problem is the so-called {\em (Hashimoto) non-backtracking 
matrix} of $G$ introduced in \cite{Hash89}. The non-backtracking matrix is  an object that has been studied extensively 
in mathematical physics.  Getting results  in terms of the spectrum  of the non-backtracking matrix is quite
desirable because  it is considered to capture  the structure of $G$ much better than that of the adjacency matrix.

Here, we further derive  a rapid mixing result for the Ising model using  the non-backtracking 
matrix of $G$, e.g., see Theorem \ref{thrm:Ising4SPRadiusNBK}. This result is only a preliminary one. 
Due to the intricacy of working with this matrix, we didn't manage to show that  the  non-backtracking spectrum 
gives  better bound  than what we get with the adjacency one. We only show that the results that someone gets with the non-backtracking matrix 
are at least  as good  as  those   from the  adjacency matrix. 
 We believe that the direction  of exploiting the spectrum of the non-backtracking matrix for the problem  
is worth  further investigation.

\subsubsection*{The Topological Method for Spectral Independence} 
As mentioned earlier,  a central  notion in SI is the so-called pairwise influence matrix  $\infmatrix^{\Lambda,\tau}_{G}$.
Given a set of parameters of the Gibbs distribution,  the endeavour is to show that the maximum  eigenvalue  of 
$\infmatrix^{\Lambda,\tau}_{G}$ is $O(1)$.  
Previously introduced approaches in \cite{OptMCMCIS,VigodaSpectralInd} have been focusing on proving that 
either $|| \infmatrix^{\Lambda,\tau}_{G} ||_1$,  or    $|| \infmatrix^{\Lambda,\tau}_{G} ||_{\infty}$  is bounded, which 
in turn implies that the maximum eigenvalue is bounded. 
These approaches are quite elegant and provide,  in a very natural way, bounds that depend on the maximum 
degree $\maxDeg$. However, for our purposes here  they    seem to be too crude.

The main contribution of our work amounts to introducing {\em novel techniques} to bound the
maximum eigenvalue of $\infmatrix^{\Lambda,\tau}_{G}$ which  (for most of the cases) turn out to 
be more precise than  the previous ones.  We introduce a  common framework that underlies  these 
techniques that we call the  {\em topological method}.

The basic idea for the topological method comes from the following, well-known, observation: each  entry   
$\infmatrix^{\Lambda,\tau}_{G}(w,v)$ can be  expressed  in terms of a topological  construction  called 
{\em tree of self-avoiding walks}  (starting from $w$),  together with a set of weights  on  the paths of this
tree, which are called {\em influences}. The influences are specified by the parameters  of the Gibbs distribution
we consider. The entry of the matrix is nothing more than the sum of influences over an appropriately chosen
set of paths in this tree

In the topological method, we generalise the above concepts by  introducing  the notion of {\em walk-matrix}.
For the entries in a walk-matrix,   we don't necessarily use  trees of self-avoiding walks. We may  use other
topological constructions of $G$ such as {\em  path-trees},  {\em universal covers} etc  (see 
more about these constructions in the excellent textbook  \cite{GodsilBookComb}).  Also, we have weights on
the paths of this construction. The weight of  each path  can  be chosen arbitrarily, i.e., it is
not necessarily an influence. Each entry in the walk-matrix is a sum of weights of appropriately chosen paths
in the topological construction.  In that respect, one might regard the influence matrix $\infmatrix^{\Lambda,\tau}_{G}$ 
to be a special case of  walk-matrix (e.g. see Lemma \ref{lemma:InfluenceMatrixIsWalkMatrix}).

Exploiting  properties of these special matrices,   i.e., walk-matrices, we introduce two  techniques, orthogonal 
to each other,   that we use  to derive  our results.

With the first technique we focus on   {\em comparing}  walk-matrices  in  terms of their 
corresponding spectral radii.  That is,  for two walk-matrices $\bC$ and $\bD$    the aim is to establish  
that the corresponding spectral radii satisfy  $\spradius(\bC)\leq \spradius(\bD)$.   Since  $\infmatrix^{\Lambda,\tau}_{G}$ is a
special case of walk-matrix, we use this technique to establish an inequality which is  similar to 
the aforementioned one,  using the walk-matrix $\sum^n_{\ell=0}\xi \cdot \simpleadj^{\ell}_G$,  where 
$\simpleadj_G$ is the adjacency matrix and $\xi>0$ is an appropriately chosen real number. Specifically, we show that 
\begin{align}\nonumber
\spradius\left(\infmatrix^{\Lambda,\tau}_{G} \right) &\textstyle \leq \spradius\left( \sum^n_{\ell=0}\xi \cdot \simpleadj^{\ell}_G \right).
\end{align}
The value of  $\xi$ depends on the parameters of the Gibbs distribution that underlies 
$\infmatrix^{\Lambda,\tau}_{G}$. It is easy to see that adjusting these parameters  
such that $\xi \leq (1-\epsilon)/\spradius(\simpleadj_G)$, for any fixed $\epsilon>0$,  implies  that 
the spectral radius of  $\infmatrix^{\Lambda,\tau}_{G}$  -and hence the maximum eigenvalue-
is bounded. Working as described above, allows us to derive  new, very interesting results about the Ising 
model,  see  Theorem \ref{thrm:Ising4SPRadius}.

We establish another inequality, similar to the above,  that is between  $\infmatrix^{\Lambda,\tau}_{G}$
 and a walk-matrix  related to   the non-backtracking matrix, e.g., see Corollary \ref{cor:MaxElement4NB}.
We exploit this  inequality  to get the rapid mixing result in Theorem \ref{thrm:Ising4SPRadiusNBK},
  which is about the Ising model, too.

It turns out that for the Hard-core model the bounds we get from the previous approach are too crude. 
To this end, we follow a different one. We introduce a new matrix norm to  bound the maximum 
eigenvalue of  $\infmatrix^{\Lambda,\tau}_{G}$. Our norm provides    better bounds 
%on the  the spectral radius of $\infmatrix^{\Lambda,\tau}_{G}$  
compared to   what we get from  $||\infmatrix^{\Lambda,\tau}_{G} ||_1$ and  $|| \infmatrix^{\Lambda,\tau}_{G} ||_{\infty}$. Note that for the 
Hard-core model the problem becomes highly non-linear because of the fact that we     use the 
{\em potential method}.  In that respect,  finding a  matrix norm that allows us to derive the kind of bounds we are 
aiming for is a  non-trivial task.

We derive our rapid mixing results  for the Hard-core model  in Theorem \ref{thrm:HC4SPRadius} by 
choosing an appropriate  non-singular matrix $\bD$ and showing that  
$$
\textstyle \lnorm  (\bD)^{-1} \cdot \infmatrix^{\Lambda,\tau}_{G}  \cdot \bD\rnorm_{\infty}=O(1).
$$

\noindent 
More specifically,   $\bD$ is the diagonal matrix such that $\bD(w,w)=\left( \mathbold{f}_1(w)\right)^{t}$, for 
appropriate  $t\geq 1$, while   $\mathbold{f}_1$ is the  eigenvector that  corresponds to the maximum 
eigenvalue  of $\simpleadj_G$. Note that our assumptions about $G$  imply 
%that $\simpleadj_G$ is irreduciblewhich in turn guarantees that 
the above norm  is well-defined,  i.e., $\bD$ is non-singular.  The parameter $t\geq 1$,
in the norm, needs to be specified in the context of the   potential method.

In hindsight,   the use of the above norm is, somehow,   natural  in the context of topological method. 
The related analysis  gives rise to a further topological  construction that  we call {\em walk-vector} 
(see Section \ref{sec:WalkVector}).

\section{Main Results}\label{sec:Results}
Consider the fixed graph $G=(V,E)$  on $n$ vertices. We assume that $G$ is undirected,  finite and
connected.

The Gibbs distribution $\mu$ on $G$  with   spins  $\{\pm 1\}$ is a distribution on the set 
of configurations $\{\pm 1\}^V$. We use  the parameters $\beta\in \mathbb{R}_{\geq 0}$ and 
$\gamma, \lambda \in \mathbb{R}_{>0}$ and specify  that  each  configuration $\sigma\in \{\pm 1\}^V$ gets a probability 
measure
\begin{align}\label{def:GibbDistr}
\mu(\sigma)  &\propto \textstyle  \lambda^{\# \textrm{assignments ``1" in $\sigma$ } }
\times \beta^{\# \textrm{edges with both ends ``1" in $\sigma$ } }
\times \gamma^{\# \textrm{edges with both ends ``-1" in $\sigma$} },
\end{align}
where the symbol $\propto$ stands for ``proportional to".  The above distribution is called
{\em ferromagnetic}   when $\beta \gamma>1$,  while for $\beta \gamma<1$ it is   called 
{\em antiferromagnetic}.  Unless otherwise specified, we always assume that $\mu$ is a  
two-spin Gibbs distribution.

Using the formalism in  \eqref{def:GibbDistr}, one  recovers the well-known {\em Ising model} by setting 
$\beta=\gamma$. In this case, the magnitude  of $\beta$ specifies  the strength of the interactions. The 
above, also,  gives rise to the so-called {\em Hard-core model} if we choose $\beta=0$ and $\gamma=1$. 
Particularly, this distribution assigns to each {\em independent set}  $\sigma$ probability measure  which 
is proportional to $\lambda^{|\sigma|}$, where $|\sigma|$ is the size of the independent set.   For the
Hard-core model  we use the term  {\em fugacity} to refer to the parameter $\lambda$.

Given a Gibbs distribution $\mu$   we use the discrete time,   (single site)   {\em Glauber dynamics} 
$\Glauber$ to approximately sample from  $\mu$.  Glauber dynamics is a very simple to describe 
Markov chain.  The state space of the chain is the support of $\mu$.
We assume that the chain   starts from an arbitrary configuration $X_0$. For  
$t\geq 0$, the transition from the state $X_t$ to $X_{t+1}$ is according to the  following rules: 
%(a)
Choose uniformly at random a vertex $v$. 
%(b)
For every vertex $w$ different than $v$, set $X_{t+1}(w)=X_t(w)$.
%(c)
Then, set $X_{t+1}(v)$ according to the marginal of $\mu$ at $v$, conditional on 
the neighbours  of $v$ having the configuration  specified by $X_{t+1}$.

For the cases we consider here,  $\Glauber$ satisfies a set  of technical  conditions that come with the 
name {\em ergodicity}.  Ergodicity implies that $\Glauber$ converges to a {\em unique} {stationary 
distribution},  which, in our case,  is the Gibbs distribution $\mu$.

In this work, we  study the  rate that the  chain converges to 
stationarity,  when the parameters $\beta,\gamma$ and $\lambda$ of the Gibbs distribution vary in a 
range of parameters that depends on the spectral radius of the adjacency matrix $\simpleadj_G$. Recall 
that $\simpleadj_G$ is a $V\times V$ matrix with entries in $\{0,1\}$, such that for any $u,w\in V$ we have
\begin{align}\nonumber
\simpleadj_G(w, u)= {\bf  1} \{\textrm{ $u, w$ are adjacent in $G$}\}.
\end{align}
The spectral radius of $\simpleadj_G$, denoted as  $\spradius(\simpleadj_G)$ is equal to the maximum in 
magnitude  eigenvalue of  the matrix $\simpleadj_G$.

\subsection{The Ising Model}
As mentioned earlier,  the Ising model  corresponds to the distribution in  \eqref{def:GibbDistr}  such 
that  $\beta=\gamma$.  This implies that each configuration $\sigma\in  \{\pm 1\}^V$ 
is assigned  probability measure
\begin{align}\label{eq:DefOfIsing}
\mu(\sigma)\propto \lambda^{\sum_{x\in  V}{\bf 1}\{\sigma(x)=1\}}\times\beta^{\sum_{\{x, y\}\in E} {\bf 1}\{\sigma(x)=\sigma(y)\}}.
\end{align}
When $\beta>1$ we have the {\em ferromagnetic} Ising model, while when  $\beta<1$ we have
the {\em antiferromagnetic}. Here, we always assume that $\lambda =1$. This corresponds to 
what we call  {\em zero external field} Ising model.

It is a well-known result that the uniqueness region of the Ising model  on the infinite $k$-ary tree, where $k\geq 2$,  
corresponds to having 
\begin{align}\nonumber
\frac{k-1}{k+1} <\beta< \frac{k+1}{k-1}.
\end{align}
Particularly, for the ferromagnetic case the measure is unique when  $1\leq \beta < \frac{k+1}{k-1}$, while
for  the antiferromagnetic case, the measure is unique when  $\frac{k-1}{k+1}<\beta\leq 1$.

For $k\geq 2$ and $\delta\in (0,1)$, we let the interval
\begin{align}\label{eq:DefOfIsingUniquReg}
\UnIsing(k,\delta) =\textstyle \left [ \frac{k-1+\delta}{k+1-\delta},   \frac{k+1-\delta}{k-1+\delta}\right].
\end{align}
We use the Spectral Independence method to get the following result about the Ising model. 

\begin{theorem}[Adjacency Matrix]\label{thrm:Ising4SPRadius}
For any $\delta\in (0,1)$ and for bounded $\spradius_G \geq  2$,     consider  the graph  $G=(V,E)$  whose 
adjacency matrix $\simpleadj_G$ has  spectral radius $\spradius_{G}$. Also,   let $\mu_G$  be 
the Ising model on $G$ with zero external field and parameter $\beta\in \UnIsing(\spradius_G,\delta)$.  

There is a constant  $C=C( \delta)$ such that the mixing time of the  Glauber dynamics 
on $\mu_G$  is at most $C n\log n$.
\end{theorem}
%The proof of Theorem \ref{thrm:Ising4SPRadius} appears in Section \ref{sec:thrm:Ising4SPRadius}.  

Theorem \ref{thrm:Ising4SPRadius} follows from a technical result we present later, 
i.e., Theorem \ref{thrm:Recurrence4InfluenceEigenBound},  which we use to establish  spectral independence 
for the zero external field Ising model $\mu_G$ with $\beta\in \UnIsing(\spradius_G,\delta)$. 
Once we establish spectral independence for $\mu_G$,  we use results from \cite{VigodaSpectralIndB} to derive
the bound on the mixing time.

Theorem \ref{thrm:Ising4SPRadius}   improves on results in \cite{Hayes06} for the Ising model 
by allowing a wider rage for $\beta$.
Specifically,  in the ferromagnetic  case  the above bound allows for $\beta$ that is a constant factor larger  
than what we had  before.  The analogous holds for the antiferromagnetic Ising, i.e., Theorem \ref{thrm:Ising4SPRadius}  
allows for $\beta$ which is a constant factor smaller than what we had before. 

\subsection{The Hard-Core Model} 
Another distributions of interest is the so-called  {\em Hard-core} model.  The formalism  
in \eqref{def:GibbDistr}  gives rise to the Hard-core model with fugacity $\lambda$ if we set   $\beta=0$ 
and $\gamma=1$.  This distribution assigns to each independent set $\sigma$ of the graph $G$, 
probability  measure $\mu(\sigma)$ such that
\begin{align}
\mu(\sigma)\propto \lambda^{|\sigma|}, 
\end{align}
where   $|\sigma|$ is the size of the independent set.

We use $\{\pm 1\}^V$ to encode the configurations of the the Hard-core model, i.e., the independent sets of $G$. 
Particularly, the assignment $+1$  implies that the vertex is in the independent set, while $-1$ implies the opposite. 
We often use physics' terminology where  the vertices with assignment $+1$ are  called  ``occupied", whereas  the 
vertices with $-1$ are the ``unoccupied" ones.

For  $z>1$, we let the function $\lcritical(z)=\frac{z^{z}}{(z-1)^{(z+1)}}$.  It is a well-known result from \cite{Kelly85} that 
the uniqueness region of the Hard-core model  on the $k$-ary tree, where $k\geq 2$,   holds for any  $\lambda$ such that
\begin{align}
 \lambda< \lcritical(k).
\end{align}

As far as the Hard-core model is concerned we use the Spectral Independence method to derive the following result.

\begin{theorem}[Adjacency matrix]\label{thrm:HC4SPRadius}
For any $\epsilon \in (0,1)$,  $\maxDeg\geq 2$ and  $\spradius_G>1$ which is bounded,  consider the graph $G=(V,E)$  
of maxim degree $\maxDeg$,  whose adjacency matrix $\simpleadj_G$ has  spectral radius $\spradius_{G}$. Also,   let $\mu_G$ be 
the Hard-core model on $G$ with  fugacity $\lambda\leq (1-\epsilon)\lcritical(\spradius_G)$.

There is a constants  $C=C(\epsilon)$  such that the mixing time of the   Glauber dynamics on $\mu_G$ 
is at most $C n\log n$.
\end{theorem}

%The proof of Theorem \ref{thrm:HC4SPRadius} appears in Section \ref{sec:thrm:HC4SPRadius}. 

%If $\spradius_G$ is  unbounded, note that   the above result implies a polynomial bound for the mixing time for the case
%only that the ratio $\sqrt{\maxDeg/\spradius_G}$ is bounded. 
%
Theorem \ref{thrm:HC4SPRadius}   follows from a technical result, i.e., Theorem 
\ref{thrm:Recurrence4InfluenceEigenBoundNonLinear},  which we use to establish  spectral  
independence for the Hard-core  model $\mu_G$ with fugacity $\lambda\leq (1-\epsilon)\lcritical(\spradius_G)$.
Once we establish spectral independence we use results from \cite{VigodaSpectralIndB} to derive
the bound on the mixing time.

The above improves on results in \cite{Hayes06} for the Hard-core model by allowing a wider rage for $\lambda$. 
The upper bound on $\lambda$ here is by a constant factor larger than the previous one. Particularly, for large
$\spradius_G$ this constant converges to $e$, i.e., the base of natural logarithms.

\subsubsection*{Notation} For the graph $G=(V,E)$ and the Gibbs distribution $\mu$ on the set of configurations
$\{\pm 1\}^V$.  For a configuration $\sigma$,  we let $\sigma(\Lambda)$ denote the configuration that $\sigma$
specifies on the set of vertices $\Lambda$.
We let   $\mu_{\Lambda}$  denote the marginal of $\mu$ at the  set $\Lambda$.
We let $\mu(\cdot \ |\ \Lambda, \sigma)$, denote the distribution 
$\mu$ conditional on the configuration at $\Lambda$ being $\sigma$. Also, we interpret the conditional
marginal $\mu_{\Lambda}(\cdot \ |\ \Lambda', \sigma)$, for $\Lambda'\subseteq V$, in the natural way.

\subsection{Applications  I}\label{sec:Applications}
There are a lot of interesting cases of graphs whose adjacency matrix has spectral radius much smaller 
than the maximum degree, and hence, our results give better rapid mixing bounds than the general one. 
 A standard example is the planar graphs for which we have the following result from \cite{SPRadiusPlannar}.

\begin{theorem}[\cite{SPRadiusPlannar}]\label{thrm:SPRadiusPlanar}
Suppose that $G=(V,E)$ is a planar  graph of maximum degree $\maxDeg$, then 
$\spradius(\simpleadj_G)\leq \spradius_{\rm planar}(\maxDeg)$
where 
\begin{align}\label{Def:SpRadPlanar}
\spradius_{\rm planar}(\maxDeg) &=  
\left \{ 
\begin{array}{lcl}
\maxDeg &\quad& \textrm{for $\maxDeg\leq 5$},\\
\sqrt{12\maxDeg-36} &\quad& \textrm{for $6\leq \maxDeg\leq 36$},\\
\sqrt{8(\maxDeg-2)}+2\sqrt{3}  &\quad& \textrm{for $37\leq \maxDeg$}.
\end{array}
\right. 
\end{align}
\end{theorem}

In what follows, we show the implications of  the above theorem  to the mixing time of Glauber dynamics for
the Ising model and the Hard-core model.
We focus on   results for graphs of bounded maximum degree. 
The results for the unbounded case are straightforward  so we omit their statement. 

As far as the Ising model on planar graphs is concerned,  we have the following result. 

\begin{corollary}[Planar Ising model]\label{cor:PlanarIsing}
For  $\delta\in (0,1)$, for fixed $\maxDeg\geq 2$, consider the planar graph $G=(V,E)$ of  maximum 
degree $\maxDeg$. Let  $\mu_G$  be the zero external field Ising model on $G$ with   parameter 
$\beta$ such that
\begin{align}\nonumber
\beta \in \UnIsing \left( \spradius_{\rm planar}(\maxDeg), \delta\right),
\end{align}
where $\spradius_{\rm planar}(\maxDeg)$ is defined in \eqref{Def:SpRadPlanar}.
There is a constant  $C=C( \delta)$ such  the Glauber dynamics on $\mu_G$ 
exhibits  mixing time which is at most $C n\log n$. 
\end{corollary}

As far as the Hard-core model on planar graphs is concerned,  we have the following result. 

\begin{corollary}[Planar Hard-core model]\label{cor:PlanarHC}
For  $\epsilon \in (0,1)$,  for fixed $\maxDeg\geq 2$, consider  the planar graph $G=(V,E)$   of maximum 
degree $\maxDeg$.  Let  $\mu_G$ be  the Hard-core model on $G$ with  fugacity $\lambda$ such that
\begin{align}\nonumber
 \lambda\leq (1-\epsilon)\lcritical( \spradius_{\rm planar}(\maxDeg) ),
\end{align}
where $\spradius_{\rm planar}(\maxDeg)$ is defined in \eqref{Def:SpRadPlanar}. 
There is a constant  $C=C(\epsilon)$ such the  Glauber dynamics on $\mu_G$ 
exhibits  mixing time which is at most $C n\log n$.
\end{corollary}

There are further examples  of graphs with spectral radius much smaller than the maximum degree. 
One very interesting  case, which  generalises  the aforementioned one, is the graphs that can
be embedded in a surface of small {\em Euler genus}.

\begin{theorem}[\cite{SPRadiusPlannar}]\label{thrm:GenusVsSPRadius}
Let the graph $G=(V,E)$ be of maximum degree $\maxDeg>0$. Suppose that $G$ can be embedded in a surface 
of Euler genus { $g\geq 0$}.  If $\maxDeg\geq d(g)+2$, then 
\begin{align}\nonumber 
\spradius(\simpleadj_G)\leq \sqrt{8(\maxDeg-d(g))}+d(g), 
\end{align}
where $d(g)$ is such that
\begin{align}
d(g) &=\left \{ 
\begin{array}{lcl}
10 & \quad & \textrm{if $g\leq 1$},\\
12 & \quad & \textrm{if $2\leq g\leq 3$}\\ 
\end{array}
\right.  &  \textrm{and} &&
d(g) &=\left \{ 
\begin{array}{lcl}
2g+6 & \quad & \textrm{if $4\leq g\leq 5$},\\
2g+4 & \quad & \textrm{if $6\geq g $}. 
\end{array}
\right.  \nonumber 
\end{align}
\end{theorem}

%Note from the above theorem that if, e.g., 

If, e.g.,  the Euler genus of $G$ is much smaller than $\maxDeg$, then, from the above theorem, it is immediate that 
 $\spradius(\simpleadj_G)\approx \sqrt{8\maxDeg}$. It is straightforward to   combine the above results with Theorems \ref{thrm:HC4SPRadius} and \ref{thrm:Ising4SPRadius} and get results 
analogous to what we have in Corollaries \ref{cor:PlanarIsing} and \ref{cor:PlanarHC}.  
 We omit the presentation of these results as their derivation is straightforward.

\subsection{Applications II - Beyond the adjacency matrix}

We further derive  rapid mixing results for the Ising model using the spectral radius of the 
non-backtracking  matrix of $G$, denoted as $\NBMatrix_G$.   What motivates the use of 
this  matrix instead of $\simpleadj_G$  is  that in many cases the 
spectral radius of $\NBMatrix_G$ is much smaller.   This,   potentially, 
could lead to  even better rapid mixing bounds.

The result we present here is only a preliminary one, while its statement is not simple.  
The purpose of this small section is, also,  to show that our techniques allow us to consider  matrices  
other than $\simpleadj_G$.

For the graph $G=(V,E)$,  let $\OrntEdges$ be the set of {\em oriented} edges obtained by doubling 
each  edge of $E$ into two directed edge, i.e.,  one edge for each direction. The non-backtracking 
matrix  $\NBMatrix_G$ is an $\OrntEdges \times \OrntEdges $ matrix with entries 
in $\{0,1\}$,  such that for any pair of  oriented edges $e=(u,w)$ and $f=(z, y)$ we have that
\begin{align}\nonumber
\NBMatrix_G(e, f) ={\bf 1}\{w=z \}\times {\bf 1}\{ u\neq y\}.
\end{align}
That is, $\NBMatrix_G(e, f)$ is equal to $1$, if $f$ follows the edge $e$ without creating a loop,  otherwise,  it is equal to zero, 
e.g. see an example in Figure \ref{fig:NBTExample}. 
As opposed to other  matrices we have seen here,    $\NBMatrix_G$ is index by  
{\em oriented edges}. Also,  note that $\NBMatrix_G$ is not normal. 
%This makes working with the non-backtracking matrix is technically more involved. 
In general,  our  understanding of  $\NBMatrix_G$  is not as good as that
of $\simpleadj_G$.

 \begin{figure}
 \centering
		\includegraphics[width=.35\textwidth]{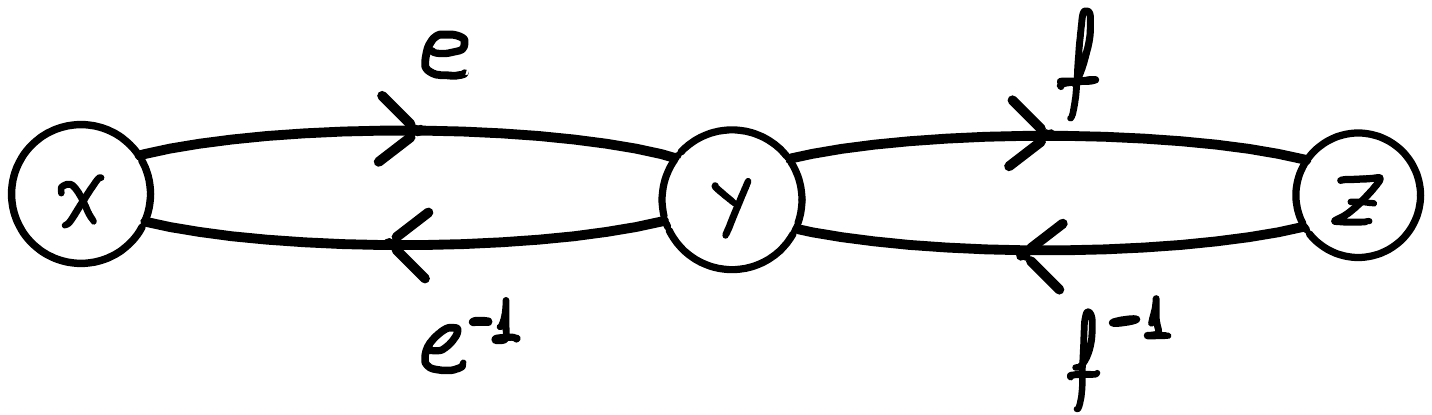}
		\caption{We have  $\NBMatrix_G(e, f)=1$, while $\NBMatrix_G(e^{-1}, f)=0$}
	\label{fig:NBTExample}
\end{figure}

In the following theorem, we use the Spectral Independence method to get a result 
for the Ising model where, rather than 
using $\simpleadj_G$,  we use   $\NBMatrix_G$.   Note that mixing time bound for Glauber dynamics
we get from the theorem below does not have a simple expression.

\begin{theorem}[Non-backtracking matrix]\label{thrm:Ising4SPRadiusNBK}
For any $\delta\in (0,1)$, for bounded $\maxDeg, \hspradius_G > 1$,  let $G=(V,E)$  
be  of maximum  degree  $\maxDeg$,  while assume that the non-backtracking matrix 
$\NBMatrix_G$ has spectral radius  $\hspradius_{G}$.
Furthermore,   let $\mu_G$  be the Ising model on $G$ 
with zero external field  and parameter  $\beta\in \UnIsing(\hspradius_G,\delta)$.  

There is a constant  $C=C( \delta)$ such that  the   Glauber dynamics on $\mu_G$ 
exhibits  mixing time  which is at most 
\begin{align}\nonumber 
\textstyle \exp\left(C \cdot \lnorm 
\left( \bI-\frac{1-\delta}{\hspradius_G}\simpleadj_G-(\frac{1-\delta}{\hspradius_G})^2(\bD-\bI) \right)^{-1}\rnorm_2 \right)n\log n
\end{align}
where $\simpleadj_G$ is the adjacency matrix of $G$ and $\bD$  is the $V\times V$ diagonal matrix such 
that for every $u\in V$ we have $\bD(u,u)={\tt degree}(u)$.
\end{theorem}

%The proof of Theorem \ref{thrm:Ising4SPRadiusNBK} appears in Section \ref{sec:thrm:Ising4SPRadiusNBK}

Note that the matrix in the 2-norm is the same as the one that  appears in Ihara's theorem \cite{Ihara}. 

We need to remark that for the bound on the  mixing time in Theorem \ref{thrm:Ising4SPRadiusNBK} 
we don't have guarantees that it is always  polynomial in $n$.  
However, we can argue that it is $O(n\log n)$, at least for the same values of $\beta$ that
Theorem \ref{thrm:Ising4SPRadius} implies $O(n\log n)$ bound for the mixing time. 
Of course, this  is not obvious at all from the  statement of the theorem. 
%
%It is not obvious at all from the  statements, but one can argue that for the same parameters $\beta$ of the Ising model 
%the bound we get  from Theorem \ref{thrm:Ising4SPRadiusNBK}  is {\em at most} that we get from  
%Theorem \ref{thrm:Ising4SPRadius}.  
For further discussion on this matter, the reader is referred  to the end  of Section \ref{sec:SPComparisonBasedOnEntries}.

\section{Our Approach - High level overview}\label{sec:HighOverview}

Consider the graph $G=(V,E)$ and a two-spin Gibbs distribution $\mu$ on this graph.
In the heart of Spectral Independence (SI)  lies  the notion of the {\em pairwise 
influence matrix} $\infmatrix^{\Lambda,\tau}_{G}$.   Let us give  the description  of this matrix  
since this is the main subject of our discussion here.

For a  set of vertices $\Lambda\subset V$ and   a configuration  $\tau$ at $\Lambda$,  we let the pairwise influence
matrix $\infmatrix^{\Lambda,\tau}_{G}$,  indexed by the vertices in $V\setminus \Lambda$,
be such that %,  for any $v,w\in V\setminus \Lambda$ we have
\begin{align}\label{def:InfluenceMatrixHighLevelA}
\infmatrix^{\Lambda,\tau}_{G}(w,u) &= \mu_{u }(1\ |\ (\Lambda, \tau),  (\{w\}, 1))- \mu_{u }(1\ |\ (\Lambda, \tau),  (\{w\}, -1))
& \forall v,w\in V\setminus \Lambda.
\end{align} 
The Gibbs marginal    $\mu_{u }(1 \ |\ (\Lambda, \tau),  (\{w\}, 1))$ indicates the probability    that  vertex $u$ gets  $1$, 
conditional on the  configuration at $\Lambda$ being $\tau$ and the   configuration  at $w$ being $1$.  
We have  the analogous for the marginal $\mu_{u }(1 \ |\ (\Lambda, \tau),  (\{w\}, -1))$.
Note that in some works,  the entry $\infmatrix^{\Lambda,\tau}_{G}(w,u)$ is denoted as $\infmatrix^{\Lambda,\tau}_{G}(w\rightarrow u)$.

Our focus  is on $\eigenval_{\rm max}(\infmatrix^{\Lambda,\tau}_{G})$,  i.e., the {\em maximum eigenvalue}  of   
$\infmatrix^{\Lambda,\tau}_{G}$.  If for {\em any} choice of  $\Lambda, \tau$ we have that  $\eigenval_{\rm max}(\infmatrix^{\Lambda,\tau}_{G})=O(1)$, 
then we say that  the Gibbs distribution $\mu$ exhibits {\em spectral independence}. Hence, the name of the method. 
Spectral independence for $\mu$ implies that   the corresponding Glauber dynamics has polynomial mixing time 
\cite{OptMCMCIS}. The precise magnitude  of the mixing time in this case is a subject of intense study. Here we don't focus
on this  problem, i.e., we focus only on establishing  spectral independence for $\mu$. 
%For related discussion,   see Section \ref{sec:SIPrelimIntro}. 

In \cite{OptMCMCIS, VigodaSpectralInd}  it was shown  that $\infmatrix^{\Lambda,\tau}_{G}$ 
 has a  remarkable  property. That is,  each  entry $\infmatrix^{\Lambda,\tau}_{G}(w,u)$ 
 can be  expressed in terms of a topological construction called tree of self-avoiding walks and a set influences on this tree. 
 The influences   are specified  by the parameters of our problem. 
At this point, it is worth giving a  high level (hence imprecise) description of the aforementioned relation.
For  for further details see  Section \ref{sec:TopoligcalMethod101}.

A walk is called self-avoiding if it does not repeat  vertices.  For each  vertex $w$ in $G$, 
we define  $\Tsaw(w)$,  the  tree of self-avoiding  walks, starting from $w$,  as follows:
Consider the set consisting of  every  walk  $v_0, \ldots, v_{r}$ in  the graph $G$ that emanates 
from vertex $w$, i.e., $v_0=w$,  while one of the following two holds
\begin{enumerate}
\item  $v_0, \ldots, v_{r}$ is a self-avoiding walk,
\item  $v_0, \ldots, v_{r-1}$ is a self-avoiding walk, while there is $j\leq r-3$  such that $v_{r}=v_{j}$.
\end{enumerate}
Each one of the walks in the set corresponds to a vertex in $\Tsaw(w)$.  Two vertices in $\Tsaw(w)$ are adjacent 
if the corresponding walks are adjacent.  Note that two  walks in the graph $G$ are considered to be  adjacent if one 
extends the other by one vertex \footnote{E.g.  the  walks    $P'=w_0, w_1, \ldots, w_{r}$ and 
$P=w_0, w_1, \ldots, w_{r}, w_{r+1}$ are adjacent with  each other.}.

We also use the following terminology:  for  vertex $u$ in $\Tsaw(w)$ that corresponds to the walk 
$v_0, \ldots, v_{r}$ in $G$ we say that   ``$u$ is a {\em copy} of vertex $v_{r}$ in $\Tsaw(w)$''.

In Figures \ref{fig:InitG4Tsaw} and \ref{fig:ResultingTsawA} we show an example of the above construction. 
Figure \ref{fig:InitG4Tsaw} shows the initial graph $G$, while Figure \ref{fig:ResultingTsawA} shows the 
tree of self-avoiding walks starting from vertex $A$.
Also, note that in   Figure \ref{fig:ResultingTsawA}, the vertices of the tree which are indicated with letter $A$
are exactly the copies of vertex $A$ in the initial graph. 

 \begin{figure}
 \begin{minipage}{.4\textwidth}
 \centering
		\includegraphics[width=.55\textwidth]{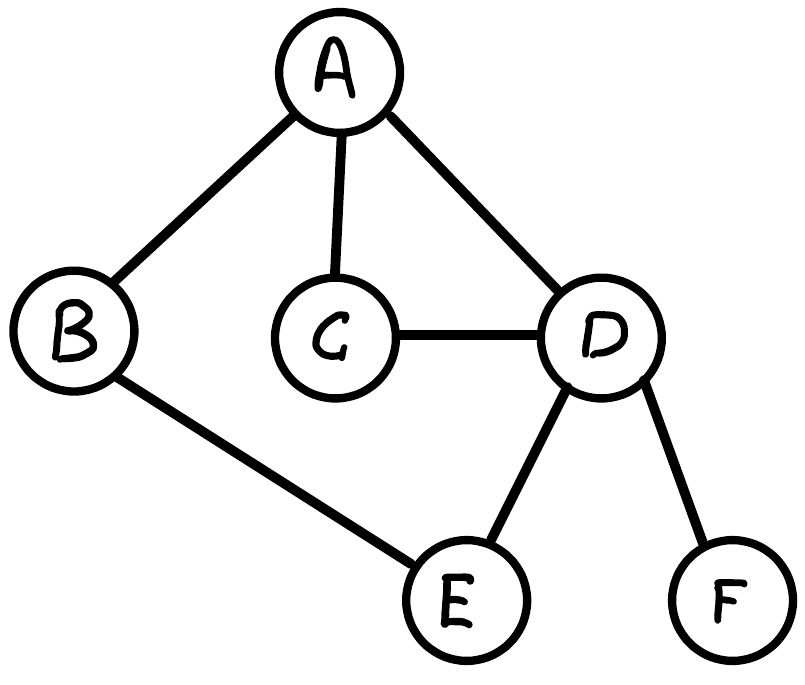}
		\caption{Initial graph $G$}
	\label{fig:InitG4Tsaw}
\end{minipage}
 \begin{minipage}{.56\textwidth}
 \centering
		\includegraphics[width=.35\textwidth]{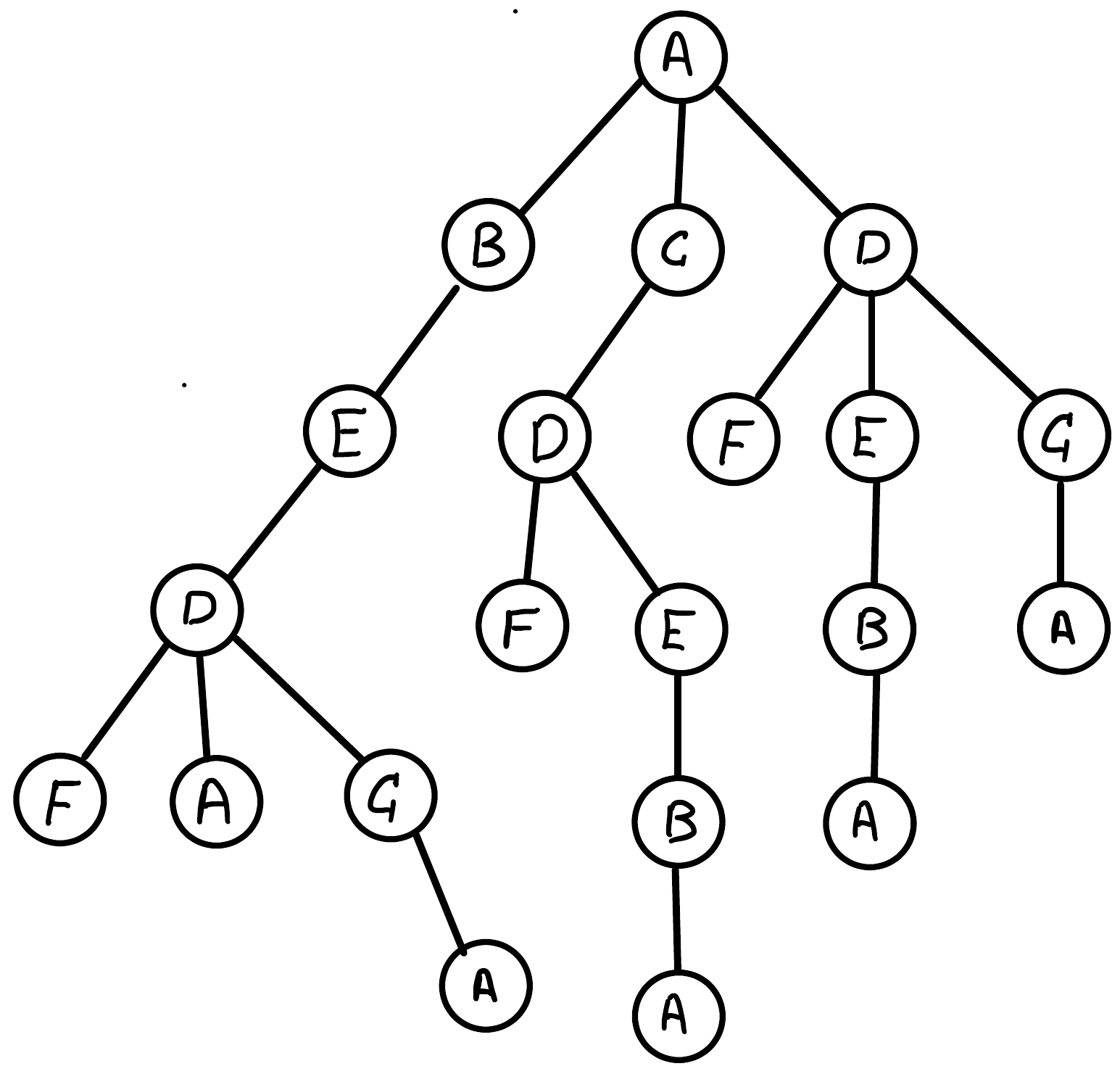}
		\caption{$T_{\rm SAW}$ of $G$  starting from $A$}
	\label{fig:ResultingTsawA}
\end{minipage}
\end{figure}

Each path in  $\Tsaw(w)$ is associated with a real number called  ``influence".
For a path of length $1$, which corresponds to an edge $e$ in the tree, we use
``{\rm Infl}(e)" to denote its influence. 
For a path $P$ with length $>1$, the influence is equal  to the following product, 
% of influences of the edges in this  path, i.e.,  
\begin{align}\nonumber
{\rm Infl}(P) &= \textstyle  \prod_{e\in P}{\rm Infl}(e).
\end{align}
That is, ${\rm Infl}(P) $ is equal to the product of influences of the edges in this  path.
\footnote{In the related literature,  influences are defined w.r.t. the vertices of the tree, not 
the edges.  In that respect, the influence of an edge $e=\{x,y\}$ here, corresponds to what is considered in other works 
as the influence at $y$, where $y$ is the child of vertex $x$ in the tree.}. 

Each entry $\infmatrix^{\Lambda,\tau}_{G}(w,u)$ can be expressed
in terms of a sum of  influences over paths in $\Tsaw(w)$, i.e., 
\begin{align}\nonumber %\label{eq:IntroEntryInflu}
\infmatrix^{\Lambda,\tau}_{G}(w,u) & =\textstyle \sum_{P} {\rm Infl}(P),
\end{align}
where $P$  in the summation varies over the paths from the root  to the copies  of vertex $u$ in $\Tsaw(w)$. 

The above construction is key to establishing spectral independence for $\mu$. This is for both
the techniques from previous works,  as well as the techniques we introduce here.

The approaches, prior to those we introduce here, establish spectral independence 
by  bounding  appropriately one of   $|| \infmatrix^{\Lambda,\tau}_{G} ||_1$  and  
$|| \infmatrix^{\Lambda,\tau}_{G} ||_{\infty}$.  Recall  that these two norms correspond 
to taking  the maximum absolute column/row sum  of the matrix, respectively.  The above 
construction,  with  $\Tsaw(w)$ and the influences, allows the  estimation of the aforementioned 
matrix norms using   recursion.

We also use the above construction to establish our results. 
We give a high level overview of  two alternative approaches  to establishing   spectral 
independence. For the sake of clarity,  in the presentation of the  results, as well as in the 
discussion that follows,  we explicitly  refer to bounding $\spradius(\infmatrix^{\Lambda,\tau}_{G})$,    
the  {\em spectral radius}  of the influence matrix, rather than the maximum eigenvalue. 
Note that bounding   the spectral radius is a stronger notion.

The first approach does not rely on matrix norms at all.  Actually, the argument  is quite simple, yet, 
powerful. It is useful to start  with   a concrete example. We   show that  there is  a real number   
$\xi>0$, which  depends on the parameters of the  Gibbs distribution $\mu$, such that  
\begin{align}\label{eq:ApprroachBasicInequalityIntro}
\spradius(\infmatrix^{\Lambda,\tau}_{G}) &\leq  \textstyle \spradius \left(\powadj_1  \right), & 
\textrm{where} && \powadj_1&= \sum^{n}_{\ell = 0} \left( \xi \cdot \simpleadj_G\right)^{\ell}.
\end{align}
The above  holds {\em without} any assumptions  about the underlying Gibbs distribution $\mu$, 
e.g., spatial mixing. 

We can directly use the the above inequality to establish spectral independence 
for the Gibbs distribution of  interest.  Particularly, we only  need to adjust the parameters of the 
Gibbs distribution such  that $\xi<(1-\epsilon)/\spradius(\simpleadj_G)$ for some fixed $\epsilon>0$. 
We rely on  the above inequality in order to get Theorem \ref{thrm:Ising4SPRadius}.

Having seen the relation between  $\infmatrix^{\Lambda,\tau}_{G}$ and the  tree of self-avoiding  walks,  
actually, it is not too hard to illustrate  (at least on a high level) how we derive the above inequality. Essentially, 
we  need to show that the matrix $\powadj_1$ from \eqref{eq:ApprroachBasicInequalityIntro} {\em dominates}   
$\infmatrix^{\Lambda,\tau}_{G}$ in the  following sense:  for any  $u,w\in V\setminus \Lambda$, it holds that
\begin{align}\label{eq:ApprroachBasicInequalityB}
\left|\infmatrix^{\Lambda,\tau}_{G}(w, u) \right| &\leq \textstyle  \powadj_1(w,u)  
\end{align}
Once we have  above, then it is  not too hard  to show that  
\eqref{eq:ApprroachBasicInequalityIntro} is true. % 

Note that the  matrices $\powadj_1, \infmatrix^{\Lambda,\tau}_{G}$ are not necessarily of the 
same dimension, i.e.,  $\powadj_1$ is  indexed by the set  $V$, while  
$\infmatrix^{\Lambda,\tau}_{G}$ is  indexed by  $V\setminus \Lambda$. For our purposes, we   
only need to focus   on the  vertices in  $w, u \in V\setminus \Lambda$.

We choose $\xi$, the parameter for $\powadj_1$, such that 
\begin{align}\nonumber 
\xi \geq \max_{e}\{ |{\rm Infl}(e)| \},
\end{align}
where $e$ varies over the edges in all self-avoiding  trees. For such $\xi$,    it is not hard to see that
\begin{align}\nonumber
\left|\infmatrix^{\Lambda,\tau}_{G}(w, u) \right| \leq \sum_{\ell \geq 0} \left( \xi \right)^{\ell} \times 
\left(\textrm {\# length $\ell$ paths from the root to a copy of $u$ in $\Tsaw(w)$} \right), 
\end{align}
where ``$\textrm{\#}$" stands for ``number of".   Then, we argue  that 
\begin{align}\label{eq:NoOfPATHSVsWALKS}
 \left(\textrm {\# length $\ell$ paths from root to a copy of $u$ in $\Tsaw(w)$} \right) \leq \left( \simpleadj_G \right)^{\ell}(w,u). 
\end{align}
Combining the two above inequalities  we immediately get   \eqref{eq:ApprroachBasicInequalityB} and
consequently \eqref{eq:ApprroachBasicInequalityIntro}.

Perhaps, it is worth elaborating a bit more as to why the previous inequality holds, since we are 
extending  it for our results with the non-backtracking matrix.  Recalling the definition of  $\Tsaw(w)$,
 let us call ${\rm SAW}$ the set  of paths $P$ in $G$ that we used earlier  to build $\Tsaw(w)$.
It is not hard to see that the  number of  length $\ell$ paths in ${\rm SAW}$  
from  $w$ and $u$ is equal to the l.h.s. of  the inequality above. 
Then, \eqref{eq:NoOfPATHSVsWALKS} follows from  the observation that {\em any} set of length
$\ell$ paths in $G$  from  $w$ to $u$ does not have cardinality larger than the  {\em total} number walks 
of length $\ell$ from  $w$ to $u$, which is equal to  $\left( \simpleadj_G \right)^{\ell}(w,u)$. 
Recall that a {walk} of length $\ell$ in the graph $G$ is any sequence of vertices $w_0, \ldots, w_{\ell}$ 
such that each consecutive  pair $(w_{i-1},  w_{i})$ is an edge in $G$.

Our result for the non-backtracking matrix builds on the above by exploiting a further observation:
It is not hard to see that every path  in ${\rm SAW}$ is also a non-backtracking walk.
Recall that the  walk $w_0, \ldots, w_{r}$  is called  non-backtracking, if we have that  $w_{i-1}\neq w_i$ for all $i$.
Then, one might argue that in \eqref{eq:NoOfPATHSVsWALKS} we could instead have used for the upper bound 
the number of length $\ell$ non-backtracking walks from $w$ to $u$.  For a further discussion,  see  Section
 \ref{sec:SPComparisonBasedOnEntries}.

All the above  are useful  to prove our results for the Ising model.  For the Hard-core model, we need to work 
differently, i.e.,  we need a more involved analysis. We typically study  the  Hard-core model, but also  general 
two spin Gibbs distributions, by means of  the  so-called {\em potential method}. In that respect, it is somehow, 
natural to return back to using matrix norms. 

We establish spectral independence for the Hard-core model by using   the following matrix norm for $\infmatrix^{\Lambda,\tau}_{G}$: 
$$
\textstyle \lnorm  \bD^{-1} \cdot \infmatrix^{\Lambda,\tau}_{G}  \cdot \bD\rnorm_{\infty}.
$$
  $\bD$ is the diagonal matrix such that $\bD(w,w)=\left( \mathbold{f}_1(w)\right)^{1/t}$, for appropriate 
$t\geq 1$, while   $\mathbold{f}_1$ is the  eigenvector that  corresponds to $\eigenval_{\rm max}(\simpleadj_G)$,  
the maximum  eigenvalue  of the adjacency matrix $\simpleadj_G$.  

Since we have assumed that the graph $G$ is connected, it is not hard to verify that the above norm is well-defined. 
The Perron-Frobenius Theorem implies that   $\mathbold{f}_1$ is positive, i.e., for every $w\in V$ we have 
$\mathbold{f}_1(w)>0$.  Hence, $\bD$ is non-singular as all the diagonal entries are positive.

%Recall that   {\em any} matrix norm  is at least as large as the spectral radius. In that respect, 
Showing that  the above norm is bounded for any choice of  $\Lambda$ and $\tau$, immediately  
implies  spectral independence  for $\mu$.

We give a high level overview of how we estimate the quantity 
$\scriptstyle \lnorm  \bD^{-1} \cdot \infmatrix^{\Lambda,\tau}_{G}  \cdot \bD\rnorm_{\infty}$.
%
%Using \eqref{eq:IntroEntryInflu}, 
Note that the entries of the matrix  $\bD^{-1}\cdot \infmatrix^{\Lambda,\tau}_{G}\cdot \bD$ satisfy  that
\begin{align}\nonumber
\bD^{-1}\cdot \infmatrix^{\Lambda,\tau}_{G}\cdot \bD (w,u) &={\textstyle \left( \frac{\mathbold{f}_1(u)}{\mathbold{f}_1(w)} \right)^{1/t} }\times \sum_{P} {\rm Infl}(P),
\end{align}
where $P$  in the summation varies over the paths from the root  to a copy of vertex $u$ in $\Tsaw(w)$. 

We  use the above and a recursion to estimate the absolute row-sum  
$\ \sum_{u} \left | \bD^{-1}\cdot \infmatrix^{\Lambda,\tau}_{G}\cdot \bD (w,u) \right |$.
In what follows, we show how to estimate  the contribution to the absolute row-sum above  coming  from paths of 
length $\ell\geq 1$ in  $\Tsaw(w)$.  Note that the contribution coming  from the path of length 0 
is, trivially, 1.

For every vertex $x$ in $\Tsaw(w)$ which is at  level $0<h< \ell$,   we estimate ${\cQ}_{x}$ 
the sum of absolute influences of the paths  from  $x$ to its decedents   at level $\ell$ of the tree. 
%
%We express ${\cQ}_{x}$ in terms of the sum of influences of its children. Specifically, 
Using the potential method we  get the following recursive relation: 
\begin{align}\nonumber
\left( {\cQ}_{x} \right)^t \leq \delta \cdot  \sum_{z} \left( \cQ_{z} \right)^t, 
\end{align} 
where  $z$ varies over the children of $x$ in $\Tsaw(w)$.  Similarly to $\cQ_x$, the quantity   $\cQ_{z}$
is the sum of absolute influences of the paths  from  $z$ to its decedents   at level $\ell$ of the tree.
The parameter  $t$ is   specified  in the setting of  the potential method. 
The quantity $\delta$ depends on the parameters of the  Gibbs distribution.  

%For  more intuition, let us remark that we choose the parameters of the Gibbs distribution so that 
%$\delta$ is of the form $\frac{1-\epsilon}{\spradius(\simpleadj_G)}$,
%for   $\epsilon>0$. 

Note that the above relation excludes the cases where $h=0$ and $h=\ell$ for vertex $x$. For $h=\ell$, i.e., $x$ is a vertex at 
level $\ell$ of $\Tsaw(w)$,   we have $\left( \cQ_{x}\right)^t= \mathbold{f}_1(x) $.
We leave the case  where $h=0$, i.e.,  $x$ is the root,   for the end of this discussion. 

 It is standard  to work with the above recurrence. Particularly, for vertex $x$ at level $0<h\leq \ell$  we get that
 \begin{align}\nonumber 
\left( {\cQ}_{x} \right)^t \leq \left ( \spradius(\simpleadj_G) \cdot \delta\right )^{\ell-h}  \mathbold{f}_1(x).
\end{align}
 In order to derive the above, we exploit that 
 for any vertex $z$ in the graph $G$ and for  $\Gamma$  the set of neighbours of $z$, we have 
\begin{align}\nonumber
\sum_{y \in \Gamma }\mathbold{f}_1(y) &=
\mathbold{f}_1(z) \cdot \eigenval_{\rm max}(\simpleadj_G) \ = \mathbold{f}_1(z)\cdot  \spradius(\simpleadj_G).
\end{align}
The first equality  follows from the definition of eigenvector  $\mathbold{f}_1$. The second equality follows from our  
assumption that   $G$ is connected,  i.e., the Perron-Frobenius Theorem
implies that $\eigenval_{\rm max}(\simpleadj_G)=\spradius(\simpleadj_G)$.

For the case where $h=0$, i.e., $x$ is  root of $\Tsaw(w)$,  we have the following: there is $c>0$ such that
\begin{align}\nonumber
\cQ_{\rm root}  &\leq c\cdot (\delta \cdot \spradius(\simpleadj_G))^{\ell/t }  
\textstyle \sum_{z\sim w}\left( \frac{\mathbold{f}_1(z)}{\mathbold{f}_1(w)} \right)^{1/t}.
\end{align}
One can simplify the above by noting that  the rightmost sum is at most 
$\maxDeg^{1-\frac{1}{t}} \left(\spradius(\simpleadj_G) \right)^{\frac{1}{t} }.$

Recall  that the quantity $\cQ_{\rm root}$ is the contribution of the paths of length $\ell$ into 
$\sum_{u} \left | \bD^{-1}\cdot \infmatrix^{\Lambda,\tau}_{G}\cdot \bD (w,u) \right |$. 
We bound the absolute row-sum  by   summing the contributions for $\ell=0, 1, \ldots, n$. 

Subsequently, we choose the parameters of the Gibbs distributions such that  
 $c=O(1)$ and $\delta=\frac{1-\epsilon}{\spradius(\simpleadj_G)}$.  
 For a graph with bounded maximum degree $\maxDeg$, this choice implies that  the absolute  row sum in 
 $\bD^{-1}\cdot \infmatrix^{\Lambda,\tau}_{G}\cdot \bD$ that corresponds to  the row of vertex $w$ 
  is bounded, for any $w$.  Hence 
 $\scriptstyle \lnorm \bD^{-1}\cdot \infmatrix^{\Lambda,\tau}_{G}\cdot \bD \rnorm_{\infty}$ is bounded, too.

\subsection{Structure of the rest of the paper.} 
 In Section \ref{sec:Preliminaries} we present some basic concepts from linear algebra, theory of Markov chains,
 spectral graph theory that we use for our results and the analysis.  Our results start appearing from Section \ref{sec:RecursionVsSpectralIneq}, 
where we present the basic set-up with Gibbs tree recursions and the basic theorems that establish 
spectral independence, i.e., Theorems \ref{thrm:Recurrence4InfluenceEigenBound}, \ref{thrm:Recur4InfSPBoundNBK}
and \ref{thrm:Recurrence4InfluenceEigenBoundNonLinear}. The proofs of these results are grouped together and
appear in Section  \ref{sec:Proofs4SIResults}. Section \ref{sec:TopoligcalMethod101} is an introduction to the basic concepts and
results  of what we call the topological method. The proofs of the results in Section \ref{sec:TopoligcalMethod101}
are, also,  grouped together and appear in Section \ref{sec:Proofs4Topol101}. In Section \ref{sec:SPComparisonBasedOnEntries}
we present results from the topological methods that we use to prove Theorems \ref{thrm:Recurrence4InfluenceEigenBound}
and \ref{thrm:Recur4InfSPBoundNBK}. The proof of any theorems in Section \ref{sec:SPComparisonBasedOnEntries}
appear in Section \ref{sec:Proofs4EntryBasedTech}. In Section \ref{sec:SPComparisonBasedOnNorms} we present
 results for the topological methods that we use to prove Theorem \ref{thrm:Recurrence4InfluenceEigenBoundNonLinear}.
The proofs of the results in Section \ref{sec:SPComparisonBasedOnNorms} are grouped together and appear in
Section \ref{sec:Proofs4NormBasedTech}.  Finally, the proofs of our main results, i.e., 
Theorems \ref{thrm:Ising4SPRadius} \ref{thrm:Ising4SPRadiusNBK} and \ref{thrm:HC4SPRadius} appear as follows:
The proofs of Theorems \ref{thrm:Ising4SPRadius} and 
\ref{thrm:Ising4SPRadiusNBK} which are about the mixing time of Glauber dynamics for the Ising model
appear in Section \ref{sec:RapidMixingIsing}. The proof of Theorem \ref{thrm:HC4SPRadius} for the Hard-core model
appears in Section \ref{sec:thrm:HC4SPRadius}. 

 At the end of this paper we have an Appendix with supplementary results and material. 

\section{Preliminaries}\label{sec:Preliminaries}

\subsection{Measuring the speed of convergence for Markov Chains.}\label{sec:MCMCIntroStaff}
For measuring the distance between two distribution we use the notion of {\em total variation distance}. 
For two distributions $\nu$ and $\hat{\nu}$ on the  discrete set $\Omega$,
the  total variation distance satisfies 
\begin{align}\nonumber %\label{def:TVD}
||\nu-\hat{\nu}||_{tv} =\frac{1}{2}\sum_{x\in \Omega}|\nu(x)-\hat{\nu}(x)|.
\end{align}
We use the notion of {\em mixing time} as a  measure  for the rate that an ergodic  Markov  chain converges to equilibrium.
More specifically,  let $P$ be the transition matrix of an ergodic Markov chain $\Glauber$  on a finite state space
$\Omega$ with stationary distribution $\mu$. For $t\geq 0$  and $\sigma\in \Omega$, we let  $P^t(\sigma, \cdot)$ 
be  the distribution of $\Glauber$ when  $X_0=\sigma$. Then, the   mixing time of  $P$ is defined as
\begin{align}\nonumber 
T_{\rm mix}(P)=\min\{t\geq 0\ :\ \forall \sigma\in \Omega \ \ || P^t(\sigma, \cdot)-\mu(\cdot)||_{tv}\leq 1/4\}.
\end{align}
The transition matrix $P$ with stationary distribution $\mu$ is called {\em time reversible} if it satisfies the so-called {\em detailed balance relation}, 
i.e., for any $x, y\in\Omega$ we have $\mu(x)P(x,y)=P(y,x)\mu(y)$. 
For $P$ that is time reversible the set of eigenvalues are real numbers and we denote them as 
$1=\eigenval_1\geq \eigenval_2\geq \ldots \eigenval_{|\Omega|}\geq -1$. 
Let $\eigenval^*=\max\{|\eigenval_2|, |\eigenval_{|\Omega|}|\}$, then from \cite{TmixSpGap} we have that
\begin{align}\label{eq:MivingTimeVsSPGap}
T_{\rm mix}(P) \leq \frac{1}{1-\eigenval^*}\log\left(\frac{4}{\min_{x\in\Omega}\mu(x)}\right).
\end{align}
The quantity $1-\eigenval^*$ is also known as the {\em spectral gap} of $P$.

\subsection{Spectral Independence}\label{sec:SIPrelimIntro}
Consider a graph $G=(V,E)$. Assume that we are given a Gibbs distribution $\mu$ on the 
configuration space $\{\pm 1\}^V$. In the heart of  SI  lies 
the notion of the {\em pairwise influence matrix} $\infmatrix^{\Lambda,\tau}_{G}$.   
Due to its importance, let use write the  definition for a second time. 

For a given  a set of vertices $\Lambda\subset V$ and a configuration $\tau$ at $\Lambda$, 
we have that $\infmatrix^{\Lambda,\tau}_{G}$ is a matrix indexed by the vertices in $V\setminus \Lambda$.
Particularly,  for any $v,w\in V\setminus \Lambda$, the 
entry $\infmatrix^{\Lambda,\tau}_{G}(w,u)$ satisfies that
%\footnote{For the entry of 
%$\infmatrix^{\Lambda,\tau}_{G}$  in some works the notation   $\infmatrix^{\Lambda,\tau}_{G}(w\rightarrow u)$  is used 
%instead of  $\infmatrix^{\Lambda,\tau}_{G}(w,u)$. }  
%
%
\begin{align}\label{def:InfluenceMatrix}
\infmatrix^{\Lambda,\tau}_{G}(w,u)= \mu_{u }(1\ |\ (\Lambda, \tau),  (\{w\}, 1))- \mu_{u }(1\ |\ (\Lambda, \tau),  (\{w\}, -1)), 
\end{align} 
where  $\mu_{u }(1 \ |\ (\Lambda, \tau),  (\{w\}, 1))$ is the Gibbs marginal  that  vertex $u$ gets  $1$, 
conditional that the configuration at $\Lambda$ is $\tau$ and  the configuration at  $w$ is $1$.  
We have  the analogous  for $\mu_{u }(1 \ |\ (\Lambda, \tau),  (\{w\}, -1))$.

As far as the influence matrix is concerned, the main focus is on  $\eigenval_{\rm max}(\infmatrix^{\Lambda,\tau}_{G})$
the maximum eigenvalue. When, this is bounded for any choice of $\Lambda\subset V$ 
and configuration $\tau\in \{\pm 1\}^{\Lambda}$, then we say that 
the underlying Gibbs distribution $\mu$ is spectral independent. Let us be more formal. 

\begin{definition}[Spectral Independence]\label{Def:SpInMu}
For a real $\eta>0$,  the Gibbs distribution $\mu_G$ on  $G=(V,E)$ is  $\eta$-spectrally
independent, if for every $0\leq k\leq |V|-2$, $\Lambda\subseteq V$ of size $k$ and $\tau\in \{\pm 1\}^\Lambda$
we have that  $\eigenval_{\rm max}(\cI^{\Lambda,\tau}_{G})\leq 1+\eta$.  
\end{definition}

\noindent
Establishing spectral independence for $\mu$ implies that the corresponding  Glauber dynamics
mixes in polynomial time, i.e., polynomial it the size of the graph $n$.   

\begin{theorem}[\cite{OptMCMCIS}]\label{thrm:SPCT-IND} 
For $\eta>0$,  if $\mu$ is an $\eta$-spectrally  independent distribution,  then  Glauber dynamics 
for sampling from $\mu$ has   {\em spectral gap} which is at least
\begin{align}
n^{-1}\prod^{n-2}_{i=0}\left( 1-\frac{\eta}{n-i-1}\right).
\end{align}
\end{theorem}

One  gets a close expression for the spectral gap we have in Theorem \ref{thrm:SPCT-IND} by working as in
\cite{VigodaSpectralInd}  (Theorem 5 in arxiv version) to get the following result.

\begin{theorem}[\cite{VigodaSpectralInd}]\label{thrm:SPCT-INDClosed}
For $\eta>0$, there is a constant $C\in [0,1]$  such that  if $\mu$ is an $\eta$-spectrally  independent distribution,  
then  Glauber dynamics  for sampling from $\mu$ has   {\em spectral gap} which is at least $Cn^{-(1+\eta)}$. 
\end{theorem}

Note that  Theorems \ref{thrm:SPCT-IND} and \ref{thrm:SPCT-INDClosed}  imply that with spectral independence 
the mixing time of Glauber dynamics is polynomial in $n$, i.e., see \eqref{eq:MivingTimeVsSPGap}.  However, this polynomial  can be very large. There has been improvements on 
Theorem \ref{thrm:SPCT-IND} since its introduction  in \cite{OptMCMCIS}, e.g., see  \cite{VigodaSpectralIndB,YitongAllDegreeOptMix}. 

Here we mainly focus on establishing spectral independence for the Gibbs distributions of interest. 
Once this is established,  we use results from other  works  to derive  bounds  on the mixing time. 
Specifically,  for the bounded $\spradius(\simpleadj_G)$  case, which also implies that the maximum
degree $\maxDeg$ is bounded,  we use  Theorem 1.9   from  \cite{VigodaSpectralIndB} (arxiv version). 
For the unbounded cases, see the discussion in Section \ref{sec:UnboundedRadius},   in the Appendix.

%For the unbounded case, we use Theorem \ref{thrm:SPCT-INDClosed} which, essentially, is a restatement
%of Theorem 5 from \cite{VigodaSpectralInd}.
%%
%Potentially, someone could use the results from \cite{YitongAllDegreeOptMix} for the unbounded $\spradius(\simpleadj_G)$
%case. However, these results seem to rely heavily on the (somehow strong) assumption that the underlying Gibbs distribution exhibits
% tree uniques. This is something we cannot assume for our distribution in our setting. In that respect we cannot use these results here. 

\newcommand{\spectrum}{{\tt spect}}

\subsection{Linear algebra}
\newcommand{\bphi}{{\mathbold{\phi}}}

For a square $N\times N$ matrix $\bM$,  we let $\eigenval_i(\bM)$, for $i\in[N]$ denote the eigenvalues of $\bM$ such that
$\eigenval_1(\bM)\geq \eigenval_2(\bM) \geq \ldots \geq \eigenval_N(\bM)$.
Also, we let $\spectrum(\bM)$ denote  the set of  distinct  eigenvalues of  $\bM$.  We also refer to $\spectrum(\bM)$ 
as the spectrum of $\bM$.

We define the {\em spectral radius} of $\bM$, denoted as $\spradius(\bM)$,  to be the real number
such that 
\begin{align}\nonumber %\label{def:SpectralRadius}
\spradius(\bM)=\max_{}\{ |\eigenval| \ : \  \eigenval\in \spectrum(\bM) \}.
\end{align}
It is a well-known result  that the spectral radius of $\bM$ is the greatest lower bound for all 
of its matrix norms, e.g. see Theorem 6.5.9 in \cite{MatrixAnalysis}. Letting $\lnorm \cdot \rnorm$ be 
a matrix norm on $N\times N$ matrices,  we have that
\begin{align}\label{eq:SpRadiusVsMNorm}
\spradius(\bM) \leq \lnorm \bM \rnorm.
\end{align}
Perhaps, it is useful to mention that for the special case where $\bM$ is symmetric, i.e., $\bM(i,j)=\bM(j,i)$ for all $i,j\in [N]$,  
we have that  $\spradius(\bM) = \lnorm \bM \rnorm_2$.

For    $\bA, \bB, \bC \in \mathbb{R}^{ N \times N}$,  we let $|\bA|$ denote the matrix having entries $|\bA_{i,j}|$.
 For  the matrices $\bB, \bC$ we define $\bB\leq \bC$ to mean that $\bB_{i,j}\leq \bC_{i,j}$ for each $i$ and $j$.
%The reader should not confuse the notation $\bB\leq \bC$ to  $\bB \preccurlyeq \bC$. 
The following is a folklore result  in linear algebra (e.g. see \cite{SIAM-LAlg,MatrixAnalysis}).

\begin{lemma}\label{lemma:MonotoneVsSRad}
For  integer $N>0$, let $\bA, \bB\in \mathbb{R}^{N\times N}$.  If $|\bA|\leq \bB$, then 
$\spradius(\bA) \leq \spradius (|\bA|) \leq \spradius(\bB)$.
\end{lemma}

\subsection{Concepts from algebraic graph theory}
A {\em walk} in the graph $G$ is any sequence of vertices $w_0, \ldots, w_{\ell}$ 
such that each consecutive  pair $(w_{i-1},  w_{i})$ is an edge in $G$.  
The length of the walk is equal to the number of consecutive  pairs $(w_{i-1},  w_{i})$.

\newcommand{\SpGMatrix}{\bS}

\subsubsection{The adjacency matrix}

For an undirected  graph $G=(V,E)$ the {\em adjacency matrix} $\simpleadj_G$ is $V\times V$ matrix with 
entries in  $\{0, 1\}$  such that for every pair $u,w\in V$ we have that
\begin{align}\nonumber 
\simpleadj_G( u,w)= {\bf  1} \{\textrm{ $u, w$ are adjacent in $G$}\}.
\end{align}
A very natural property of the adjacency matrix is that  for any two  $u,w\in V$ and  $\ell\geq 1$
we have that
\begin{align}\label{eq:NoOfWalksVSA2L}
\left ( \simpleadj^{\ell}_G\right) (u,w)&=\textrm{$\#$ length $\ell$ walks\   from}\ u\ \textrm{to}\  w.
\end{align}
Since we assume that the graph is undirected, we have that $\simpleadj^{\ell}_G$ is symmetric,
for any integer $\ell\geq 0$.   Hence,  $\simpleadj_G$  has real eigenvalues, while 
the eigenvectors corresponding to distinct  eigenvalues are orthogonal with each other.

We denote with  $\eigenv_i  \in \mathbb{R}^V$  the  eigenvector of $\simpleadj_G$ 
that corresponds to the eigenvalue $\eigenval_i(\simpleadj_G)$, i.e., the $i$-th largest eigenvalue. 
Unless otherwise specified,  we have  $\textstyle \lnorm \eigenv_i \rnorm_2=1$.

Our assumption that $G$ is always undirected, connected implies that $\simpleadj_G$ is non-negative
and irreducible. Hence,  the Perron Frobenius Theorem (see Section\ref{sec:PerronFrobeniusThrm} in 
the Appendix) implies that
\begin{align}\label{eq:DefOfBSMatrix}
\spradius(\simpleadj_G)&=\eigenval_1(\simpleadj_G)& \textrm{and} && \maxeigenv(u) &>0 \qquad \forall  u\in V.
\end{align}
Note that if $G$ is bipartite, then we also have $\spradius(\simpleadj_G)=|\eigenval_n(\simpleadj_G)|$.

Furthermore,  we let $\SpGMatrix$ be the $V\times V$ diagonal  matrix such that for any
$u\in V$ we have that
\begin{align}\label{def:AdjMatrixEigenvals}
\SpGMatrix(v,v)=\maxeigenv(u).
\end{align}
For a connected graph $G$, which is the case here,  we have that 
$\SpGMatrix$ is non-singular, i.e., since $\maxeigenv(u)>0$.

\newcommand{\SpGHasMatrix}{\bL}

\subsubsection{The Hashimoto non-backtracking  matrix}
Here we define the well-known Hashimoto non-backtracking matrix, first introduced in  \cite{Hash89}
and has been studied in mathematical physics.

For the graph $G=(V,E)$,  let $\OrntEdges$ be the set of oriented edges 
obtained by doubling each  edge of $E$ into two directed edge, i.e.,  one edge for each direction. The 
non-backtracking matrix,  denoted as $\NBMatrix_G$,  is an $\OrntEdges \times \OrntEdges $ matrix  
with entries in $\{0,1\}$,  such that for any pair of  oriented edges $e=(u,w)$ and $f=(z, y)$ we have that
\begin{align}\nonumber 
\NBMatrix_G(e, f) ={\bf 1}\{w=z \}\times {\bf 1}\{ u\neq y\}.
\end{align}
That is, $\NBMatrix(e, f)$ is equal to $1$, if $f$ follows the edge $e$ without creating a loop,  otherwise,  
it is equal to zero.  

Note that   $\NBMatrix_G$ is  not normal,  i.e., it does not  commute with its transpose.  
However, it posses a certain kind of symmetry that we exploit in  the analysis. 
Denoting  with $e^{-1}$, the edge that has the opposite direction to the edge $e\in \OrntEdges$, $\NBMatrix_G$ exhibits the
following symmetry: for any $e,f\in \OrntEdges$ and for any $\ell\geq 0$ we have that
\begin{align}\label{def:PTInvariance}
\NBMatrix^{\ell}_G(e,f) = \NBMatrix^{\ell}_G(f^{-1}, e^{-1}). 
\end{align}
In mathematical physics, this type of symmetry is called {\em PT-invariance}, where PT stands  for parity-time.

Drawing an analogy to \eqref{eq:NoOfWalksVSA2L},   for any integer $\ell\geq 1$ and any $e,f\in\OrntEdges $
we have that  
\begin{align}\label{eq:NoOfWalksVSH2L}
\textstyle \left ( \NBMatrix^{\ell}_G  \right) (e,f)&=\# \ \textrm{length $\ell$ non-backtracking walks that start\ from}\ e\ \textrm{and\ end\ at}\  f.
\end{align}
Recall that the  walk $w_0, \ldots, w_{r}$  is called  non-backtracking, if we have that  $w_{i-1}\neq w_i$ for all $i$.

\newcommand{\gratio}{{R}}
\newcommand{\trecur}{F}
\newcommand{\logtrecur}{H}
\newcommand{\potF}{{\Psi}}
\newcommand{\potFimg}{{S}}
\newcommand{\dlogtrecur}{h}
\newcommand{\ratiorange}{J}

\newcommand{\dpotF}{{\psi}}
\newcommand{\xdpotF}{\chi}

\section{Spectral Bounds  for  $\infmatrix^{\Lambda,\tau}_{G}$ using  Tree Recursions}\label{sec:RecursionVsSpectralIneq} 

We start, by considering  the tree $T=(V_T,E_T)$,  rooted at $r$, such that every vertex has at most 
$\maxDeg$ children,  for some integer $\maxDeg>0$. Also, let $\mu$ be a  Gibbs distribution 
on  $\{\pm 1\}^{V_T}$, specified as in \eqref{def:GibbDistr}  with respect to  the parameters $\beta, \gamma$ 
and $\lambda$.

For the region  $K \subseteq V_T\setminus\{r\}$ and   $\tau\in  \{\pm 1\}^{K}$, we consider the {\em ratio of 
marginals} at the root $\gratio^{K, \tau}_r$  such that
\begin{align}\label{eq:DefOfR}
\gratio^{K, \tau}_r=\frac{\mu_r(+1\ |\ K,  \tau )}{\mu_r(-1\ |\ K,  \tau)}.
\end{align}
Recall that  $\mu_r(\cdot \ |\ K,  \tau )$ denotes the marginal of the Gibbs distribution $\mu(\cdot \ |\ K,  \tau )$ 
at the root $r$. Also, note that the  above  allows  for $\gratio^{K, \tau}_r=\infty$,  e.g.,  when 
$\mu_r(-1\ |\ K,  \tau)=0$ and  $\mu_r(+1\ |\ K,  \tau )\neq 0$.

For a vertex $u\in V_T$,  we let $T_u$ be the subtree of $T$ that includes $u$ and all its descendents. 
We always assume  that the root of $T_u$ is the vertex $u$.  
With a slight abuse of notation,  we let  $\gratio^{K, \tau}_u$ denote the ratio of marginals at the root for 
the subtree $T_u$, where the Gibbs distribution is, now,  with respect to $T_u$, while we impose the boundary 
condition $\tau(K\cap T_u)$.

Suppose that the root $r$ is of degree $d>0$, while let the vertices $w_1, \ldots, w_{d}$
be its children.  We express $\gratio^{K, \tau}_r$ it terms of  $\gratio^{K, \tau}_{w_i}$'s 
by having  $\gratio^{K, \tau}_{r}=\trecur_d(\gratio^{K, \tau}_{w_1}, \gratio^{K, \tau}_{w_2}, 
\ldots, \gratio^{K, \tau}_{w_d})$,   for  
\begin{align}\label{eq:BPRecursion}
\trecur_d:[0, +\infty]^d\to [0, +\infty] & &\textrm{such that}& &
 (x_1, \ldots, x_d)\mapsto \lambda \prod^d_{i=1}\frac{\beta {x}_i+1}{{x}_i+\gamma}.
\end{align}

For  the analysis that  follows, we get cleaner results by equivalently working with log-ratios rather 
than  ratios of Gibbs marginals.  Let  $\logtrecur_d=\log \circ F_d \circ \exp$,  which means that 
\begin{align}\label{eq:DefOfH}
\logtrecur_d:[-\infty, +\infty]^d\to [-\infty, +\infty] &&\textrm{s.t.}&&  \textstyle (x_1, \ldots, x_d)\mapsto \log \lambda+\sum^d_{i=1}
\log\left( \frac{\beta \exp(x_i)+1}{\exp(x_i)+\gamma} \right).
\end{align}
From  \eqref{eq:BPRecursion}, it is elementary to verify that   $\log \gratio^{K, \tau}_{r}=\logtrecur_d(\log \gratio^{K, \tau}_{w_1}, 
\ldots, \log \gratio^{K, \tau}_{w_d})$.

Finally,  we let the function  
\begin{align}\label{eq:DerivOfLogRatio}
\dlogtrecur:[-\infty,+\infty]\to \mathbb{R}&& \textrm{s.t. }&&  
x \mapsto -\frac{(1-\beta\gamma)\cdot \exp(x)}{(\beta \exp(x)+1)(\exp(x)+\gamma)}.
\end{align}
It is straightforward  that for any $i\in  [d]$, we have that   $\frac{\partial }{\partial x_i}\logtrecur_d(x_1,  \ldots, x_d)=\dlogtrecur(x_i)$.
Note that, for any  integer $N>0$,  we let the set $[N]=\{1,2,\ldots, N\}$.

\newcommand{\mybrO}{[}
\newcommand{\mybrC}{]}

Furthermore, let  the  interval $\ratiorange_d \subseteq \mathbb{R}$  be defined  as follows: 
\begin{align} \nonumber 
\ratiorange_d &= \left \{
\begin{array}{lcl}
\mybrO (\log \lambda\beta^d), \log(\lambda/\gamma^d  \mybrC & \quad & \textrm{if   $\beta\gamma<1$},  \\ \vspace{-.3cm } \\
\mybrO  (\log \lambda/\gamma^d), \log(\lambda\beta^d)   \mybrC & \quad & \textrm{if   $\beta\gamma>1$.} 
\end{array}
\right . 
\end{align}
Standard  algebra implies that $\ratiorange_d$ contains all the log-ratios for a vertex with $d$ children.  
Also, let 
\begin{align}\label{eq:DefOfRatiorange}
 \ratiorange &=\textstyle \bigcup_{d\in [\maxDeg]}\ratiorange_d.
\end{align}
Note that the set $\ratiorange$  contains all log-ratios in the tree $T$. 

\subsection{A first attempt.}
Having introduced the notion of the (log-)ratio of Gibbs marginals and the related recursions
we present the first set of results that we use to establish spectral independence. We utilise these
results to prove  Theorems \ref{thrm:Ising4SPRadius} and \ref{thrm:Ising4SPRadiusNBK}  about 
the Ising model.

\begin{definition}[$ \delta$-contraction]\label{def:HContraction}
Let  $\delta\geq 0$,  the integer $\maxDeg\geq 1$ and $\beta,\gamma, \lambda\in \mathbb{R}$  are
such that  $0\leq \beta\leq  \gamma$, 
$\gamma >0$ and $\lambda>0$.  
We say that the set of functions $\{ \logtrecur_d\}_{d\in [\Delta]}$, defined in \eqref{eq:DefOfH}, 
exhibits  $\delta$-contraction,   with respect to  $(\beta,\gamma,\lambda)$,   if it satisfies  the following condition: 

For any   $d\in [\maxDeg]$  and every  $({y}_1,  \ldots, {y}_d)\in  [-\infty,+\infty]^d$  we have that 
\begin{align} \label{eq:ContractionDefNoPot}
||\nabla \logtrecur_d ({y}_1, \ldots, {y}_d) ||_{\infty}
 \leq \delta.
\end{align}
\end{definition}

Clearly, the condition in \eqref{eq:ContractionDefNoPot} is equivalent to having  
$\dlogtrecur(z)\leq \delta$, for any  $z\in [-\infty,+\infty]$.

\begin{theorem} \label{thrm:Recurrence4InfluenceEigenBound}
Let  $\maxDeg \geq 1$, $\spradius_G \geq 1$,    $\epsilon\in (0,1)$ and  $\beta,\gamma, \lambda\in \mathbb{R}$  
be such that  $0\leq \beta\leq  \gamma$,  $\gamma >0$ and $\lambda>0$. 
Let   $G=(V,E)$ be of maximum degree $\maxDeg$, while   $ \simpleadj_G$ is of spectral radius
$\spradius_G$. Also,  consider $\mu_G$  the Gibbs distribution on $G$, specified by the parameters
$(\beta, \gamma, \lambda)$. 

For $\delta =\frac{1-\epsilon}{\spradius_G}$,  suppose that  the set  of functions $\{ \logtrecur_d\}_{d\in [\maxDeg]}$ 
specified  with respect to  $(\beta,\gamma,\lambda)$ exhibits   $\delta$-contraction. Then, for  any 
$\Lambda\subset V$ and   any $\tau\in \{\pm 1\}^{\Lambda}$,  the pairwise influence matrix 
$\infmatrix^{\Lambda,\tau}_{G}$,  induced by $\mu_G$, satisfies that 
\begin{align}\nonumber  
\textstyle \spradius \left( \infmatrix^{\Lambda,\tau}_{G} \right) &\leq   \epsilon^{-1}.
\end{align}
\end{theorem}

We use the above result to prove  Theorem \ref{thrm:Ising4SPRadius} for the Ising model. Note
that Theorem \ref{thrm:Recurrence4InfluenceEigenBound} applies to a general Gibbs distribution, i.e., 
not necessarily only on the Ising model. In Section \ref{sec:GeneralGibbs}, in the Appendix, we show
how the above result implies rapid mixing for Glauber dynamics on  general Gibbs distribution.

In order to  prove Theorem \ref{thrm:Ising4SPRadiusNBK}, we use  the following Theorem
\ref{thrm:Recur4InfSPBoundNBK} which  is similar in spirit to   Theorem \ref{thrm:Recurrence4InfluenceEigenBound} 
but   exploits the spectrum of the non-backtracking   matrix. Note that the bound on the spectral radius of 
$\infmatrix^{\Lambda,\tau}_{G}$ we get from the theorem below  does not have a simple expression.

\begin{theorem}\label{thrm:Recur4InfSPBoundNBK}
Let  $\maxDeg \geq 1$, $\hspradius_G \geq 1$,    $\epsilon\in (0,1)$ and  $\beta,\gamma, \lambda\in \mathbb{R}$  
be such that  $0\leq \beta\leq  \gamma$,  $\gamma >0$ and $\lambda>0$. 
Let   $G=(V,E)$ be of maximum degree $\maxDeg$, while   $ \NBMatrix_G$ has  spectral radius
$\hspradius_G$. Also,  consider $\mu_G$  the Gibbs distribution on $G$, specified by the parameters
$(\beta, \gamma, \lambda)$. 

For $\delta =\frac{1-\epsilon}{\hspradius_G}$,  suppose that  the set  of functions $\{ \logtrecur_d\}_{d\in [\maxDeg]}$ 
specified  with respect to  $(\beta,\gamma,\lambda)$ exhibits   $\delta$-contraction. Then, for  any 
$\Lambda\subset V$ and   any $\tau\in \{\pm 1\}^{\Lambda}$,  the pairwise influence matrix $\infmatrix^{\Lambda,\tau}_{G}$, 
induced by $\mu_G$, satisfies that 
\begin{align}\label{eq:thrm:Recur4InfSPBoundNBK}
\textstyle \spradius \left( \infmatrix^{\Lambda,\tau}_{G} \right) &\leq  
\textstyle  \lnorm  \left( \bI -   \frac{1-\epsilon}{\hspradius_G} \simpleadj_G +\left( \frac{1-\epsilon}{\hspradius_G}\right)^2(\bD-\bI) \right)^{-1}   \rnorm_2, 
\end{align}
where $\simpleadj_G$ is the adjacency matrix of $G$ and $\bD$  is the $V\times V$ diagonal matrix such 
that for every $u\in V$ we have $\bD(u,u)={\tt degree}(u)$.
\end{theorem}
The proof of Theorem \ref{thrm:Recur4InfSPBoundNBK} appears in Section \ref{sec:thrm:Recur4InfSPBoundNBK}

We need to remark that for the bound on $ \spradius \left( \infmatrix^{\Lambda,\tau}_{G} \right)$ we get from 
Theorem \ref{thrm:Recur4InfSPBoundNBK}  we don't have guarantees that is always  bounded.
The assumption of $\delta$-contraction for $\delta =\frac{1-\epsilon}{\hspradius_G}$ implies that the quantity 
 on the  r.h.s. of \eqref{eq:thrm:Recur4InfSPBoundNBK} 
is finite, however it might be increasing with  $n$.  

It  is interesting to compare the bounds  on spectral radius  we get from  Theorems \ref{thrm:Recur4InfSPBoundNBK} 
and  \ref{thrm:Recurrence4InfluenceEigenBound}.
It is not obvious at all from its  statement, but for the same parameters $\beta$ of the Ising model 
the bound we get  from Theorem \ref{thrm:Recur4InfSPBoundNBK}  is {\em at most} that we get for
 from   Theorem \ref{thrm:Recurrence4InfluenceEigenBound}.  For further discussion on this matter, the reader is referred  to
the end  of Section \ref{sec:SPComparisonBasedOnEntries}.

In light of all the above theorems, one might be tempted to use a condition which is weaker than $\delta$-contraction, e.g., 
consider norms of $\nabla \logtrecur_d$  different  than $\ell_{\infty}$. 
This is a perfectly reasonable idea and  has been investigated in the literature in various different settings. 
Interestingly, with this approach it is natural to use  the so-called {\em potential method}.
In what follows, we introduce   techniques that exploit  the potential method to derive further 
results, somehow, stronger than  those we have so far.

\newcommand{\potSpace}{\Sigma}

\subsection{A second attempt.} 
Perhaps it is interesting to mention that using Theorem \ref{thrm:Recurrence4InfluenceEigenBound}
and working as in the proof of Theorem \ref{thrm:Ising4SPRadius} one can retrieve the rapid mixing 
results for the Hard-core model  in  \cite{Hayes06}.
In order to get improved results for the Hard-core mode, we  make use of potential  functions, while  we exploit results   
from \cite{ConnectiveConst}.

Let $\potSpace$ be  the set of functions $F:[-\infty,+\infty]\to(-\infty,+\infty)$ which are {\em differentiable} 
and {\em increasing}.

\begin{definition}[$(s,\delta, c)$-potential]\label{def:GoodPotential}
Let $s\geq 1$,  allowing  $s=\infty$,  $\delta,c>0$ and  let the integer $\maxDeg\geq 1$. 
Also, let  $\beta,\gamma, \lambda\in \mathbb{R}$  be  such that  $0\leq \beta\leq  \gamma$, 
$\gamma >0$ and $\lambda>0$.

Consider $\{ \logtrecur_d\}_{d\in [\Delta]}$, defined in \eqref{eq:DefOfH} with respect to 
$(\beta,\gamma,\lambda)$. The function $\potF\in \potSpace$, with image $S_{\potF}$,  is called $(s, \delta, c)$-potential  
if it satisfies   the following two conditions:  
\begin{description}
\item[Contraction]
For    $d\in [\maxDeg]$,  for    $(\tilde{\bf y}_1,  \ldots, \tilde{\bf y}_d)\in  (S_{\potF})^d$,  
 and ${\bf m}=({\bf m}_1,  \ldots, {\bf m}_d)\in \mathbb{R}^{d}_{ \geq 0}$ we have that 
\begin{align} \label{eq:contractionRelationPF}
 \xdpotF\left( \logtrecur_d({\bf y}_1,  \ldots, {\bf y}_d) \right)  \cdot \sum^{d}_{j=1}
\frac{ \left| \dlogtrecur\left({\bf y}_j \right)\right|}{ \xdpotF\left( {\bf y}_j \right)}  \cdot  {\bf m}_j  \leq  \delta^{\frac{1}{s}} \cdot \lnorm {\bf m}  \rnorm_{s},
\end{align}
where $\xdpotF=\potF'$,  ${\bf y}_j=\potF^{-1}(\tilde{\bf y}_j)$,    while  $\dlogtrecur(\cdot)$  is the function defined in \eqref{eq:DerivOfLogRatio}.
\item[Boundedness] We have that
\begin{align}\label{eq:BoundednessPF}
 \max_{{\bf z}_1,{\bf z}_2\in \ratiorange}\left\{ \xdpotF({\bf z}_1) \cdot \frac{|\dlogtrecur({\bf z}_2)|}{\xdpotF({\bf z}_2)} \right\} & \leq c.
\end{align}
\end{description}
\end{definition}
Recall that the set $\ratiorange$ in the index of $\max$ in \eqref{eq:BoundednessPF}, is  defined in \eqref{eq:DefOfRatiorange}
and  includes all the values of the   log-ratios for a vertex with $d$ children, where $d\in [\maxDeg]$.

The notion of  $(s,\delta, c)$-potential function  we have above,  is a generalisation of the 
so-called ``$(\alpha,c)$-potential function"  that was introduced   in  \cite{VigodaSpectralInd}. Note that 
the notion of $(\alpha,c)$-potential function  implies the use of the   $\ell_1$-norm in the analysis. 
The setting we consider here is more general. The condition in \eqref{eq:contractionRelationPF},  
somehow, implies that  we are using the $\ell_{r}$-norm in our analysis, where $r$ is the H\"older conjugate  of the
parameter $s$ in  the $(s,\delta, c)$-potential function\footnote{That is, $r$ satisfies that  $r^{-1}+s^{-1}=1$. }.

\begin{theorem}\label{thrm:Recurrence4InfluenceEigenBoundNonLinear}
Let  $\maxDeg\geq 1$, $\spradius_G\geq 1$,   $s\geq 1$,  allowing $s=\infty$,  let $\epsilon\in (0,1)$
and $\zeta>0$. Also, let   $\beta,\gamma, \lambda\in \mathbb{R}$ be such that  $\gamma >0$,
$0\leq \beta\leq  \gamma$  and $\lambda>0$.  Consider the graph $G=(V,E)$ of  maximum degree $\maxDeg$,
while   $\simpleadj_G$  is of  spectral radius $\spradius_G$. Consider, also,  $\mu_G$  the Gibbs
distribution on $G$ specified by the parameters $(\beta, \gamma, \lambda)$.

For $\delta= \frac{1-\epsilon}{\spradius_{G}}$ and $c=\frac{\zeta}{\spradius_G}$, suppose that there
is a  $(s,\delta,c)$-potential  function $\potF$ with respect to  $(\beta,\gamma,\lambda)$. Then, for any
$\Lambda\subset V$, for any  $\tau\in \{\pm 1\}^{\Lambda}$,   the influence matrix
$\infmatrix^{\Lambda,\tau}_{G}$, induced by $\mu_G$, satisfies that 
\begin{align}\label{eq:thrm:Recurrence4InfluenceEigenBoundANonLin}
\spradius  \left( \infmatrix^{\Lambda,\tau}_{G} \right) &\leq
1+ \zeta \cdot (1-(1-\epsilon)^s)^{-1} \cdot \left( {\maxDeg}/{\spradius_G} \right)^{1-\frac{1}{s}}.
\end{align}
\end{theorem}

Note that the bound  in \eqref{eq:thrm:Recurrence4InfluenceEigenBoundANonLin} includes the maximum
degree  $\maxDeg$. We can, easily, remove the dependency on $\maxDeg$ by using that  $\sqrt{\maxDeg} \leq \spradius_G$ 
and get 
\begin{align} \nonumber 
\spradius  \left( \infmatrix^{\Lambda,\tau}_{G} \right) &\leq 
1+ \zeta \cdot (1-(1-\epsilon)^s)^{-1}  \cdot \left( {\spradius_G} \right)^{1-\frac{1}{s}}.
\end{align}
We use Theorem \ref{thrm:Recurrence4InfluenceEigenBoundNonLinear} in order to prove 
Theorem \ref{thrm:HC4SPRadius} for the Hard-core model. 

Similar to Theorem \ref{thrm:Recurrence4InfluenceEigenBound}, the above result  applies to a general 
Gibbs distribution, i.e.,  not necessarily only to  the Hard-core model. In Section \ref{sec:GeneralGibbs}, in 
the Appendix, we show how Theorem \ref{thrm:Recurrence4InfluenceEigenBoundNonLinear} implies rapid mixing 
for Glauber dynamics on  general Gibbs distribution.

\section{The topological method - Basic Concepts}\label{sec:TopologicalMethod}\label{sec:TopoligcalMethod101}

\subsection{Walk-Trees}\label{sec:WalkTrees}

In this section, we introduce the notion of  {\em walk-tree}.  Walk-trees are  topological constructions which 
are defined with  respect to the   graph   $G=(V,E)$ and a set of walks $\cP$ in this graph. 
This notion generalises constructions  from the {algebraic graph theory}   and elsewhere such as the 
{\em tree of self-avoiding walks}, or  the {\em path-trees} (introduced  by Godsil, e.g., see  
\cite[Chapter 6]{GodsilBookComb}), or the   {\em universal cover}  (e.g. see  \cite{InterlacingFam}) etc.

In the graph $G$, a  {\em walk}  $P$ is a sequence  of  vertices  $w_0, \ldots, w_{\ell}$ 
such that each consecutive  pair $(w_{i-1},  w_{i})$ is an edge in $G$. The length of the  walk 
$P$, denoted as $|P|$, is equal to the number of consecutive  pairs $(w_{i-1},  w_{i})$. 
Unless otherwise specified  all the walks we consider will be of {finite length}. 
With the above definition, we consider the  single vertex to be a walk of  zero length.

Any two walks  are considered  to be {\em adjacent} with each other if and only if one  of them extends 
the other,  e.g.  the  walks    $P'=w_0, w_1, \ldots, w_{\ell}$ and $P=w_0, w_1, \ldots, w_{\ell}, w_{\ell+1}$ 
are adjacent. Furthermore,  the set $\cP$ of walk  in $G$ is called   {\em connected},  
if for every walk $P\in \cP$ of  length $\ell\geq 1$, there exists another  walk $P'\in \cP$ of length $\ell-1$, such 
that $P$ and $P'$ are adjacent.

The above definition implies that if the non-empty set $\cP$ is connected, then it should contain at least one
walk of length zero. Particularly,  the following holds for $\cP$: if there is a vertex $r$ and a path in  
$\cP$ that  emanates from $r$,  then $\cP$ must include the path of length zero that consists only
of the vertex $r$.

Let  $\cP$ be a  connected set of walks in $G$ and let  the vertex $r\in V$. Suppose  that there  is at least one 
walk in $\cP$  that starts from $r$.  We define $\walkT_{\cP}(r)$, the {\em walk-tree}  induced by  $(\cP, r)$, 
as follows:  the vertices  of $\walkT_{\cP}(r)$ correspond   to the walks in $\cP$ that emanate from the 
vertex $r$.  If two walks in $\cP$ are adjacent, then their corresponding vertices in $\walkT_{\cP}(r)$ are 
adjacent, too. The root of $\walkT_{\cP}(r)$ is the vertex  that corresponds to the walk that includes only 
the vertex $r$.

%Note that  $\walkT_{\cP}(r)$ is well-defined. The assumptions that $\cP$ is connected and that there is at
%least one path  that emanates from $r$, implies that $\cP$ includes the single vertex walk that consists of  
%the vertex $r$. Unless otherwise specified, 
Since we always assume that $\cP$ is a finite set,   the corresponding walk-trees that are generated by
$\cP$ are finite graphs.   When $\cP$ does not include a path that starts from the vertex $r$, then we follow the 
convention that $\walkT_{\cP}(r)$ is the empty graph.

There are natural extensions of the notion of walk-tree $\walkT_{\cP}(r)$ when the set of walks $\cP$ is not 
connected.  We don't consider  these cases here, as we will always deal with sets of walks 
that are connected.  Unless otherwise specified  when we refer to a set of walks $\cP$ we 
{always} assume that it is connected.

\newcommand{\WEdges}{{\tt E}}
\newcommand{\WVertices}{\cp}

For each  $u\in V$, we let $\cp_{\cP, r}(u)$ be the subset  of vertices in $\walkT_{\cP}(r)$ which correspond to 
walks $w_0, \ldots, w_{\ell}$ in $\cP$ such that $w_{\ell}=u$ . We  refer to  $\cp_{\cP, r}(u)$ as the 
{\em set of  copies}  of vertex $u$ in $\walkT_{\cP}(r)$. Note that it could be that $\cp_{\cP, r}(u)=\emptyset$. 
Also, we let 
\begin{align}\nonumber
\WVertices_{\cP, r}=\textstyle\cup_{u\in V}\cp_{\cP, r}(u).
\end{align}
That is,   $\cp_{\cP, r}$ corresponds to the vertex set of $\walkT_{\cP}(r)$. We also let $\WEdges_{\cP, r}$ be 
the set of edges of $\walkT_{\cP}(r)$.

We consider   various  kinds of walk-trees  in order to prove our results.   In what follows we give  examples 
of  walk-trees which  we use in  the analysis. In the  first example  we have the tree of self-avoiding walks, 
encountered Section \ref{sec:HighOverview}. We present it now using the terminology of 
walk-trees.

\newcommand{\saw}{{\rm SAW}}

\subsubsection*{$\saw$ walk-tree.}
The walk $w_0, \ldots, w_{\ell}$ in $G$ is called  {\em self-avoiding}, if we have that $w_i\neq w_j$ for any  $i\neq j$.
A single vertex is also considered to be a self-avoiding walk. 

Let $\saw$ be the set  of  the  sequences of vertices $w_0, \ldots, w_{\ell}$, for $\ell\geq 0$, such that  one of 
the following two holds:
\begin{enumerate}
\item  $w_0, \ldots, w_{\ell}$ is a self-avoiding walk, 
\item  $w_0, \ldots, w_{\ell-1}$ is a self-avoiding walk, while there is $0\leq j\leq \ell-3$ such that $w_{\ell}=w_{j}$.
\end{enumerate}
It is straightforward that the longest sequence in $\saw$ has length $n$, i.e., the number of vertices in $G$. In that respect $\saw$ is a finite set. 
Furthermore, it is elementary to verify that for any $\ell\geq 1$ and $P=w_0, \ldots, w_{\ell-1}, w_{\ell}$ such that $P\in \saw$, 
the walk $P'=w_0, \ldots, w_{\ell-1}$ is also in $\saw$.  This implies that $\saw$ is connected. 

The $\saw$ walk-tree $\Tsaw(r)$,  (or the tree of self-avoiding walks that starts from $r$) is  the walk-tree that is induced by 
the set of walks $\saw$ and  vertex $r\in V$.

\newcommand{\nb}{\textrm{NB}}
\newcommand{\nbk}{\textrm{NB-$k$}}
\newcommand{\nbn}{\textrm{ NB-$n$}}
\newcommand{\nbInf}{\textrm{ NB-$\infty$}}

\subsubsection*{Non-backtracking walk-tree.}
Recall that the walk $w_0, \ldots, w_{\ell}$ in $G$ is called  {\em non-backtracking}, if we have that 
$w_{i-1}\neq w_i$ for all $i$.  A single vertex is also considered to be a non-backtracking walk. 

For  integer $k\geq 0$, we let $\nb$-$k$ be the set of  all  non-backtracking 
walks in $G$ which  are of length at most $k$. It is elementary to verify that $\nbk$  is connected, i.e.,  
for any $\ell\geq 1$  and  $P=w_0, \ldots, w_{\ell-1}, w_{\ell}$ such that $P\in\nbk$,   
the walk $P'=w_0, \ldots, w_{\ell-1}$ also belongs to  $\nbk$.

The  $\nbk$ walk-tree  $\walkT_{\nbk}(r)$  corresponds to the walk-tree that is induced by the 
set of walks $\nbk$  and   vertex $r\in V$.

A variant of the non-backtracking walk-tree  that is commonly used in the algebraic graph theory is the 
tree  $\walkT_{\nbk}(r)$, for $k= \infty$. This is known as the {\em universal cover}  of $G$,  e.g., see 
\cite{GodsilBookComb}.    Note that the universal cover is an infinite walk-tree.

\subsubsection*{$\KWalks$ walk-tree.}
Another  kind of walk-tree that we  consider here is  what we call   $\KWalks$ walk-tree, where $k\geq 0$ is an  integer 
parameter.  % this is also called ``quotient"?

 We let    $\KWalks$ be the set that contains all the  walks in the graph $G$ that  have  
length at most $k$.
Arguing as in the previous two examples, we have that   $\KWalks$ is connected.

The  $\KWalks$ walk-tree  $\Tall(r)$,  corresponds to the walk-tree that is induced by the set $\KWalks$  and 
the vertex $r\in V$.

\subsubsection*{Relations between walk-trees}
In our analysis it is common that we consider $\cP$ and $\cQ$  two  sets of walks in the graph $G=(V,E)$ and
we need to deduce a relationship between the walk-trees $\walkT_{\cQ}(w)$  and $\walkT_{\cP}(w)$, for some 
vertex $w\in V$.

Recall that   we denote $\WVertices_{\cP, w}, \WEdges_{\cP, w}$  the sets of vertices and edges, respectively, 
of the walk-tree $\walkT_{\cP}(w)$.  Similarly,  $\WVertices_{\cQ, w}, \WEdges_{\cQ, w}$  for the walk-tree $\walkT_{\cQ}(w)$.

Each element in $\WVertices_{\cP, w}$ and $\WVertices_{\cQ, w}$ corresponds to a walk in the graph $G$. 
If there are  elements $w_1\in  \WVertices_{\cP, w}$ and   $w_2\in \WVertices_{\cQ, w}$
that correspond to the same walk $P$ in the graph $G$, then we consider that $w_1$ and $w_2$ are {\em identical}. 
This allows us  to define the standard set relations  for  
$\WVertices_{\cP, w}$ and $\WVertices_{\cQ, w}$, e.g., containment, intersection, etc.

We work similarly for  $\WEdges_{\cP, w}$ and $\WEdges_{\cQ, w}$. Each element in  
$\WEdges_{\cP, w}$ (and similarly $\WEdges_{\cQ, w}$) corresponds to an edge that extends 
 walk $P'$  to its adjacent walk  $P$,  in the graph $G$. If there are elements $e_1\in  \WEdges_{\cP, w}$ 
and  $e_2\in \WEdges_{\cQ, w}$ that both correspond to the same edge that extends  walk 
$P'$ to its adjacent walk $P$, then we consider $e_1$ and $e_2$ to be identical.  As before, this 
allows us to define the standard set relations  for   $\WEdges_{\cP, w}$ and $\WEdges_{\cQ, w}$.

\begin{lemma}\label{lemma:StrongSubtreeRelation}
Let    $\cP$ and $\cQ$ be  two  sets of walks in   $G=(V,E)$
such  that  $\cQ\subseteq \cP$.  Suppose that for the vertex $w\in V$ the walk-trees $\walkT_{\cP}(w)$ and 
$\walkT_{\cQ}(w)$ are nonempty.  
Then, we have that $\walkT_{\cQ}(w)$, is a subtree of  $\walkT_{\cP}(w)$.
\end{lemma}
The  proof of Lemma \ref{lemma:StrongSubtreeRelation} appears in Section \ref{sec:lemma:StrongSubtreeRelation}.

\subsection{Walk-Matices}\label{sec:DefOfWalkMatrix}

For a set of walks  $\cP$ in  $G=(V,E)$,  we let $V_{\cP}\subseteq V$ be the set of vertices $r\in V$ 
for which  there is at least one  path $P\in \cP$ such that either $P$ start from vertex $r$, or $P$ ends at vertex $r$. 

For every $r\in V_{\cP}$, recall that  $\WEdges_{\cP, r}$ is the set of edges of the walk-tree $\walkT_{\cP}(r)$. 
We let
\begin{align}\nonumber 
\WEdges_{\cP}=\cup_{r\in V_{\cP}}\WEdges_{\cP, r}.
\end{align}
Note that it is possible that an edge $e$  appears in multiple walk trees. In $\WEdges_{\cP}$ 
we treat each occurrence of $e$ in the sets $\WEdges_{\cP, r}$ as a different element. 
That is, each element in $\WEdges_{\cP}$ is identified by the edge $e$ and the tree that this
edge belongs. 

We  let the set of weights $\bxi \in \mathbb{R}^{\WEdges_{\cP}}$. That is,  $\bxi$ assigns to each 
edge  $e\in \WEdges_{\cP}$ weight $\bxi(e)$.  Note that the definition of $\WEdges_{\cP}$ allows
the same edge in different trees  to take different weights.

We define the  {\em walk-matrix}   $\Pweight_{\cP, \bxi} \in \mathbb{R}^{V_{\cP}\times V_{\cP}}$  such that
for every pair of  vertices $r, u\in V_{\cP}$ we have that
\begin{align}\label{eq:WalkMatrixEntryAsSumOfPathweights}
\Pweight_{\cP, \bxi}(r,u) &=\textstyle \sum_{M}{\tt weight}(M),
\end{align} 
where $M$ varies over all paths in $\walkT_{\cP}(r)$ from the root to the set of copies of $u$ in the tree, 
i.e., the  set of vertices $\cp_{\cP, r}(u)$,  while 
\begin{align}\label{def:WeightPath}
{\tt weight}(M) &=\textstyle \prod_{e\in M}  \bxi(e).
\end{align}
That is, ${\tt weight}(M)$ is equal to the product of weights of the edges in $M$. When $M$ does not have 
any edges, i.e., $M$ is a single vertex path,  we follow the convention that ${\tt weight}(M)=1$.

For the sake of completeness  we consider  the following  extreme cases.  If there are $r, u\in V_{\cP}$ 
such that either $\cp_{\cP, r}(u)=\emptyset$,  
%or there are no paths  from the root of $\walkT_{\cP}(r)$ to  $\cp_{\cP, r}(u)$, 
then  we have that $\Pweight_{\cP, \bxi}(r,u)=0$. Also,   if $\walkT_{\cP}(r)$  is empty, 
then $\Pweight_{\cP, \bxi}(r,u)=0$ for all $u\in V_{\cP}$.

We remark that  the walk-matrix $\Pweight_{\cP, \bxi}$ is a {\em general matrix} and 
{ does not}  necessarily have any special algebraic structure, e.g.,   symmetry,   irreducibility, or 
being positive etc. In what follows, we show some interesting cases of  walk-matrices, some of 
which we have already  encountered in our discussion, while some others we will
encounter in the analysis that follows.

For  integer $k\geq 0$, consider $\Pweight_{\KWalks, \bpsi}$, the walk-matrix that is induced 
by the set of walks $\KWalks$ and the constant vector $\bpsi\in \mathbb{R}^{\WEdges_{\KWalks}}$
such that for every  edge $e$ we have  $\bpsi(e)=\zeta$, where $\zeta>0$.

\begin{lemma}\label{lemma:WalkMatrixVsAdjacency}
For the matrix $\Pweight_{\KWalks, \bpsi}$ we define above
% and $\simpleadj_G$,  the adjacency matrix of $G$,    
we have that
\begin{align}\nonumber %\label{eq:KWalkVsAdjacencyMatrix}
\Pweight_{\KWalks, \bpsi} &= \textstyle \sum^{k}_{\ell=0} (\zeta \cdot \simpleadj_G)^{\ell}.
\end{align} 
\end{lemma}

The above results implies that $\sum^{k}_{\ell=0} (\zeta \cdot \simpleadj_G)^{\ell}$ can be regarded
as a walk-matrix.  The proof of Lemma \ref{lemma:WalkMatrixVsAdjacency} is standard. For the sake 
of completeness, the reader can find  its proof  in Section \ref{sec:lemma:WalkMatrixVsAdjacency}.

In the following example, we consider a walk-matrix that is related to the non-backtracking matrix $\NBMatrix_G$.
For  integer $k\geq 0$, consider $\Pweight_{\nbk, \bpsi}$, the walk-matrix  induced by the set of walks $\nbk$
and the constant vector $\bpsi\in \mathbb{R}^{\WEdges_{\KWalks}}$
such that for every  edge $e$ we have  $\bpsi(e)=\zeta$, where $\zeta>0$. 
%
%all the non-backpacking  walks in $G$ which are of length  
%$\leq k$, while for every walk-tree $\walkT_{\nbk}(r)$ every edge $e$ has the same weight $\bpsi(e)=\zeta$. 

\begin{lemma}\label{lemma:WalkMatrixVsHashimoto}
For the walk-matrix $\Pweight_{\nbk, \bpsi}$ we define above, we have that 
\begin{align}\label{eq:NBKVsHashimoto}
\Pweight_{\nbk, \bpsi} &= \textstyle \bI + \vtooedge \cdot \left( \sum^{k-1}_{\ell=0} \zeta^{\ell+1} \cdot \NBMatrix^{\ell}_G\right)\cdot \oedgetov,
\end{align} 
where $\NBMatrix_G$ is the $\OrntEdges\times \OrntEdges$ non-backtracking matrix of $G$, while 
$\vtooedge$, $\oedgetov$ are   $V\times \OrntEdges$ and $\OrntEdges\times V$  matrices, respectively,   such that 
for any $v\in V$ and for any $(x,z)\in \OrntEdges$ we have 
\begin{align}\label{def:HasVertex2EdgeMatrices}
\vtooedge(v,(x,z)) &= {\bf 1}\{x=v\}
& \textrm{and} &&
\oedgetov((x,z), v) &=  {\bf 1}\{z=v\}.
\end{align}
\end{lemma}
The proof of Lemma \ref{lemma:WalkMatrixVsHashimoto} appears in Section \ref{sec:lemma:WalkMatrixVsHashimoto}. 

Note that the matrices $\vtooedge$ and $\oedgetov$ are used to address the discrepancy between 
$\Pweight_{\nbk, \bpsi} $ and $\NBMatrix_G$ that the first matrix is a $V\times V$, while the second one 
is  $\OrntEdges\times \OrntEdges$.

In what follows,  we show that the influence matrix $\infmatrix^{\Lambda,\tau}_{G} $ can also be considered as 
a special case of walk-matrix. 
%
%This implies that we can apply to $\infmatrix^{\Lambda,\tau}_{G} $  the  results we prove in this work for 
%walk-matrices. 
   
\subsection{$\infmatrix^{\Lambda,\tau}_{G}$ viewed as a walk-matrix}\label{sec:InfluenceVsWalkMatrix}
Consider the Gibbs distribution $\mu_G$ on the graph $G=(V,E)$,  defined as in \eqref{def:GibbDistr} with 
parameters $\beta,\gamma,\lambda\geq 0$.
For  $\Lambda\subseteq V$ and $\tau\in \{\pm1\}^{\Lambda}$, recall  the definition of
the  influence matrix $\infmatrix^{\Lambda,\tau}_{G}$ induced by $\mu$,  from \eqref{def:InfluenceMatrix}.

Consider the walk-tree $T=\Tsaw(w)$.   In what follows, we describe how the  entry 
$\infmatrix^{\Lambda,\tau}_{G}(w,v)$  can be expressed using   an appropriately defined  
spin-system on $T$.  The exposition  relies on  results from \cite{OptMCMCIS,VigodaSpectralInd}.

Assume w.l.o.g. that  there is a {\em total ordering} of the vertices in  $V$, i.e., the vertex set of 
$G$. Recall that for every $u\in V$, $\cp_{\saw,w}(u)$ corresponds to the set of copies of $u$ 
in $T$. In what follows, we abbreviate $\cp_{\saw,w}(u)$ to $\cp(u)$.

Let $\mu_T$ be a Gibbs distribution on $T$ which  has the same specification as $\mu_G$.  That is, 
for $\mu_T$ we use the same parameters $\beta,\gamma$ and $\lambda$ as those  we have for 
$\mu_G$.   Each  $z \in \cp(u)$ in the tree $T$, such that   $u \in \Lambda$, is  assigned fixed 
configuration  equal to  $\tau(u)$. Furthermore, if we have a vertex $z$ in $T$ which corresponds 
to  a  path  $w_0, \ldots, w_{\ell}$ in $G$  such  that  $w_{\ell}=w_j$, for   $0\leq j \leq \ell-3$, then 
we set a  boundary condition at vertex $z$, as well.  This boundary condition depends on the 
total ordering  of the vertices. Particularly, we  set at $z$
\begin{enumerate}[(a)]	
\item $-1$ if $w_{\ell}>w_{\ell-1}$, 
\item $+1$ otherwise.
\end{enumerate}
Let $\Gamma=\Gamma(G,\Lambda)$ be the set of vertices in $T$ which have a boundary condition 
in the above construction, while let $\sigma=\sigma(G,\tau)$  be the configuration we obtain  at $\Gamma$.

\newcommand{\infweight}{\mathbold{\beta}}
\newcommand{\LamSAW}{\mathcal{S}}

Consider the set $\WEdges_{\saw, w}$,  i.e., the set of edges in $\Tsaw(w)$. For  each 
$e\in \WEdges_{\saw, w}$  we specify  weight $\infweight(e)$ as follows: letting $e=\{x,z\}$ such 
that $x$ is the parent of $z$ in $\Tsaw(w)$, we set 
\begin{align}\label{def:OfInfluenceWeights}
\infweight(e)=\left \{ 
\begin{array}{lcl}
0 & \quad&\textrm{if there is  boundary condition at either $x$, or $z$},\\
\textstyle \dlogtrecur\left(\log \gratio^{\Gamma, \sigma}_{z}\right) & & \textrm{otherwise}.
\end{array}
\right .
\end{align}
The function $\dlogtrecur(\cdot)$ is from  \eqref{eq:DerivOfLogRatio},  while $\gratio^{\Gamma, \sigma}_{z}$ 
is a ratio of Gibbs marginals at $z$ (see definitions in Section \ref{sec:RecursionVsSpectralIneq}). 

It is easy to  see that  \eqref{def:OfInfluenceWeights} specifies weights for every edge 
$e\in \WEdges_{\saw, w}$, for every $w\in V$.  In that respect, all the above construction gives rise to 
the $V\times V$ walk-matrix $\Pweight_{\saw,\infweight}$. 
As we will show later (see proof of Lemma \ref{lemma:InfluenceMatrixIsWalkMatrix})  that
\begin{align}\label{eq:InfluenceEntryAsWeightedSum}
\infmatrix^{\Lambda,\tau}_{G}(w,v)=\sum_{M}\prod_{e\in M}\infweight(e),
\end{align}
where $M$ varies over all paths from the root of $\Tsaw(w)$ to the set of vertices in $\cp(v)$. 

In light of the above, it is immediate  that   $\infmatrix^{\Lambda,\tau}_{G}$ is a principal 
submatrix of $\Pweight_{\saw,\infweight}$. 
That is,   removing the rows and the columns of  $\Pweight_{\saw,\infweight}$  that correspond 
to the vertices $u\in \Lambda$ we obtain $\infmatrix^{\Lambda,\tau}_{G}$.

We can be more precise than the above. 
% and show that $\infmatrix^{\Lambda,\tau}_{G}$ is not
%only a principal submatrix of a walk-matrix but itself is a walk-matrix.  
Consider walk-matrix  $\Pweight_{\LamSAW,\bphi}$, where $\LamSAW\subseteq \saw$ is the  set of self-avoiding 
walks  in $G$ that do not use vertices in $\Lambda$, while  $\bphi\in \mathbb{R}^{\WEdges_{\LamSAW}}$ is 
a restriction $\infweight$, i.e., 
\begin{align}\label{def:OfRestInfluenceWeights}
\bphi(e) &=\infweight(e),  & \forall e\in \WEdges_{\LamSAW}.
\end{align}
Note  that since $\LamSAW\subseteq \saw$, we also have that  
$\WEdges_{\LamSAW}\subseteq \WEdges_{\saw}$. Hence, $\bphi$ is well defined. 
In what follows, we show that $\Pweight_{\LamSAW,\bphi}$ and $\infmatrix^{\Lambda,\tau}_{G}$
are identical.

\begin{lemma}\label{lemma:InfluenceMatrixIsWalkMatrix}
For the graph $G=(V,E)$, for   $\beta,\gamma, \lambda$ such that  $\gamma>0$,  
$0\leq \beta\leq  \gamma$ and $\lambda>0$ consider  $\mu_G$, the Gibbs distribution on $G$ 
with  parameters $\beta,\gamma, \lambda$.  

 For $\Lambda\subset V$ and  $\tau\in \{\pm 1\}^{\Lambda}$, consider the  influence matrix $\infmatrix^{\Lambda,\tau}_{G}$ induced by $\mu_G$
 as well as the walk-matrix   $\Pweight_{\LamSAW,\bphi}$, where $\LamSAW$ and $\bphi\in \mathbb{R}^{\WEdges_{\cP}}$ are defined above. 
Then, we have that $\infmatrix^{\Lambda,\tau}_{G}=\Pweight_{\LamSAW,\bphi}$.
\end{lemma}

The proof of Lemma \ref{lemma:InfluenceMatrixIsWalkMatrix} appears in Section \ref{sec:lemma:InfluenceMatrixIsWalkMatrix}.

\section{Spectral Independence using   entry-based comparisons} \label{sec:SPComparisonBasedOnEntries}

In this section we present the first set of results that we use to prove Theorems 
\ref{thrm:Recurrence4InfluenceEigenBound}  and \ref{thrm:Recur4InfSPBoundNBK}.
The objective here is to use the notion from the topological method to develop criteria 
which  allow us to compare two walk-matrices in terms of their  corresponding spectral radii.  
To this end we prove the following result.

\begin{theorem}[Monotonicity for walk-matrices]\label{thrm:Monotonicity4WeiightedMatrix}
Let   $\cP$ and $\cQ$ be sets of  walks in the graph $G=(V,E)$ such that $\cQ\subseteq \cP$.
Furthermore, suppose that we have   $\bxi_1\in \mathbb{R}^{\WEdges_{\cQ}}$ and  $\bxi_2\in \mathbb{R}^{\WEdges_{\cP}}_{\geq 0}$ 
such that
\begin{align}
0\leq \left| \bxi_1(e)\right|  &\leq \bxi_2(e), & \forall e\in \WEdges_{\cQ}\subseteq \WEdges_{\cP}.
\end{align}
Then,  for any $u,v\in V_{\cQ}\subseteq V_{\cP}$, we have that  
\begin{align}\label{eq:thrm:Monotonicity4WeiightedMatrixA}
\Pweight_{\cQ,\bxi_1}(u,v) &\leq \Pweight_{\cP,\bxi_2}(u,v).
\end{align}
Furthermore, if either $V_{\cP}=V_{\cQ}$, or  the matrix $\Pweight_{\cP,\bxi_2}$ is a symmetric matrix, 
then  we also have that 
\begin{align}\label{eq:thrm:Monotonicity4WeiightedMatrixB}
\spradius(\Pweight_{\cQ,\bxi_1})\leq \spradius(\Pweight_{\cP,\bxi_2}).
\end{align}
\end{theorem}

The proof of Theorem \ref{thrm:Monotonicity4WeiightedMatrix} appears in Section \ref{sec:thrm:Monotonicity4WeiightedMatrix}.

For  any set of walks  $\cP$ in $G$ such every $M\in \cP$ is  of length  at most $k\geq 0$,  
we have that 
$$\cP\subseteq\KWalks.$$ 
This follows by noting that every  $M\in \cP$ is a walk  in $G$ of length $\leq k$, hence,  we  also have that $M \in \KWalks$.

Combining the above observation with Theorem \ref{thrm:Monotonicity4WeiightedMatrix} 
we  get a natural upper bound for the spectral radius of {\em any} walk-matrix $\Pweight_{\cP,\bxi}$. 

\begin{corollary}\label{cor:MaxElement} 
For integer $k\geq 0$,  let  the set of walks $\cP$ in  $G=(V,E)$ be such  that  every $M\in \cP$ 
is of length $\leq k$.  Let $\bxi_1\in \mathbb{R}^{\WEdges_{\cP}}$
and ${\bf 1}\in  \mathbb{R}^{\WEdges_{\KWalks}}_{\geq 0}$ be  the all-ones vector.

For any $\zeta \geq  \lnorm \bxi_1 \rnorm_{\infty}$ and  $\bxi_{2}= \zeta\times {\bf 1}$,
we have that $\spradius(\Pweight_{\cP,\bxi_1})  \leq  \spradius(\Pweight_{\KWalks,\bxi_2}).$
\end{corollary}
\begin{proof}
As argued before, our assumption that $\cP$ only contains paths of length $\leq k$ imply that $\cP\subseteq \KWalks$.
Then, the corollary follows from Theorem \ref{thrm:Monotonicity4WeiightedMatrix} 
by showing that  $\Pweight_{\KWalks,\bxi_2}$ is symmetric. 

To see  that  $\Pweight_{\KWalks,\bxi_2}$ is symmetric, note the following:  since  all the components of vector $\bxi_{2}$ 
have the same value  $\zeta>0$, from  Lemma \ref{lemma:WalkMatrixVsAdjacency} we have that 
%\begin{align}\nonumber
$\Pweight_{\KWalks,\bxi_2} =\textstyle \sum^k_{\ell=0}\left(\zeta \cdot \simpleadj_G\right)^{\ell}.$
%\end{align}
We conclude that   $\Pweight_{\KWalks,\bxi_2}$ is symmetric since  $\simpleadj^{\ell}_G$ is symmetric for any $\ell\geq 0$.  
The claim follows. 
\end{proof}

We use   Corollary \ref{cor:MaxElement} we prove  Theorem \ref{thrm:Recurrence4InfluenceEigenBound}.  For the  proof 
of Theorem \ref{thrm:Recurrence4InfluenceEigenBound},  see  Section \ref{sec:thrm:Recurrence4InfluenceEigenBound}.

Having $\infmatrix^{\Lambda,\tau}_{G}$ in mind, it is not hard to see that each walk in
$\saw$ is also non-backtracking, and hence 
\begin{align}\label{eq:SAWVsNBN}
\saw\subseteq \nbn. 
\end{align}
For the above, we also use the observation that $\saw$ does not contain paths of 
length larger than $n$.  We use the above observation to get the following corollary.

\begin{corollary}[Non-backtracking]\label{cor:MaxElement4NB}
For the graph $G=(V,E)$, for   $\beta,\gamma, \lambda$ such that  $\gamma>0$,  
$0\leq \beta\leq  \gamma$ and $\lambda>0$ consider  $\mu_G$, the Gibbs distribution on $G$ 
with  parameters $\beta,\gamma, \lambda$.  

 For $\Lambda\subset V$ and  $\tau\in \{\pm 1\}^{\Lambda}$, consider the  influence matrix 
 $\infmatrix^{\Lambda,\tau}_{G}$ induced by $\mu_G$. 
Let ${\bf 1}\in  \mathbb{R}^{\WEdges_{\nbn}}_{\geq 0}$ be  the all-ones vector, while 
consider the vector $\bphi$ defined in \eqref{def:OfRestInfluenceWeights}.

For any $\zeta \geq  \lnorm\bphi \rnorm_{\infty}$ and  $\bxi= \zeta\times {\bf 1}$,
we have that $\spradius(\infmatrix^{\Lambda,\tau}_{G} )  \leq  \spradius(\Pweight_{\nbn,\bxi}).$
\end{corollary}
\begin{proof}
From Lemma \ref{lemma:InfluenceMatrixIsWalkMatrix} we have that $\infmatrix^{\Lambda,\tau}_{G}=\Pweight_{\LamSAW,\bphi}$.
Recall that  $\LamSAW\subseteq \saw$ is the  set of self-avoiding  walks  in $G$ that do 
not use vertices in $\Lambda$.  Since $\LamSAW\subseteq \saw$, from \eqref{eq:SAWVsNBN} we
also have that $\LamSAW\subseteq \nbn$. 

Furthermore, from Lemma \ref{lemma:WNBKSymmetric} in the Appendix we have that  $\Pweight_{\nbn,\bxi}$ is symmetric. 
The corollary follows by making a standard application of  Theorem \ref{thrm:Monotonicity4WeiightedMatrix}, i.e.,  using 
\eqref{eq:thrm:Monotonicity4WeiightedMatrixB}.  The claim follows. 
\end{proof}

We use Corollary \ref{cor:MaxElement4NB} to  prove  Theorem \ref{thrm:Recur4InfSPBoundNBK}.  For a  proof of this 
result see Section \ref{sec:thrm:Recur4InfSPBoundNBK}.

Comparing the  Corollaries \ref{cor:MaxElement} and \ref{cor:MaxElement4NB} one might remark  that the first one
is the most  general, while the second one is more specific to the influence matrix  $\infmatrix^{\Lambda,\tau}_{G}$. 
Combining the two corollaries
we get that
\begin{align}\label{eq:CorFromCor}
\spradius(\infmatrix^{\Lambda,\tau}_{G})  \leq  \spradius(\Pweight_{\nbn,\bpsi_1}) \leq \spradius(\Pweight_{\NWalks,\bpsi_2}),
\end{align}
where both $\bpsi_1$ and $\bpsi_2$ are constant vectors such that each entry, in both vectors, is equal to 
$\max_{e}\{|\bphi(e)|\}$ and  $\bphi(e)$ is defined in \eqref{def:OfRestInfluenceWeights}. 
That is, the first inequality follows from Corollary \ref{cor:MaxElement4NB} and the second one
from Corollary \ref{cor:MaxElement}.

From \eqref{eq:CorFromCor},  one might remark that the bound we get from Corollary \ref{cor:MaxElement4NB} 
for the spectral  radius of $\infmatrix^{\Lambda,\tau}_{G}$ is at least as good as that  we get from Corollary \ref{cor:MaxElement}. 
Furthermore, note that for Theorem \ref{thrm:Recurrence4InfluenceEigenBound} we use the bound
$\spradius(\infmatrix^{\Lambda,\tau}_{G})  \leq   \spradius(\Pweight_{\NWalks,\bpsi_2})$ while for 
 Theorem \ref{thrm:Recur4InfSPBoundNBK} we use the bound 
$\spradius(\infmatrix^{\Lambda,\tau}_{G})  \leq  \spradius(\Pweight_{\nbn,\bpsi_1})$.
In that respect, for the same parameters $\beta,\gamma$ and $\lambda$ of the Gibbs distribution 
the bound on the spectral radius of  $\infmatrix^{\Lambda,\tau}_{G}$ we get  from 
Theorem \ref{thrm:Recur4InfSPBoundNBK} is {\em at most} that we get from 
Theorem \ref{thrm:Recurrence4InfluenceEigenBound}. 

In light of the above, a similar relation holds for the bounds on the mixing time of
Glauber dynamics we get from Theorems \ref{thrm:Ising4SPRadius} and \ref{thrm:Ising4SPRadiusNBK}.
That is, for the same parameter $\beta$ in the zero external field Ising model, the 
bound on the mixing time of Glauber dynamics that we get from Theorem \ref{thrm:Ising4SPRadiusNBK}
is at most the corresponding bound we get from   Theorem \ref{thrm:Ising4SPRadius}.

\section{Spectral Independence Using Matrix Norms} \label{sec:SPComparisonBasedOnNorms}

%As opposed to the previous  section where we  rely on entry based criteria to compare the spectral 
%radii  of two walk matrices,   
In this section we consider the natural approach of bounding the  spectral radius of a walk-matrix using  norms. 
Consider $G=(V,E)$.  For  $K \subseteq V$ and the diagonal, non-singular matrix 
$\bD\in \mathbb{R}_{\geq 0}^{K\times K}$, for  $p\geq 1$  (allowing $p=\infty$)  and $t\geq 0$,  we let the matrix 
norm $\lnorm \cdot  \rnorm_{\bD,t,p}$   be such that for any  $K\times K$ matrix $\bM$ we have 
\begin{align}\label{def:MyNorm}
\lnorm \bM \rnorm_{\bD,t,p} &= \textstyle \lnorm  (\bD^{\circ t})^{-1} \cdot \bM \cdot \bD^{\circ t}\rnorm_{p}, 
\end{align}
where $\bD^{\circ N}$ denotes the $N$-th {\em Hadamard power} of the matrix $\bD$, i.e., $\bD^{\circ N}(w,u)=(\bD(w,u))^N$. 
It is elementary to verify that the norm  $\lnorm \cdot  \rnorm_{\bD,t,p}$ is a well-defined as long as $\bD$ is non-singular.  
Note that the matrix $\bD$ is assumed to be {\em non-negative}, i.e., all its entries are non-negative numbers. 

It is a well-known that the spectral radius of a matrix $\bA$  is the greatest lower bound for the values 
of all matrix norms of $\bA$, e.g., see  
Theorem 5.6.9 in \cite{MatrixAnalysis}.    In that respect, we have the following result.

\begin{corollary}\label{cor:MyNormVsSPRadius}
For any $K \subseteq V$, for any non-singular, diagonal  matrix $\bD\in \mathbb{R}^{K \times K}_{\geq 0}$, 
for  real numbers  $t\geq 0$ and $p\geq 1$, allowing $p=\infty$,  the following is true:
for any  $K \times K$ matrix  $\bM$, we have that
\begin{align}\nonumber 
\spradius(\bM) \leq \lnorm \bM \rnorm_{\bD,t,p}.
\end{align}
\end{corollary}

The idea here is to use  the norm defined in \eqref{def:MyNorm} and Corollary \ref{cor:MyNormVsSPRadius} to bound the spectral
radius of walk-matrices, with particular focus on $\infmatrix^{\Lambda,\tau}_{G}$. For our purposes we consider  $p=\infty$.

Note that the  definition of the $(s,\delta, c)$-potential  function, i.e.,   Definition \ref{def:GoodPotential}, 
is a bit standard and it is tailored  to working  with  ratios of Gibbs marginals, e.g., it includes into its properties   
the function $\logtrecur_d(\cdot)$ from \eqref{eq:DerivOfLogRatio}, etc.  In the following definition we
describe the related notion of ``{\em potential vector}" which  can be applied to general walk-matrices,  
i.e., dealing with general weights on the edges of the walk-trees rather than distributional recursions.
For an  example on how to directly relate the  potential vectors with the  potential functions,  see  the proof  of Theorem
\ref{thrm:Recurrence4InfluenceEigenBoundNonLinear},  in Section \ref{sec:thrm:Recurrence4InfluenceEigenBoundNonLinear}.  

\begin{figure}
\centering
	\centering
		\includegraphics[width=.2\textwidth]{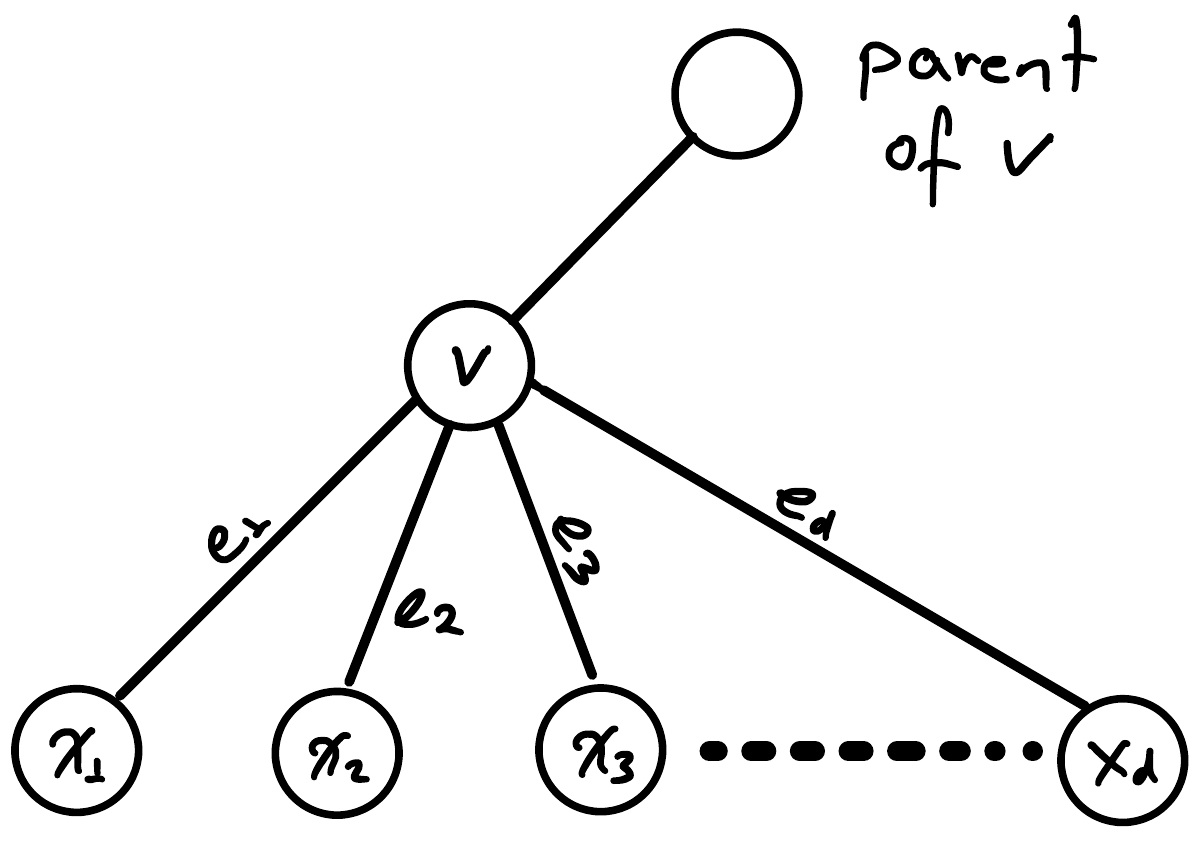}
		\caption{Vertex $v$ in Definition   \ref{def:PotentialWeights}. The vertex $x_i$ is the $i$-th child of $v$. }
	\label{fig:PotenWeight}
\end{figure}

\begin{definition}\label{def:PotentialWeights}
Let $s\geq 1$, allowing $s=\infty$, and $\delta,c>0$. For  a set of walks $\cP$ in $G=(V,E)$ and 
$\bpsi\in \mathbb{R}^{\WEdges_{\cP}}$, 
we say that $\bgamma\in \mathbb{R}^{\WEdges_{\cP}}_{>0}$ is  a $(s,\delta,c)$-potential vector  with 
respect to  $\cP$ and  $\bpsi$  if the following holds:  

For  each  $r\in V_{\cP}$,  for any  {\em non-root} vertex   $v\in \walkT_{\cP}(r)$ 
which has $d>0$ children and  for any ${\bf z}\in \mathbb{R}^d_{\geq 0}$ we have that
\begin{align}\nonumber
\bgamma(e)\cdot \sum^{d}_{i=1} \frac{|\bpsi(e_i)|}{\bgamma(e_i)} \cdot {\bf z}_i & \leq  \delta^{\frac{1}{s}}\cdot \lnorm {\bf z}\rnorm_{s}, 
\end{align}
where $e$ is the edge that connects $v$ to its parent and $e_i$, for $i=1,\ldots,d$, is the edge that 
connects $v$ to its $i$-th child in $\walkT_{\cP}(r)$ (e.g., see Figure \ref{fig:PotenWeight}).
Furthermore, we have that
\begin{align}\nonumber 
\max_{e_a,e_b\in  \WEdges_{\cP}} \left\{ \bgamma(e_a) \cdot \frac{| \bpsi(e_b)|}{\bgamma(e_b)} \right\} &\leq  c.
\end{align}
\end{definition}

%It is possible that for some $r\in V_{\cP}$  the tree  $\walkT_{\cP}(r)$ is empty, or it is of height $1$. 
%According  to Definition \ref{def:PotentialWeights}, the first condition does not apply to such  trees.
%That is,  these trees do not affect  whether $\bgamma\in \mathbb{R}^{\WEdges_{\cP}}_{>0}$ is  a $(s,\delta,c)$-potential 
%vector, or not.

In the standard setting,  potential functions are introduced   in the  recursions of the ratio of Gibbs marginals
by means of the  {\em mean value theorem} of the real analysis.  Here, instead, the potential vector arises naturally 
in our analysis by employing a  simple {\em telescopic trick},  e.g., see the proof of Theorem \ref{thrm:NormRedaux2WalkVector}.

The following theorem is the main result of this part of the paper.

\begin{theorem}\label{thrm:MySpectralMatrixNorm}
For  $\maxDeg \geq 1$, $\spradius_G\geq 1$, for $s\geq 1$,  allowing $s=\infty$,    for $\delta,  c>0$  and  integer $k\geq 0$,  let $G=(V,E)$ 
be a graph of  maximum degree $\maxDeg$, while   $\simpleadj_G$ has  spectral radius $\spradius_G$.

For any  set of walks $\cP$ in $G$   such that the longest walk in $\cP$ is of length $\leq k$ and 
for any  $\bxi\in \mathbb{R}^{\WEdges_{\cP}}$ such that there is 
$\bgamma\in \mathbb{R}^{\WEdges_{\cP}}_{>0}$ which is  $(s,\delta,c)$-potential vector with respect to  $\cP$ and  $\bxi$
the following is true:
 
For the walk-matrix $\Pweight_{\cP,\bxi}$ we have that 
\begin{align}\nonumber  
\lnorm \Pweight_{\cP,\bxi}  \rnorm_{\SpGMatrix_{\cP}, \frac{1}{s}, \infty}  &  \textstyle 
\leq 1+c \cdot (\maxDeg)^{1-\frac{1}{s}} \cdot \left( \spradius_G \right)^{\frac{1}{s}} \cdot  \sum^{k-1}_{\ell=0}   (\delta \cdot \spradius_G)^{\frac{\ell}{s}},
\end{align}
where   $\SpGMatrix_{\cP}$  is the $V_{\cP}\times V_{\cP}$ {\em diagonal} matrix such that 
$\SpGMatrix_{\cP}(u,u)=\maxeigenv(u)$, for  $u\in V_{\cP}$, while $\maxeigenv$ is from  \eqref{eq:DefOfBSMatrix}.   
\end{theorem}

Note that $\SpGMatrix_{\cP}$ is a principal submatrix of the $V\times V$ matrix $\SpGMatrix$  we define in  \eqref{def:AdjMatrixEigenvals}.
Specifically, we obtain $\SpGMatrix_{\cP}$ by removing all the rows and columns of $\SpGMatrix$ that correspond to
 vertices outside $V_{\cP}$.

Theorem  \ref{thrm:Recurrence4InfluenceEigenBoundNonLinear}  follows  immediately from  Theorem 
\ref{thrm:MySpectralMatrixNorm}, Corollary \ref{cor:MyNormVsSPRadius} and Lemma 
\ref{lemma:InfluenceMatrixIsWalkMatrix}.
For the full proof of Theorem \ref{thrm:Recurrence4InfluenceEigenBoundNonLinear}, 
see  Section \ref{sec:thrm:Recurrence4InfluenceEigenBoundNonLinear}.

\subsection{Walk-vectors}\label{sec:WalkVector}
In his section, as well as in the following one, we show how we derive Theorem  \ref{thrm:MySpectralMatrixNorm}.
Consider the norm $\lnorm \cdot  \rnorm_{\bD,t,p}$ from  \eqref{def:MyNorm} and set   $p=\infty$.
Having  $p=\infty$, essential the norm  corresponds to 
the maximum absolute row sum of the matrix $(\bD^{\circ t})^{-1} \cdot \bM \cdot  \bD^{\circ t}$. 

Evaluating the above norm in the context of walk-matrices  gives rise to another topological construction which 
we call   {\em walk-vector}.

\begin{definition}[Walk-Vector]\label{def:WalkVector}
For  $s\geq 1$, allowing  $s=\infty$,   and $\delta, c>0$,  for the set of walks $\cP$ in the graph $G=(V,E)$ and  
the invertible, diagonal matrix  $\bD\in \mathbb{R}^{K \times K}_{\geq 0}$, such that 
$V_{\cP}\subseteq K \subseteq V$,  we define the 
walk-vector $\DBounded=\DBounded(\cP, \bD, s, \delta, c)\in \mathbb{R}^{V_{\cP}}_{\geq 0}$ as follows:

For  $r\in V_{\cP}$, suppose that the root of  $\walkT_{\cP}(r)$ has $d>0$ children, while let $T_i$ be the subtree that
includes the $i$-th child of the root  and its decedents. Then, for the  component $\DBounded(r)$ of the walk-vector 
we have that
\begin{align}\label{eq:WeightNonLinearRelation}
\DBounded(r) & = 1+  \frac{c}{\bD(r,r)}\times \sum^d_{i=1} \sum_{\ell \geq 0} 
\left( \delta^{\ell} \cdot \sum_{w\in V_{\cP}} |\cp_{i,\ell}(w)| \cdot (\bD(w,w))^{s}
\right)^{\frac{1}{s}} >0 ,
\end{align}
where $\cp_{i,\ell}(w)$ is the set of copies of vertex $w$ in the subtree $T_i$ that are 
at distance $\ell$ from the  root of $T_i$.   
\end{definition}

\noindent
Note that $\DBounded$ is a positive vector.  For us here, the parameter $s$ is a bounded number, 
however, since the size of the  walk-trees is assumed to always be finite,  it is easy to see that 
the quantity in the r.h.s. of \eqref{eq:WeightNonLinearRelation} is well-defined even for $s=\infty$.

To get an intuition of what is the walk-vector  $\DBounded$ in Definition \ref{def:WalkVector}, consider the case 
where $\bD$ is the identity matrix, i.e., $\bD=\bI$ and $s=1$. Then, $\DBounded(r)$ is nothing more than
the weighted sum over all paths of the tree $\walkT_{\cP}(r)$ that start from the root, such that each path of length
$\ell>0$ has weight $c\cdot \delta^{\ell-1}$, while the path of length $0$ has weight $1$, that is
\begin{align}
\DBounded(r)=1+ c\cdot {\textstyle \sum_{P}}\  \delta^{|P|-1},
\end{align}
where $P$ varies over all paths of length greater than 0 in $\walkT_{\cP}(r)$ that emanate from the root. 

In the general case, the walk-vectors arise when we are dealing with a walk-matrix $\Pweight_{\cP,\bxi}$
which admits  a $(s,\delta,c)$-potential vector $\bgamma$.  Particularly,  the component $\DBounded(r)$ in 
\eqref{eq:WeightNonLinearRelation}  expresses a bound on the  absolute row sum  at the row $r$
of $(\bD^{\circ s})^{-1}\cdot \Pweight_{\cP,\bxi} \cdot \bD^{\circ s}$. We derive this bound 
by using the potential vector $\bgamma$. 
In that respect, the following result comes  naturally.

\begin{theorem}\label{thrm:NormRedaux2WalkVector}
For  $s\geq 1$,  allowing $s=\infty$, for $\delta,  c>0$, let the set of walks $\cP$ in $G$ 
and  $\bxi\in \mathbb{R}^{\WEdges_{\cP}}$,  while assume  that    $\bgamma\in \mathbb{R}^{\WEdges_{\cP}}_{>0}$ 
is  a $(s,\delta,c)$-potential vector with respect to  $\cP$ and  $\bxi$.  
For any diagonal, non-singular  matrix  $\bD\in\mathbb{R}^{V_{\cP}\times V_{\cP}}_{\geq 0}$, 
% we have the following: For 
the walk-vector $\DBounded=\DBounded(\cP, \bD^{\circ \frac{1}{s}}, s, \delta, c)$ satisfies that
\begin{align}\label{eq:thrm:NormRedaux2WalkVector}
\lnorm \Pweight_{\cP,\bxi}  \rnorm_{\bD, \frac{1}{s}, \infty} \leq  \lnorm \DBounded  \rnorm_{\infty}.
\end{align}
\end{theorem}
The proof of Theorem \ref{thrm:NormRedaux2WalkVector} appears in Section \ref{sec:thrm:NormRedaux2WalkVector}.

In light of Theorem \ref{thrm:NormRedaux2WalkVector},  Theorem \ref{thrm:MySpectralMatrixNorm} follows by
bounding appropriately $\lnorm \DBounded  \rnorm_{\infty}$, for the special case where  $\bD=\bS_{\cP}$, i.e.,
$\bS_{\cP}$ is  the matrix defined in the statement of Theorem \ref{thrm:MySpectralMatrixNorm}.  We study this problem in the 
following  section by  investigating  the properties of walk-vectors.

\subsection{Monotonicity properties for   walk-vectors}

In this section we prove monotonicity results for walk-vectors. These  are similar in spirit to what we had in 
Theorem \ref{thrm:Monotonicity4WeiightedMatrix} for walk-matrices.

\begin{theorem}[Monotonicity for walk-vectors]\label{thrm:MonotoneWVector}
Let  $s\geq 1$,  allowing $s=\infty$,  and $\delta, c \geq 0$.  For any  $\cP$ and $\cS$, sets  walks  in the graph $G=(V,E)$ 
such that $\cS\subseteq \cP$ and for any  invertible, diagonal matrix $\bD\in \mathbb{R}^{V_{\cP}\times V_{\cP}}_{\geq 0}$, 
we consider the walk-vectors $\DBounded_{\cP}=\DBounded_{\cP}(\cP, \bD, s, \delta, c) \in  \mathbb{R}^{V_{\cP}}_{\geq 0}$ and 
$\DBounded_{\cS}=\DBounded_{\cS}(\cS, \bD, s, \delta, c) \in  \mathbb{R}^{V_{\cS}}_{\geq 0} $.

For any $r\in V_{\cS}\subseteq V_{\cP}$ we have that
\begin{align}\label{eq:thrm:MonotoneWVector}
0<\DBounded_{\cS}(r)\leq \DBounded_{\cP}(r).
\end{align}
Furthermore, for any $t\geq 1$, including $t=\infty$, we have that
\begin{align}\label{eq:thrm:MonotoneWVectorBB}
\lnorm \DBounded_{\cS}  \rnorm_t & \leq  \lnorm \DBounded_{\cP} \rnorm_t. % & \textrm{for any $t\geq 1$, including $t=\infty$}.
\end{align}
\end{theorem}

The proof of Theorem \ref{thrm:MonotoneWVector} appears in Section \ref{sec:thrm:MonotoneWVector}.

Note that in the above theorem all the parameters of the two walk-vectors $\DBounded_{\cP}$ and $\DBounded_{\cS}$ are the 
same, apart from the set of the corresponding walks.

At this point, we need to recall   the observation we use for Corollary \ref{cor:MaxElement}. That is,  for  any set of walks  $\cP$ in $G$  
 such that there is no  walk in the set which is of length   greater than $k\geq 0$,  we have that $\cP\subseteq\KWalks$.
Combining this observation   with Theorem \ref{thrm:MonotoneWVector} we get the following result. 

\begin{corollary}\label{cor:MaxElement2}
Let  $k\geq 1$,   let   $s\geq 1$, allowing  $s=\infty$ and $\delta, c \geq 0$. For 
the  set  of walks $\cP$   in the graph $G=(V,E)$ 
such that the length of the longest walk in $\cP$ is at most $k$,    for any  invertible, diagonal matrix 
$\bD\in \mathbb{R}^{V\times V}_{\geq 0}$,  the following holds:

Consider the walk-vectors $\DBounded_{\cP}=\DBounded_{\cP}(\cP, \bD, s, \delta, c) \in  \mathbb{R}^{V_{\cP}}_{\geq 0}$ and 
$\DBounded_{\KWalks}=\DBounded_{\KWalks}(\KWalks,  \bD, s, \delta, c) \in  \mathbb{R}^{V}_{\geq 0}$.
For any $r\in V_{\cP}$ we have that  
\begin{align}\nonumber 
\DBounded_{\cP}(r)\leq \DBounded_{\KWalks}(r).
\end{align}
Furthermore, for any $t\geq 1$, allowing  for $t=\infty$,  we have that $\lnorm \DBounded_{\cP}  \rnorm_t \leq \lnorm \DBounded_{\KWalks}  \rnorm_t .$
\end{corollary}

Theorems \ref{thrm:NormRedaux2WalkVector}, \ref{thrm:MonotoneWVector}  and Corollary \ref{cor:MaxElement2} 
imply that we can bound the quantity $\scriptsize \lnorm \Pweight_{\cP,\bxi}  \rnorm_{\SpGMatrix_{\cP}, \frac{1}{s}, \infty}$, 
in the statement of  Theorem \ref{thrm:MySpectralMatrixNorm}, by using  the $\ell_\infty$ norm of the walk-vector
$\DBounded_{\KWalks}=\DBounded_{\KWalks} \left( \KWalks, \SpGMatrix^{\circ ({1}/{s})}, s, \delta, c\right) $, 
where  $\SpGMatrix$ is the $V\times V$ matrix  defined in \eqref{def:AdjMatrixEigenvals} and the parameters 
$s,\delta, c$ are specified in Theorem \ref{thrm:MySpectralMatrixNorm}. Specifically, we have 
\begin{align}\nonumber  
\lnorm \Pweight_{\cP,\bxi}  \rnorm_{\SpGMatrix_{\cP}, \frac{1}{s}, \infty}  &  \leq  \lnorm \DBounded_{\KWalks} \rnorm_{\infty}. 
\end{align}
Theorem \ref{thrm:MySpectralMatrixNorm} follows by using the above and 
by  bounding appropriately $\lnorm \DBounded_{\KWalks} \rnorm_{\infty}$. For the full proof of Theorem 
\ref{thrm:MySpectralMatrixNorm}, see Section \ref{sec:thrm:MySpectralMatrixNorm}.

\section{Proof of Theorems \ref{thrm:Ising4SPRadius} and \ref{thrm:Ising4SPRadiusNBK}}\label{sec:RapidMixingIsing}

\noindent
For $d>0$,  consider the functions $\logtrecur_d$ and $\dlogtrecur(\cdot)$ defined in \eqref{eq:DefOfH} and \eqref{eq:DerivOfLogRatio}, 
respectively.   Recalling that the zero external field Ising model $\mu_G$ 
corresponds to setting the parameters of $\mu_G$ such that $\beta=\gamma$ and 
$\lambda=1$, we have that
\begin{align}\label{eq:LogRatioIsing}
\logtrecur_d:[-\infty, +\infty]^d\to [-\infty, +\infty] &&\textrm{s.t.}&&  \textstyle (x_1, \ldots, x_d)\mapsto  \sum^d_{i=1}
\log\left( \frac{\beta \exp(x_i)+1}{\exp(x_i)+\beta} \right).
\end{align}
Since  $\frac{\partial }{\partial x_i}\logtrecur_d(x_1, \ldots, x_d)=\dlogtrecur(x_i)$, we have that 
\begin{align}\label{eq:DerivOfLogRatioIsing}
\dlogtrecur(x) &\textstyle =-\frac{(1-\beta^2)  \exp(x)}{(\beta \exp(x)+1)(\exp(x)+\beta)}.
\end{align}

\begin{lemma}\label{lemma:IsingInfNormBound}
For any $d>0$,  $\zeta \in (0,1)$,   $R\geq 1$ and   $\beta\in \UnIsing(R,\zeta )$ we have the following: 
Consider the functions $\logtrecur_d$  defined in \eqref{eq:DefOfH} with respect to the Ising model with parameter $\beta$
and no external field. We have that 
\begin{align}\label{eq:lemma:IsingInfNormBound}
||\nabla \logtrecur_d ({y}_1, {y}_2, \ldots, {y}_d) ||_{\infty} &\leq {(1-\zeta)}/{R}.
\end{align}
\end{lemma}
\begin{proof}
It suffices to show that  any  $d>0$ and any 
$({y}_1, {y}_2, \ldots, {y}_d)\in [-\infty, +\infty]^d$  we have that 
\begin{align}\label{eq:Target4thrm:SI4Ising}
||\nabla \logtrecur_d ({y}_1, {y}_2, \ldots, {y}_d) ||_{\infty}
&\leq \frac{|\beta-1|}{\beta +1}.
\end{align}
Before showing that \eqref{eq:Target4thrm:SI4Ising} is true, let us show how it implies 
\eqref{eq:lemma:IsingInfNormBound}.  That is, we show that for any $\beta \in \UnIsing(R,\zeta)$,
we have that $\frac{|\beta-1|}{\beta +1}\leq \frac{1-\zeta}{R}$.

Consider the function $f(x)=\frac{|x-1|}{x+1}$ defined on the closed interval 
$\left[\frac{R-1}{R+1}, \frac{R+1}{R-1} \right]$. 
Taking derivatives,  it is elementary  to verify that  $f(x)$ is increasing in the interval 
$1<x  \leq \frac{R+1}{R-1}$, while it is decreasing
in the interval $\frac{R-1}{R+1}\leq x<1$. 
Furthermore, noting that  $f(1)=0$, it is direct that  
\begin{align}\nonumber
\sup_{\beta \in \UnIsing(R,\zeta)}f(\beta)=\textstyle f\left(\frac{R-1+\zeta}{R+1-\zeta}\right)= \textstyle f\left(\frac{R+1-\zeta}{R-1+\zeta}\right)
=\frac{1-\zeta}{R}.
\end{align}
It is immediate that  indeed \eqref{eq:Target4thrm:SI4Ising} implies \eqref{eq:lemma:IsingInfNormBound}.
Hence,  it remains to show that \eqref{eq:Target4thrm:SI4Ising} is true.

Since we have that  $\frac{\partial }{\partial x_i}\logtrecur_d(x_1, x_2, \ldots, x_d)=\dlogtrecur(x_i)$, 
it suffices to show that for any $x\in [-\infty, +\infty]$ we have that
\begin{align}\label{eq:TargetB4thrm:SI4Ising}
\left |  \dlogtrecur(x) \right | &\leq \textstyle \frac{|1-\beta | }{1+\beta}.
\end{align}
For the distribution we consider here, the function $\dlogtrecur()$ is given from \eqref{eq:DerivOfLogRatioIsing}.
From the above we get that
\begin{align}
|\dlogtrecur(x)| & \textstyle =   \frac{|1-\beta^2|\exp(x)}{(b \exp(x) +1)(b+ \exp(x))} \ =\ 
\frac{|1-\beta^2| }{(\beta+\exp(-x) )(\beta+ \exp(x) )}
%\nonumber \\ & = 
\ =\   \frac{|1-\beta^2 |}{\beta^2+1 +\beta (\exp(-x)+\exp(x) ) }. \nonumber
\end{align}
It is straightforward to verify that  $\phi(x)=e^{-x}+e^{x}$ is convex and for any  $x\in [-\infty, +\infty]$
the function $\phi(x)$ attains its minimum at $x=0$, i.e., we have that $\phi(x)\geq 2$.  Consequently,   we get that
\begin{align} \nonumber 
|\dlogtrecur(x)|  & \leq   \textstyle  \frac{|1-\beta^2 |}{\beta^2+1 +2\beta } \ = \ 
\frac{|1-\beta^2 |}{(\beta+1)^2 } \ =\ \frac{|1-\beta | }{1+\beta},
\end{align}
for any $x\in [-\infty, +\infty]$.
The above proves that \eqref{eq:TargetB4thrm:SI4Ising} is true and  concludes our proof.
\end{proof}

\subsection{Proof of Theorem \ref{thrm:Ising4SPRadius}}\label{sec:thrm:Ising4SPRadius}

Note that if $\spradius_G$, the spectral radius of $\simpleadj_G$, is bounded, then the same holds for 
the maximum degree $\maxDeg$ of $G$.  This follows  from  the standard relation  that 
$\maxDeg\leq (\spradius_G)^2\leq \maxDeg^2$.   Similarly, if $\spradius_G$  is unbounded then 
$\maxDeg$  is unbounded, too.  

Using Lemma \ref{lemma:IsingInfNormBound} we get the following: for any $\beta\in \UnIsing(\spradius_G,\delta)$, 
the set of functions  $\{ \logtrecur_d\}_{d\in [\maxDeg]}$ defined in \eqref{eq:LogRatioIsing} with respect to the
zero external field Ising model with parameter $\beta$ exhibits  $ ({1-\delta})/{\spradius_G}$-contraction, i.e.,
\begin{align}\nonumber 
||\nabla \logtrecur_d ({y}_1, {y}_2, \ldots, {y}_d) ||_{\infty}  &\leq ({1-\delta})/{\spradius_G}  & \forall  d\in [\maxDeg].
\end{align}
The above, combined with Theorem  \ref{thrm:Recurrence4InfluenceEigenBound} imply that 
for   $\Lambda\subseteq V$ and $\tau\in \{\pm 1\}^{\Lambda}$,  the pairwise influence 
matrix $\infmatrix^{\Lambda,\tau}_{G}$,  induced by $\mu_G$,  satisfies that  
\begin{align}\label{eq:thrm:Ising4SPRadiusBasis}
\spradius(\infmatrix^{\Lambda,\tau}_{G})\leq \delta^{-1}.
\end{align}
%Part $1$ in Theorem \ref{thrm:Ising4SPRadius}, i.e., the case where  $\spradius_G$ is bounded,
The theorem follows as a corollary from \eqref{eq:thrm:Ising4SPRadiusBasis} and  Theorem 1.9 in  \cite{VigodaSpectralIndB}.
\hfill $\Box$

\subsection{Proof of Theorem \ref{thrm:Ising4SPRadiusNBK}}\label{sec:thrm:Ising4SPRadiusNBK}

Using Lemma \ref{lemma:IsingInfNormBound} we get the following: for any $\beta\in \UnIsing(\hspradius_G,\delta)$, 
the set of functions  $\{ \logtrecur_d\}_{d\in [\maxDeg]}$ defined in \eqref{eq:LogRatioIsing} with respect to the
zero external field Ising model with parameter $\beta$ exhibits  $ ({1-\delta})/{\hspradius_G}$-contraction, i.e.,
\begin{align}\nonumber 
||\nabla \logtrecur_d ({y}_1, \ldots, {y}_d) ||_{\infty}\leq {(1-\delta)}/{\hspradius_G}.
\end{align}
The above, combined with Theorem  \ref{thrm:Recur4InfSPBoundNBK} imply that 
for   $\Lambda\subseteq V$ and $\tau\in \{\pm 1\}^{\Lambda}$,  the pairwise influence 
matrix $\infmatrix^{\Lambda,\tau}_{G}$,  induced by $\mu_G$,  satisfies that  
\begin{align}\nonumber 
\spradius(\infmatrix^{\Lambda,\tau}_{G}) &\leq \textstyle  \lnorm  \left( \bI -   \frac{1-\delta}{\hspradius_G} \simpleadj_G +\left( \frac{1-\delta}{\hspradius_G}\right)^2(\bD-\bI) \right)^{-1}   \rnorm_2.
\end{align}
% The above implies that, for  bounded $\maxDeg$ end $\lnorm (\bI-\delta  \NBMatrix_G)^{-1} \rnorm_2$, we also 
% have that $\spradius(\infmatrix^{\Lambda,\tau}_{G})$ is bounded.  
The theorem follows by making a standard use of and  Theorem 1.9   in  \cite{VigodaSpectralIndB}.
\hfill $\Box$

\newcommand{\HCFactor}{{\Xi}}

\section{Proof of Theorem \ref{thrm:HC4SPRadius}}\label{sec:thrm:HC4SPRadius}

We start our proof by deriving a bound on $\spradius  \left( \infmatrix^{\Lambda,\tau}_{G} \right)$,
the spectral radius of the influence matrix.

\begin{theorem}\label{thrm:SPIndependenceHC}
For any $\epsilon\in (0,1)$,  for $\maxDeg\geq 2$ and $\spradius_G\geq 2$,   let   $G=(V,E)$  be of maximum 
degree $\maxDeg$, while $\simpleadj_G$ has  spectral radius $\spradius_{G}$.  
Also,  let $\mu_G$ be the Hard-core model on $G$, with  fugacity $0< \lambda\leq (1-\epsilon)\lcritical(\spradius_G)$.

There is a constant  $0<z<1$, that depends only on $\epsilon$, such that 
for any  $\Lambda\subseteq V$ and $\tau\in \{\pm 1\}^{\Lambda}$,  the pairwise influence 
matrix  $\infmatrix^{\Lambda,\tau}_{G}$, induced by $\mu_{G}$,   satisfies that  
\begin{align} \nonumber %\label{eq:thrm:Recurrence4InfluenceEigenBoundANonLin}
\spradius  \left( \infmatrix^{\Lambda,\tau}_{G} \right) &\leq 
1+  e^3   \left( {\maxDeg}/{\spradius_G} \right)^{1/2}  z^{-1}.
\end{align}
\end{theorem}
The proof of Theorem \ref{thrm:SPIndependenceHC} appears in Section \ref{sec:thrm:SPIndependenceHC} .

Note that if $\spradius_G$ is unbounded, then there are no guarantees 
from Theorem \ref{thrm:SPIndependenceHC}  that  $\textstyle \spradius  \left( \infmatrix^{\Lambda,\tau}_{G} \right)$ 
is bounded. This is due to  the quantity  ${\maxDeg}/{\spradius_G}$, on the r.h.s. of 
\eqref{eq:thrm:Recurrence4InfluenceEigenBoundANonLin}, which can  be unbounded if $\spradius_G$ 
is unbounded.

\begin{proof}[Proof of Theorem \ref{thrm:HC4SPRadius}]

%In light of Theorem \ref{thrm:SPIndependenceHC} the proof of Theorem \ref{thrm:HC4SPRadius} is somewhat standard. 
%

As argued in the proof of Theorem \ref{thrm:Ising4SPRadius},  if the spectral radius $\spradius_G$ is bounded, 
then the  same holds for the maximum degree $\maxDeg$ as we always have that %. This follows  from  the standard relation  that 
$\maxDeg\leq (\spradius_G)^2\leq \maxDeg^2$. 
%Since we assume that $\spradius_G$ is bounded, we have that
%$\maxDeg$ is also bounded. 

Since both $\spradius_G$ and $\maxDeg$ are bounded,  for  fugacity $0\leq \lambda\leq (1-\epsilon)\lcritical(\spradius_G)$,   
Theorem \ref{thrm:SPIndependenceHC} implies that  $\spradius  \left( \infmatrix^{\Lambda,\tau}_{G} \right)=O(1)$
for any $\Lambda\subseteq V$ and $\tau\in \{\pm 1\}^{\Lambda}$.

Then,  Theorem \ref{thrm:HC4SPRadius} follows by a standard  application of Theorem 1.9 in  \cite{VigodaSpectralIndB}.
\end{proof}

\newcommand{\oldpotF}{\Phi}
\subsection{Proof of Theorem \ref{thrm:SPIndependenceHC}} \label{sec:thrm:SPIndependenceHC}

In order to prove Theorem \ref{thrm:SPIndependenceHC} we introduce the  potential function
$\potF$ as follows: we define  $\potF$ indirectly, i.e., in terms of $\xdpotF=\potF'$.   We have that
\begin{align}\label{def:XPotentialFunction}
\xdpotF &: \mathbb{R}_{>0} \to \mathbb{R}  & \textrm{such that} &&  y&  \textstyle \mapsto  \sqrt{\frac{e^y}{1+e^y}},
\end{align}
while $\potF(0)=0$.   

Note that the  potential function $\potF$ was  proposed -in a more general form- in \cite{VigodaSpectralInd}. 
 It is standard to show that $\potF$ is   well-defined, e.g.,  see  \cite{VigodaSpectralInd}. 
Later in our analysis we need to use certain results from \cite{ConnectiveConst}, which (essentially) use 
another, but related   potential function from  \cite{LLY13}.   
We postpone this discussion  until later.

For any given $\lambda>0$,  we define, implicitly, the function $\dcritical(\lambda)$ to be the positive 
number $z>1$ such  that $\frac{z^{z}}{(z-1)^{(z+1)}}=\lambda$. 
Form its definition it is not hard to see that  $\dcritical(\cdot)$ is the inverse map of $\lcritical(\cdot )$, i.e., we have  that 
$\dcritical(x)=\lcritical^{-1}(x)$.  In that respect,  $\dcritical(x)$ is well-defined as  $\lcritical(x)$ is monotonically 
decreasing in $x$.

\begin{theorem}\label{thrm:GoodPotentialHC}
For $\lambda>0$,   let $\dcritical=\dcritical(\lambda)$.  The potential function  $\potF$ defined in \eqref{def:XPotentialFunction} is a 
$(s_0, \delta_0, c_0)$-potential  function (as in Definition \ref{def:GoodPotential}) such that
\begin{align}
s_0^{-1}&= \textstyle 1-\frac{\dcritical-1}{2}\log\left(1+\frac{1}{\dcritical-1} \right), &\delta_0& \textstyle \leq \frac{1}{\dcritical}  &\textrm{and} && c_0 \leq \textstyle \frac{\lambda}{1+\lambda}.
\end{align}
\end{theorem}
The proof of Theorem \ref{thrm:GoodPotentialHC} appears in Section \ref{sec:thrm:GoodPotentialHC}.

\begin{claim}\label{claim:InterpretGoodPotentialHC}
For  $\epsilon\in (0,1)$,   $R\geq 2$ and  $0<\lambda<(1-\epsilon)\lcritical(R)$ the following is true: 
There is $0<z<1$ , which only depend on $\epsilon$,  such that for $\dcritical=\dcritical(\lambda)$, we have 
\begin{align}\label{eq:claim:InterpretGoodPotentialHC}
\textstyle \frac{1-z}{R}  & \geq \frac{1}{\dcritical} &\textrm{and} && \frac{\lambda}{1+\lambda} &<\frac{e^3}{R}. 
\end{align}
\end{claim}
\begin{proof}
It elementary to verify that $\dcritical(z)$ is decreasing in $z$.  This implies that for $\lambda\leq (1-\epsilon)\lcritical(R)$,  
$\dcritical(\lambda)\geq  \dcritical(\lcritical(R))=R$. Particularly, this implies that there is $0<z<1$, which only depends on  
$\epsilon$  such that $\dcritical(\lambda) \geq \frac{R}{(1-z)}$.
This proves the leftmost inequality in  \eqref{eq:claim:InterpretGoodPotentialHC}.

As far as the rightmost inequality is concerned, we have that 
\begin{align}\label{eq:Base4Righteq:claim:InterpretGoodPotentialHC}
\textstyle \frac{\lambda}{1+\lambda} \leq \lambda  < \lcritical(R).
\end{align}
The first inequality follows since $\lambda>0$, while the second follows since $\lambda<\lcritical(R)$.
 From the definition of $\lcritical(\cdot)$, we have that 
\begin{align}
\lcritical(R)&\textstyle =\frac{R^{R}}{(R-1)^{(R+1)}}\ = \  
\frac{1}{R}\left( 1-R^{-1}\right)^{-(R+1)} \ =\  \frac{1}{R}\left( 1+\frac{1}{R-1}\right)^{R+1} \  \leq \  \frac{1}{R}\exp\left(\frac{R+1}{R-1}\right) 
\  \leq \ e^3/R. 
\end{align}
For the one before the last  inequality we use that $1+x\leq e^{x}$. For the last inequality we note that $\frac{R+1}{R-1}$ is decreasing in 
$R$, hence,  for $R\geq 2$, we have that $\frac{R+1}{R-1}\leq 3$. 
Plugging the above  bound into \eqref{eq:Base4Righteq:claim:InterpretGoodPotentialHC}, gives the rightmost inequality in
\ref{eq:claim:InterpretGoodPotentialHC}. 
The claim follows. 
\end{proof}

Furthermore, using the standard inequality that $\log(1+y)\leq y$, for the quantity $s_0$ in 
Theorem \ref{thrm:GoodPotentialHC},  we have that $1\leq s_0\leq 2$.

Combining this observation with Theorem \ref{thrm:GoodPotentialHC} and Claim \ref{claim:InterpretGoodPotentialHC} we get the following corollary.

\begin{corollary}\label{cor:GoodPotentialHCRange}
For  $0<\epsilon<1$, for $\maxDeg\geq 2$ and $\spradius_G\geq 2$, let $0<\lambda<(1-\epsilon) \lcritical(\spradius_G)$.
Let the graph $G=(V,E)$ be of maximum degree $\maxDeg$, whose adjacency matrix has spectral radius $\spradius_G$. 
Let, also the  Hard-core model $\mu_G$ on $G$,  with fugacity $\lambda$. There is a constant $0<z<1$, which 
depends on $\epsilon$, such that the following is true:  

The potential function $\potF$ defined in \eqref{def:XPotentialFunction} is $(q, \gamma, g)$-potential with respect to $\mu_G$
 where 
\begin{align}\nonumber
q &\in [1,2], &\gamma & \leq {(1-z)}/{\spradius_G}  &\textrm{and} && g\leq {e^3}/{\spradius_G}.
\end{align}
\end{corollary}

\noindent
Theorem \ref{thrm:SPIndependenceHC} follows by combining Corollary \ref{cor:GoodPotentialHCRange} and  Theorem \ref{thrm:Recurrence4InfluenceEigenBoundNonLinear}.
\hfill $\Box$

\subsection{Proof of Theorem \ref{thrm:GoodPotentialHC}}\label{sec:thrm:GoodPotentialHC}
Recall  that  the ratio of Gibbs marginals $\gratio^{\Lambda, \tau}_x$, defined in Section 
\ref{sec:RecursionVsSpectralIneq},  is possible to be equal to  zero, or $\infty$.  Typically, this 
happens if the vertex $x$ with respect to which we consider the ratio is a part of  the boundary 
set $\Lambda$, or has a neighbour in $\Lambda$.  When we are dealing with the Hard-core model, 
there is a standard way to avoid these infinities  and zeros in our calculations.

Suppose that we have  the Hard-core model with fugacity $\lambda>0$ on a tree $T$, while at the 
set of vertices $\Lambda$ we have a boundary condition $\tau$. Then, it is elementary to verify that 
this distribution is identical to the Hard-core model with the same fugacity on the tree (or forest) $T'$ 
which is obtained from $T$ by working  as follows: we remove from $T$ every vertex $w$ which
 either belongs to  $\Lambda$, or   has  a neighbour  $u\in \Lambda$ such  that $\tau(u)=$``occupied".

Perhaps it is useful to  write down the functions that arise from the  recursions
in Section \ref{sec:RecursionVsSpectralIneq}, for 
the Hard-core model with fugacity  $\lambda$.  Recall that, in this case, we have $\beta=0$ and $\gamma=1$. 
In the following definitions, we take into consideration that boundary conditions have been removed as 
described above. 

For   integer $d\geq 1$,  we have that
\begin{align}\label{eq:HC-BPRecursion}
\trecur_d:\mathbb{R}^d_{>0}\to (0, \lambda) & &\textrm{such that}& &
\textstyle (x_1, \ldots, x_d)\mapsto \lambda \prod^d_{i=1}\frac{1}{{x}_i+1}.
\end{align}
We also define $\trecur_{d,{\rm sym}}:\mathbb{R}^d_{>0}\to (0, \lambda)$  the {\em symmetric} version of the above function, that is
\begin{align}\label{def:SymTreeHC}
x&\mapsto \trecur_d(x,x,\ldots, x).
\end{align}
Recall, also, that   $\logtrecur_d=\log \circ F_d \circ \exp$. For the  Hard-core model with fugacity $\lambda$, we have that
\begin{align}\label{eq:HC-DefOfH}
\logtrecur_d:\mathbb{R}^d\to \mathbb{R} &&\textrm{s.t.}&&  \textstyle (x_1, \ldots, x_d)\mapsto \log \lambda+\sum^d_{i=1}
\log\left( \frac{1}{\exp(x_i)+1} \right).
\end{align}
For $\dlogtrecur(\cdot)$ such that  $\frac{\partial }{\partial x_i}\logtrecur_d(x_1, \ldots, x_d)=\dlogtrecur(x_i)$, 
we have
\begin{align}\label{def:GradientRecHC}
\dlogtrecur:\mathbb{R} \to \mathbb{R}& &\textrm{such that}& &
\textstyle x\mapsto -\frac{e^{x_i}}{e^{x_i}+1}.
\end{align}
Finally,  the set of log-ratios $\ratiorange$, defined in \eqref{eq:DefOfRatiorange}, 
satisfies that 
\begin{align}\label{eq:RangeJ4HC}
 \ratiorange=(-\infty, \log(\lambda)).
\end{align}  
Note also, that the image of $\potF$, i.e., the set $S_{\potF}$, satisfies that  $S_{\potF}=(-\infty, \infty)$.

\newcommand{\easyspace}{Q_{\potF}}
\newcommand{\easyspaceX}{L_{\potF}}

With all the above, we proceed to prove the theorem. We need to show that $\potF$ satisfies the contraction
and the boundedness conditions, for appropriate parameters. 

We start with the contraction.   For any integer  $d>0$, we let 
$\cE_{d}:\mathbb{R}^d\times \mathbb{R}^d \to \mathbb{R}$ be such that 
for ${\bf m}=({\bf m}_1,  \ldots, {\bf m}_d)\in \mathbb{R}^{d}_{\geq 0}$, and 
${\bf y}=({\bf y}_1,  \ldots, {\bf y}_d)\in \mathbb{R}^d$ 
we have that 
\begin{align}
\cE_{d}({\bf m}, {\bf y})={\textstyle \xdpotF\left(\logtrecur_{d}({\bf y}) \right) }\sum^{d}_{j=1}
\frac{\textstyle \left| \dlogtrecur\left( {\bf  y}_j \right)\right|}{\textstyle \xdpotF\left(   {\bf y}_j \right)} 
\times  {\bf m}_j. \nonumber
\end{align}

\begin{proposition}[contraction]\label{prop:PotContracts}
For $\lambda>0$, let $\dcritical=\dcritical(\lambda)$.  Let  $q>0$ be such that 
\begin{align}
q^{-1}&= \textstyle 1-\frac{\dcritical-1}{2}\log\left(1+\frac{1}{\dcritical-1} \right).
\end{align}
For  $d>0$, for  ${\bf m}\in \mathbb{R}^{d}_{\geq 0}$ we have that
\begin{align}
\sup_{{\bf y}\in (\easyspace)^d }\{ \cE_{d}({\bf m}, {\bf y}) \} \leq  \dcritical^{-\frac{1}{q}} \cdot \lnorm {\bf m} \rnorm_{q},
\end{align}
where $\easyspace \subseteq \mathbb{R}$  contains  every $y \in \mathbb{R}$ such that  there is 
$\tilde{y} \in S_{\potF} $ for which we have ${y}=\potF^{-1}(\tilde{y})$. 
\end{proposition}

The proof of Proposition \ref{prop:PotContracts} appears in Section \ref{sec:prop:PotContracts}.

Note that  Proposition \ref{prop:PotContracts} implies that $\potF$ satisfies the contraction condition with the parameter 
we need in order to prove our theorem. We now focus on establishing the boundedness property of $\potF$.

\begin{lemma}[boundedness]\label{lemma:PotBooundedness}
For $\lambda>0$,  we have that $\textstyle \max_{y_1,y_2\in \ratiorange }\left\{ \xdpotF(y_2) \cdot \frac{|\dlogtrecur(y_1)|}{\xdpotF(y_1)} \right\} 
\leq \frac{\lambda}{1+\lambda}$.
\end{lemma}
\begin{proof}
Using the definitions of the functions $\xdpotF$ and $\dlogtrecur$ from \eqref{def:PotentialFunction} and 
\eqref{def:GradientRecHC}, respectively, we have that
\begin{align}
 \max_{y_1,y_2\in \ratiorange }\left\{ \xdpotF(y_2) \cdot \frac{|\dlogtrecur(y_1)|}{\xdpotF(y_1)} \right\} &=
 \max_{y_1,y_2\in \ratiorange} \left\{\sqrt{\dlogtrecur(y_1)\dlogtrecur(y_2) }  \right\}  %\nonumber\\
 \ =\  \max_{y_1,y_2\in \ratiorange} \left\{\sqrt{ \frac{e^{y_1}}{1+e^{y_1}}\frac{e^{y_2}}{1+e^{y_2}} }  \right\} % \nonumber\\
\ =\ {  \frac{\lambda}{1+\lambda}}.  \nonumber%\label{eq:Bound4TopConstRecur}
\end{align}
The last inequality follows from the observation that the function $g(x)=\frac{e^x}{1+e^x}$ is increasing in $x$, 
while, from \eqref{eq:RangeJ4HC}, we have that   $e^{y_1},e^{y_2}\leq \lambda$. 
The claim follows. 
\end{proof}

The  above lemma implies the contraction we need in order to prove our result. 
In light of Proposition \ref{prop:PotContracts} and Lemma \ref{lemma:PotBooundedness}, Theorem \ref{thrm:GoodPotentialHC} follows. 
 \hfill $\Box$

\subsection{Proof of Proposition \ref{prop:PotContracts}}\label{sec:prop:PotContracts}

The proposition follows by using results from \cite{ConnectiveConst}. However, in order to apply these results, we need
to bring $\cE_{d}({\bf m}, {\bf y})$ into an  appropriate form. 

For any $d>0$, we let $\cJ_{d}:\mathbb{R}^d_{\geq 0}\times \mathbb{R}^d_{\geq 0}\to \mathbb{R}$ be such that 
for ${\bf m}=({\bf m}_1,  \ldots, {\bf m}_d)\in \mathbb{R}^{d}_{\geq 0}$ and 
${\bf z}=({\bf z}_1,  \ldots, {\bf z}_d)\in \mathbb{R}^{d}_{\geq 0}$ we  have
\begin{align}
\cJ_{d}({\bf m}, {\bf z}) &= \textstyle \xdpotF\left(\log \trecur_{d}({\bf z}) \right)  \sum^{d}_{j=1}
\frac{\textstyle \left| \dlogtrecur\left(\log {\bf  z}_j \right)\right|}{\textstyle \xdpotF\left( \log  {\bf z}_j \right)} 
\times  {\bf m}_j. \nonumber
\end{align}
Using the definitions in \eqref{eq:HC-BPRecursion} and \eqref{eq:HC-DefOfH}, it is elementary to verify that for any  $d>0$, for any  
${\bf m}\in \mathbb{R}^d_{\geq 0}$,   ${\bf z}\in \mathbb{R}^d_{>0}$ and ${\bf y}\in \mathbb{R}^{d}$ such that ${\bf z}_j=e^{{\bf y}_j}$,
we have that
\begin{align}\nonumber
\cJ_{d}({\bf m}, {\bf z})= \cE_{d}({\bf m}, {\bf y}).
\end{align}
In light of the above, the proposition follows by showing that 
\begin{align}\label{prop:Target4PotContractsA}
{\textstyle \sup_{{\bf z}\in  \mathbb{R}^d_{>0}}} \{ \cJ_{d}({\bf m}, {\bf z}) \} \leq  \dcritical^{-{1}/{s}} \cdot \lnorm {\bf m} \rnorm_{s}.
\end{align}
In order to prove \eqref{prop:Target4PotContractsA}, we let 
\begin{align}\label{def:PotentialFunction}
\dpotF &: \mathbb{R}_{>0}\to \mathbb{R}  & \textrm{such that} &&  y& \textstyle  \mapsto \frac12\sqrt{\frac{1}{y(1+y)}}.
\end{align}

\begin{claim}\label{claim:OnePotVsOtherPot}
For any ${\bf m}=({\bf m}_1,  \ldots, {\bf m}_d)\in \mathbb{R}^{d}_{ \geq 0}$ and  
${\bf z}=({\bf z}_1,  \ldots, {\bf z}_d)\in \mathbb{R}^{d}_{ > 0}$
we have that
\begin{align}\label{eq:SubCondVsPotentials}
\cJ_{d}({\bf m}, {\bf z})&\textstyle = 
\dpotF (\trecur_{d}({\bf z})) \times \sum^{d}_{i=1} \frac{{\bf m}_i }{\dpotF({\bf z}_i)} \left .  
\left| \frac{\partial }{\partial {\bf t}_i}\trecur_{d}({\bf t}) \right|_{{\bf t} ={\bf z}}\right|,   
\end{align}
where   $\trecur_{d}$ and $\dpotF$  are defined
in \eqref{eq:HC-BPRecursion} and \eqref{def:PotentialFunction}, respectively. 
\end{claim}

\begin{proof}
The claim follows by using simple rearrangements.  We have that
\begin{align}%\label{eq:RecRelSubContr}
\cJ_{d}({\bf m}, {\bf  z}) &= 
\textstyle \xdpotF\left(\log \trecur_{d}({\bf z}) \right)  \sum^{d}_{j=1}
\frac{\textstyle \left| \dlogtrecur\left(\log {\bf  z}_j \right)\right|}{\textstyle \xdpotF\left( \log  {\bf z}_j \right)} 
\times  {\bf m}_j \nonumber\\
&= \textstyle \sqrt{\frac{\trecur_{d}({\bf z})}{1+\trecur_{d}(\bf z)}}
\sum^{d}_{j=1} \sqrt{\frac{{\bf  z}_j}{1+{\bf  z}_j}}\times {\bf m}_j  \label{eq:AppOfFX}\\
&= \textstyle \sqrt{\frac{1}{\trecur_{d}({\bf z})(1+\trecur_{d}(\bf z))}}
\sum^{d}_{j=1} \sqrt{{\bf  z}_j(1+{\bf  z}_j) }\times\frac{ \trecur_{d}(\bf z)}{1+{\bf  z}_j}  \times {\bf m}_j.   \nonumber
\end{align}
In \eqref{eq:AppOfFX}, we  substitute  $\xdpotF$ and $\dlogtrecur$ according to 
\eqref{def:XPotentialFunction} and \eqref{def:GradientRecHC}, respectively.
Using the definition of $\dpotF $ from \eqref{def:PotentialFunction}, we get that 
\begin{align}
\cJ_{d}({\bf m}, {\bf z})  &=  \textstyle \dpotF(\trecur_{d}({\bf z})) 
\sum^{d}_{j=1}\frac{1}{ \dpotF({\bf z}_j)}  \times  \frac{ \trecur_{d}(\bf z)}{1+{\bf z}_i} \times {\bf m}_j. \nonumber
\end{align}
The above implies  \eqref{eq:SubCondVsPotentials},  note that  
$\left|  \frac{\partial }{\partial {\bf t}_i}\trecur_{d}({\bf t}) \right|=\frac{\trecur_{d}({\bf t})}{1+{\bf t}_i}$, 
for any $i\in [d]$.  
The claim follows.
\end{proof}

In light of  Claim \ref{claim:OnePotVsOtherPot}, \eqref{prop:Target4PotContractsA} follows by showing that
for any ${\bf m}=({\bf m}_1,  \ldots, {\bf m}_d)\in \mathbb{R}^{d}_{ \geq 0}$ and  
${\bf z}=({\bf z}_1,  \ldots, {\bf z}_d)\in \mathbb{R}^{d}_{>0}$  we have that
\begin{align}\label{prop:Target4PotContractsB}
\dpotF (\trecur_{d}({\bf z})) \times \sum^{d}_{i=1} \frac{{\bf m}_i }{\dpotF({\bf z}_i)} \left .  
\left| \frac{\partial }{\partial {\bf t}_i}\trecur_{d}({\bf t}) \right|_{{\bf t} ={\bf z}}\right|
& \leq  \dcritical^{-\frac{1}{s}} \cdot \lnorm {\bf m} \rnorm_{s}. 
\end{align}

\noindent 
The above follows by using standard results form \cite{ConnectiveConst}.  For any $s\geq 1, d>0$ and $x\geq 0$,  we let the function
\begin{align} \nonumber 
\HCFactor(s, d, x)=\frac{1}{d}\left( \frac{\dpotF(\trecur_{d,{\rm sym}}(x))}{\dpotF(x)} \trecur'_{d,{\rm sym}}(x) \right)^s,
\end{align}
where the  functions $\trecur_{d,{\rm sym}}$, $\dpotF$ are defined in \eqref{def:SymTreeHC} and \eqref{def:PotentialFunction}, 
respectively,  while  $\trecur'_{d,{\rm sym}}(x)=\frac{d}{dx}\trecur_{d,{\rm sym}}(x)$.

\begin{lemma}[\cite{ConnectiveConst}]\label{lemma:InsteadOfHolder}
For any  $\lambda>0$, for integer $d\geq 1$, for  $s\geq 1$,  for ${\bf x}\in \mathbb{R}^d_{>0}$ 
and  ${\bf m}\in \mathbb{R}^d_{\geq 0}$,    the following holds: 
there exists $\bar{x}>0$  and  integer  $0\leq k\leq d$ such that 
\begin{align}\nonumber
\dpotF (\trecur_d({\bf x})) \times \sum^d_{i=1} \frac{{\bf m}_i }{\dpotF({\bf x}_i)} 
\left| \left .  \frac{\partial }{\partial z_i}\trecur_d({\bf z}) \right|_{{\bf z} ={\bf x}}  \right|
& \leq  \left( \HCFactor(s, k, \bar{x}) \right)^{1/s} \times || {\bf m}||_{s},
\end{align}
where ${\bf x}=({\bf x}_1,  \ldots, {\bf x}_d)$  and ${\bf m}=({\bf m}_1,  \ldots, {\bf m}_d)$. 
\end{lemma}

In light of the above lemma, our proposition follows as a corollary from the following result.

\begin{lemma}[\cite{ConnectiveConst}] \label{lemma:XiBoundAtFixPoint}
For $\lambda>0$, consider  $\dcritical=\dcritical(\lambda)$ and $\trecur_{\dcritical, {\rm sym}}$ with fugacity $\lambda$.
Let $q\geq 1$ be such that
\begin{align}\nonumber
q^{-1}&=\textstyle 1-\frac{\dcritical-1}{2}\log\left(1+\frac{1}{\dcritical-1} \right).
\end{align}
For any $x> 0$, $d> 0$, we have that  
\begin{align}\nonumber
\HCFactor(q,d,x) \leq  \HCFactor(q,\dcritical, \tilde{x})=(\dcritical)^{-1},
\end{align}
where $\tilde{x}\in[0,1]$ is the unique fix-point of $\trecur_{\dcritical, {\rm sym}}$, i.e.,  $\tilde{x}=\trecur_{\dcritical, {\rm sym}}(\tilde{x})$. 
\end{lemma}

By combining Lemmas \ref{lemma:InsteadOfHolder} and \ref{lemma:XiBoundAtFixPoint} we get 
\eqref{prop:Target4PotContractsB}. This concludes the   proof of Proposition \ref{prop:PotContracts}. 
\hfill $\Box$

\section{Proofs of the results in Section \ref{sec:RecursionVsSpectralIneq}}\label{sec:Proofs4SIResults}

\subsection{Proof of Theorem \ref{thrm:Recurrence4InfluenceEigenBound}}\label{sec:thrm:Recurrence4InfluenceEigenBound}
Firstly we note that for  $\bphi\in \mathbb{R}^{\WEdges_{\LamSAW}}$  as defined in 
\eqref{def:OfRestInfluenceWeights}, we have that  
\begin{align}\label{eq:thrm:Recurrence4InfluenceEigenBoundStep1A}
\max_{e\in \WEdges_{\LamSAW}} \left\{ |\bphi(e)| \right\} &\leq \sup\{ h(x)\  |\  x\in [-\infty,+\infty] \}.
\end{align}
The assumption that  the set of functions  $\{ \logtrecur_d\}_{d\in [\maxDeg]}$ exhibits  $\delta$-contraction 
implies that
\begin{align}\label{eq:thrm:Recurrence4InfluenceEigenBoundStep1B}
\sup\{ h(x)\  |\  x\in [-\infty,+\infty] \} \leq \delta.
\end{align}
Combining \eqref{eq:thrm:Recurrence4InfluenceEigenBoundStep1A} and \eqref{eq:thrm:Recurrence4InfluenceEigenBoundStep1B},
 we have that
\begin{align}\label{eq:thrm:Recurrence4InfluenceEigenBoundStep1C}
\max_{e\in \WEdges_{\LamSAW}} \left\{ |\bphi(e)| \right\} &\leq \delta.
\end{align}
Recall from  Lemma \ref{lemma:InfluenceMatrixIsWalkMatrix}  that $\infmatrix^{\Lambda,\tau}_{G}$
is a walk-matrix. Furthermore, using   Corollary  \ref{cor:MaxElement}   we get  that
\begin{align}\label{eq:InluenceVsNWalkSPRadBasicConc}
\textstyle \spradius\left( \infmatrix^{\Lambda,\tau}_{G} \right) & \leq 
\textstyle \spradius\left(\Pweight_{\NWalks, \bar{\bpsi}} \right). 
\end{align} 
where the vector of weights  $\bar{\bpsi}$ has all its components equal to $\delta$. 
%For the second inequality   we use  Lemma \ref{lemma:MonotoneVsSRad} and the fact that 
%both $\Pweight_{\NWalks, \bar{\bpsi}}$ and $\Pweight_{\InfWalks, \bar{\bpsi}}$ are  non-negative matrices. 
Note that we  choose $\NWalks$ because the  length of any  self-avoiding path is $\leq n$. 
Since all the components of $\bar{\bpsi}$ have the same value $\delta>0$, Lemma 
\ref{lemma:WalkMatrixVsAdjacency} implies that
\begin{align} \nonumber  
\Pweight_{\NWalks, \bar{\bpsi}}  &= \textstyle \sum^{n}_{\ell=0}\left(\delta \cdot \simpleadj_G\right)^{\ell}. 
\end{align}
We have that
\begin{align}\nonumber
\spradius\left (\Pweight_{\NWalks, \bar{\bpsi}}  \right ) & \leq \lnorm \Pweight_{\NWalks, \bar{\bpsi}}  \rnorm_2  \ = \ \textstyle \lnorm \sum^{n}_{\ell=0}(\delta \simpleadj_G)^{\ell} \rnorm_2
\ \leq \  \textstyle  \sum^{n}_{\ell=0} |\delta|^{\ell} \cdot \lnorm (\simpleadj_G)^{\ell} \rnorm_2.
\end{align}
Furthermore, since  $\delta=\frac{1-\epsilon}{\spradius_G}>0$ 
and  $\lnorm (\simpleadj_G)^{\ell} \rnorm_2 = (\spradius_G)^{\ell}$, because $\simpleadj_G)^{\ell}$ is symmetric,  
we get that
\begin{align}\nonumber
\spradius\left (\Pweight_{\NWalks, \bar{\bpsi}}  \right ) 
& \leq  \textstyle  \sum^{n}_{\ell=0} (1-\epsilon)^{\ell} \ \leq \  \sum^{\infty}_{\ell=0} (1-\epsilon)^{\ell}=\epsilon^{-1}.
\end{align}
The theorem follows  by plugging the above into \eqref{eq:InluenceVsNWalkSPRadBasicConc}.
\hfill $\Box$

\subsection{Proof of Theorem \ref{thrm:Recurrence4InfluenceEigenBoundNonLinear}.}
\label{sec:thrm:Recurrence4InfluenceEigenBoundNonLinear}
As we prove in Lemma \ref{lemma:InfluenceMatrixIsWalkMatrix} the matrix 
$\infmatrix^{\Lambda,\tau}_{G}$ is identical to the  walk-matrix $\Pweight_{\LamSAW,\bphi}$, 
where  $\LamSAW$ is the  set of self-avoiding walks in  $G$ that do not use vertices in 
$\Lambda$, while $\bphi\in \mathbb{R}^{\WEdges_{\LamSAW}}$ is defined in 
\eqref{def:OfRestInfluenceWeights}.  In that respect, it suffices to prove that under the 
assumption of Theorem \ref{thrm:Recurrence4InfluenceEigenBoundNonLinear} we have that
\begin{align}\nonumber  
\spradius  \left( \Pweight_{\LamSAW,\bphi}  \right) &\leq 
1+ \zeta \cdot (1-(1-\epsilon)^s)^{-1}  \cdot \left( {\maxDeg}/{\spradius_G} \right)^{1-\frac{1}{s}},
\end{align}
where $\zeta, s$ and $\epsilon$ are specified  in the statement of Theorem \ref{thrm:Recurrence4InfluenceEigenBoundNonLinear}.

In light of Corollary \ref{cor:MyNormVsSPRadius} the above follows by showing that 
\begin{align}\label{eq:thrm:Recurrence4InfluenceEigenBoundANonLinBasis}
\lnorm \Pweight_{\LamSAW,\bphi}  \rnorm_{\SpGMatrix_{\cP}, \frac{1}{s}, \infty}    &\leq 
1+ \zeta \cdot (1-(1-\epsilon)^s)^{-1}\cdot \left( {\maxDeg}/{\spradius_G} \right)^{1-\frac{1}{s}},
\end{align}
where  $\SpGMatrix_{\cP}\in \mathbb{R}^{V_{\cP}\times V_{\cP}}_{\geq 0}$ is the diagonal matrix 
whose entry$\SpGMatrix(u,u)=\maxeigenv(u)>0$.

In order to prove \eqref{eq:thrm:Recurrence4InfluenceEigenBoundANonLinBasis}, we use Theorem 
\ref{thrm:MySpectralMatrixNorm}. Particularly, \eqref{eq:thrm:Recurrence4InfluenceEigenBoundANonLinBasis}  
follows from Theorem \ref{thrm:MySpectralMatrixNorm} 
once we show that the assumptions of Theorem \ref{thrm:Recurrence4InfluenceEigenBoundNonLinear} imply that 
we can have $\bgamma\in \mathbb{R}^{\WEdges_{\LamSAW}}_{>0}$ which is a $(s',\delta',c')$-potential with respect 
to  $\LamSAW$ and  $\bphi$ where  $s'=s$, $\delta'=\frac{1-\epsilon}{\spradius_G}$  and $c'=\frac{\zeta}{\spradius_G}$.
As mentioned above,  $s, \epsilon$ and  $\zeta$  are specified in the statement of Theorem \ref{thrm:Recurrence4InfluenceEigenBoundNonLinear}.

Note that Theorem \ref{thrm:Recurrence4InfluenceEigenBoundNonLinear} assumes the existence of a
$(s,\delta, c)$-potential $\potF$, for  $s\geq 1$, $\delta=\frac{1-\epsilon}{\spradius_G}$ and $c=\frac{\zeta}{\spradius_G}$. We let the function 
$\xdpotF=\potF'$.
We define  the weights $\bgamma\in \mathbb{R}^{\WEdges_{\LamSAW}}_{>0}$ as follows: for any edge $e\in \WEdges_{\LamSAW}$ such 
that $e=\{u,w\}$ and $w$ is the child of $u$, we have 
\begin{align}\nonumber 
\bgamma(e) &=\xdpotF(w). 
\end{align}
From the fact that $\potF$ is a $(s,\delta, c)$-potential,  it follows  immediately  (i.e., from Definitions \ref{def:GoodPotential}, \ref{def:PotentialWeights}) 
that the weights  $\bgamma$ are $(s',\delta',c')$-potential with respect to  $\LamSAW$ and  $\bphi$.
For the above it is useful to recall that the edge weights $\bphi$ are specified in \eqref{def:OfRestInfluenceWeights}.
This  concludes the proof of Theorem \ref{thrm:Recurrence4InfluenceEigenBoundNonLinear}.
\hfill $\Box$

\subsection{Proof of Theorem \ref{thrm:Recur4InfSPBoundNBK}}\label{sec:thrm:Recur4InfSPBoundNBK}
Throughout the proof we use either $\hspradius_G$, or $\spradius(\NBMatrix_G)$ to refer 
to the spectral radius of $\NBMatrix_G$ and  this should create no confusion.

Working as in the proof of Theorem \ref{thrm:Recurrence4InfluenceEigenBound}, we have the following:
for  $\bphi$ defined in \eqref{def:OfRestInfluenceWeights}, we have that
\begin{align}\nonumber 
\max_{e} \left\{ |\bphi(e)| \right\} &\leq \delta.
\end{align}
Furthermore, using the above and  Corollary  \ref{cor:MaxElement4NB}   we get  that
\begin{align} \nonumber 
\textstyle \spradius\left( \infmatrix^{\Lambda,\tau}_{G} \right) & \leq 
\textstyle \spradius\left( \Pweight_{\nbn, \bar{\bpsi}} \right) \leq \spradius\left( \Pweight_{\nbInf, \bar{\bpsi}} \right).
\end{align} %\nbInf
For the second inequality   we use  Lemma \ref{lemma:MonotoneVsSRad} and the fact that 
both $\Pweight_{\nbn, \bar{\bpsi}}$ and $\Pweight_{\nbInf, \bar{\bpsi}}$ are 
non-negative matrices. 
 Since $\spradius\left( \Pweight_{\nbInf, \bar{\bpsi}} \right)\leq \lnorm  \Pweight_{\nbInf, \bar{\bpsi}} \rnorm_2$,
we have that
 \begin{align}\label{eq:InluenceVsNNBKWalkSPRadBasicConc}
 \textstyle \spradius\left( \infmatrix^{\Lambda,\tau}_{G} \right) & \leq \lnorm  \Pweight_{\nbInf, \bar{\bpsi}} \rnorm_2.
 \end{align}

The assumption   that $\delta =  (1-\epsilon)/\spradius(\NBMatrix_G)$,  implies that $\spradius(\delta  \NBMatrix_G)<1$, 
hence,   it is standard that $ \sum_{\ell\geq 0} \delta^{\ell}  \NBMatrix^{\ell}_G=(\bI-\delta \cdot \NBMatrix_G)^{-1}.$ 
Hence, we get that
\begin{align}\label{eq:WnbInfDeltaWell}
\Pweight_{\nbInf, \bar{\bpsi}} &=\bI+ \delta\cdot \vtooedge\cdot  (\bI-\delta \cdot \NBMatrix_G)^{-1} \cdot\oedgetov.
\end{align}
The theorem follows by simplifying the above and showing that 
\begin{align}\label{eq:Target4thrm:Recur4InfSPBoundNBK}
\Pweight_{\nbInf, \bar{\bpsi}}=\left( \bI -   \delta\simpleadj_G +\delta^2(\bD-\bI) \right)^{-1}.
\end{align}
%That is, we just need to plug the above into \eqref{eq:InluenceVsNNBKWalkSPRadBasicConc}.

For $k\geq 0$,  let $\bW^{(k)}$ be the $V\times V$ matrix, such that 
for every $u,w\in V$ the entry $\bW^{(k)}(u,w)$ is equal to the number of 
non-backtracking walks  of length exactly  $k$ from vertex $u$ to vertex $w$.  
Note that 
\begin{align}
\Pweight_{\nbInf, \bar{\bpsi}} &=\sum^{\infty}_{k=0} \delta^{k}\cdot \bW^{(k)}.
\end{align}
From \eqref{eq:WnbInfDeltaWell} we have that the above summation is well defined.
Furthermore, we have that  $\bW^{(0)}=\bI$ and $\bW^{(1)}=\simpleadj_G$,  while for $k \geq 2$, we have 
the following recursive relation:
\begin{align}\label{eq:RecurWK}
\bW^{(k)}&= \simpleadj_G\cdot \bW^{(k-1)}-(\bD-\bI) \bW^{(k-2)}.
\end{align}
The above recursive relation is standard in the literature, e.g. see \cite{PhDNBM}. 

Consider the generating function $\bF(x)=\sum^{\infty}_{k=0}x^k\cdot \bW^{(k)}$.  
Eq.  \eqref{eq:WnbInfDeltaWell} implies that  the radius of convergence for $\bF(x)$ 
is the open interval $(-\hspradius^{-1}_G, \hspradius^{-1}_G)$.
 Specifically, from \eqref{eq:RecurWK} it elementary to show that for any $x\in (-\hspradius^{-1}_G, \hspradius^{-1}_G)$ 
 we have that  $\bF(x)=(\bI-x\simpleadj_G+x^2(\bD-\bI))^{-1}$. 
Then, \eqref{eq:Target4thrm:Recur4InfSPBoundNBK} follows by noting that  $\Pweight_{\nbInf, \bar{\bpsi}} = \bF(\delta)$.
The theorem follows.
\hfill $\Box$

\section{Proof of results from Section \ref{sec:TopoligcalMethod101}}\label{sec:Proofs4Topol101}

\subsection{Proof of Lemma \ref{lemma:StrongSubtreeRelation}}\label{sec:lemma:StrongSubtreeRelation}

Consider the trees $\walkT_{\cP}(w)$ and  $\walkT_{\cQ}(w)$.  Using standard notation, for $\walkT_{\cP}(w)$
we have $\WVertices_{\cP, w},  \WEdges_{\cP, w}$   the set of vertices and edges. Similarly for 
$\walkT_{\cQ}(w)$ we have the corresponding sets $\WVertices_{\cQ, w}$ and  $\WEdges_{\cQ, w}$. 
The lemma follows by showing that $\WVertices_{\cQ, w}\subseteq \WVertices_{\cP, w}$ and 
$\WEdges_{\cQ, w}\subseteq \WEdges_{\cP, w}$.  

The relation that $\WVertices_{\cQ, w}\subseteq \WVertices_{\cP, w}$ follows immediately from the fact that
$\cQ\subseteq \cP$. Specifically, every walk $M\in \cQ\cap \cP$  corresponds to the same vertex $z$ in both trees. 
Since $\cQ\subseteq \cP$ we have that for $M\in \cQ$ we  have $M\in \cQ\cap \cP$, as well. This implies that every
vertex $z$ in the tree $\walkT_{\cQ}(w)$ also belongs to the tree  $\walkT_{\cP}(w)$.

In order to prove that $\WEdges_{\cQ, w}\subseteq \WEdges_{\cP, w}$ we work as follows:
Let the edge $e=\{x, z\}$ in $\WEdges_{\cQ, w}$, while assume that $x$ is the parent of $z$
in the tree $\walkT_{\cQ}(w)$. There are walks  $M_a=z_0,z_1, \ldots, z_{\ell}$ and 
$M_b=z_0,z_1, \ldots, z_{\ell-1}$, for $\ell\geq 1$ such that $M_a, M_b\in \cQ$ while
$M_a$  corresponds to the vertex $z$ in $\walkT_{\cQ}(w)$ and 
$M_b$  corresponds to the vertex $x$. 
Since we have assumed that $\cQ\subseteq \cP$, we have that $M_a, M_b\in \cP$.
Consequently vertices $x,z$ appear in $\walkT_{\cP}(w)$, while they are adjacent and $x$ is the parent of $z$. 
This implies  that the edge $e\in \WEdges_{\cP, w}$.  Hence, we conclude that $\WEdges_{\cQ, w}\subseteq \WEdges_{\cP, w}$.

The above concludes the proof of the lemma.  \hfill $\Box$

\subsection{Proof of Lemma \ref{lemma:WalkMatrixVsAdjacency}}\label{sec:lemma:WalkMatrixVsAdjacency}
Let  $\bC=  \textstyle \sum^{k}_{\ell=0} (\zeta \cdot \simpleadj_G)^{\ell}$.
We prove the lemma by showing   that for any  $k\geq 0$ and any pair of 
vertices   $u, v\in V$, we have that
\begin{align}\label{eq:Target4thrm:MaxElement}
\Pweight_{\KWalks,\bpsi}(u,  v) &=   \bC(u, v).
\end{align}
Note that both $\Pweight_{\KWalks,\bpsi}, \bC \in \mathbb{R}^{V\times V}$.

For any $u, v\in V$ and  any integer $\ell\geq 1$,  let $\cM^{\ell}_{u,v}\subseteq V^{\ell+1}$ contain
every $\ell+1$-tuple of vertices  $w_0, w_1, \ldots, w_{\ell} \in V^{\ell+1}$ such that $w_0=u$ and $w_{\ell}=v$.
Note that we allow $w_i=w_j$.
We also define $\cM^{\ell}_{u,v}$ for $\ell=0$. In this case, we have that $\cM^{0}_{u,v}=\{u\}$ only if $u=v$.
Otherwise, the set is empty.

For $\ell\geq 0$,  let $S_{\ell}$ be the set of all walks of length exactly $\ell$ that start from $u$ and end at $v$ in the graph
$G$. 
From the definition of the set $\KWalks$  we have that 
\begin{align}\label{eq:WNmaxUVExplicit}
\Pweight_{\KWalks,\bpsi}(u,  v) & =\  \sum^k_{\ell=0}  \sum_{w_0, \ldots, w_{\ell} \in \cM^{\ell}_{u,v}} \zeta^{\ell}\times
 {\bf 1}\{w_0,  \ldots, w_{\ell}\in S_{\ell}\}.
\end{align}
Now, consider  the adjacency matrix $\simpleadj_G$.
It is standard that for any  $u,v\in V$ and  any $\ell\geq 0$,  we have that
\begin{align} \nonumber
\left( \simpleadj_G \right)^{\ell}(u,v) & =\sum_{w_0,  \ldots, w_{\ell} \in \cM^{\ell}_{u,v}}\prod_{i\in [\ell]}\simpleadj_G(w_{i-1},w_{i}).
\end{align}
Note that each summad on the r.h.s. is equal to $1$ if $w_0, w_1, \ldots, w_{\ell}$ is a walk from $u$ to $v$ in
$G$. Otherwise the summad is zero.  Using this observation, we conclude that
\begin{align} \nonumber 
\left( \simpleadj_G \right)^{\ell}(u,v) & = \sum_{w_0, \ldots, w_{\ell} \in \cM^{\ell}_{u,v}} {\bf 1}\{w_0,  \ldots, w_{\ell}\in S_{\ell}\}.
\end{align}
The above implies that  
\begin{align}\label{eq:DeltaUVExplicit}
\bC (u,v)  & =  \sum^{k}_{\ell=0} 
\sum_{w_0,  \ldots, w_{\ell} \in \cM^{\ell}_{u,v}} \zeta ^{\ell}\times {\bf 1}\{w_0, \ldots, w_{\ell}\in S_{\ell}\}.
\end{align}
Comparing \eqref{eq:DeltaUVExplicit} and \eqref{eq:WNmaxUVExplicit} we immediately get 
\eqref{eq:Target4thrm:MaxElement}. 
The lemma follows. \hfill $\Box$

\subsection{Proof of Lemma \ref{lemma:WalkMatrixVsHashimoto}}\label{sec:lemma:WalkMatrixVsHashimoto}
Let $\bM =\sum^{k-1}_{\ell=0} \zeta^{\ell+1} \cdot \NBMatrix^{\ell}_G$, while  let  $\bD=  \bI+\textstyle \vtooedge \cdot \bM \cdot \oedgetov$, 
where $\bI$ is the $V\times V$ identity matrix and  $\vtooedge, \oedgetov$ are defined in \eqref{def:HasVertex2EdgeMatrices}.
Note that $\bM$ is a $\OrntEdges\times \OrntEdges$ matrix, while 
both $\Pweight_{\nbk,\bpsi}$ and  $\bD$ are $ V\times V$.

We prove the lemma by showing   that for any  $k\geq 0$ and any pair of
vertices   $u, v\in V$, we have that
\begin{align}\label{eq:Target4lemma:WalkMatrixVsHashimoto}
\Pweight_{\nbk,\bpsi}(u,  v) &=   \bD(u, v).
\end{align}
For   $(x,y), (r,s) \in V\times  V$,  allowing repetitions, for  integer $\ell\geq 3$, let $\cM^{\ell}_{(x,y),(r,s)}\subseteq V^{\ell+1}$ 
contain every $\ell+1$-tuple of vertices  
$w_0, w_1, \ldots, w_{\ell} \in V^{\ell+1}$ such that $w_0=x$,  $w_1=y$, $w_{\ell-1}=r$ and $w_{\ell}=s$.
Note that we allow $w_i=w_j$.  

We also define $\cM^{\ell}_{(x,y),(r,s)}$ for $1 \leq \ell  \leq 2$, but there are some restriction on $x,y,r,s$. 
For  $\ell=1$, we have  that  $\cM^{1}_{(x,y),(r,s)}=\{(x,y)\}$ only if $x=r$ and $y=s$. 
Otherwise, we have $\cM^{1}_{(x,y),(r,s)}=\emptyset$. 
Similarly,  for  $\ell=2$, we have $\cM^{2}_{(x,y),(r,s)}=\{x,y,s\}$  is only  defined for $y=r$. Otherwise, we have
 $\cM^{2}_{(x,y),(r,s)}=\emptyset$.

Furthermore, for $\ell\geq 1$,  let $S_{\ell}(e,f)$ be the set of all non-backtracking walks in $G$ of length exactly $\ell$ that 
start with edge $e$, while the last edge in the path is $f$. For $e=(x,y)$  and $f=(r,s)$,   $S_{\ell}(e,f)$ 
contains the  non-backtracking walks such that  the first vertex in the walk is $x$, the second is $y$
while the one before the last vertex is $r$ and the last vertex is $s$.
Note that for  the extreme case $\ell=1$,  $S_{\ell}(e,f)$ is non-empty only if $e=f$.

From the definition of the set $\nbk$  we have that 
\begin{align}\label{eq:WNBkUVExplicit}
\Pweight_{\nbk,  \bpsi}(u,  v) & = {\bf 1}\{u=v\}+\sum_{\substack{e=(x,y)\in \OrntEdges :\\ x=u}} \ \  
\sum_{\substack{f=(r,s)\in \OrntEdges:\\ s=v}}\ \    \sum_{\ell \in [k] } \ \  \sum_{w_0, \ldots, w_{\ell} \in \cM^{\ell}_{e,f}} \zeta^{\ell}\times
 {\bf 1}\{w_0,  \ldots, w_{\ell}\in S_{\ell}(e,f)\}.
\end{align}
Working as in the proof of Lemma \ref{lemma:WalkMatrixVsAdjacency}, we get the following: 
for any $\ell\geq 0$ and $e, f\in \OrntEdges$, we have that
\begin{align}\nonumber
 \NBMatrix^{\ell}_G(e,f) &=  \textstyle \sum_{w_0, \ldots, w_{\ell} \in \cM^{\ell+1}_{e,f}} 
 {\bf 1}\{w_0,  \ldots, w_{\ell}\in S_{\ell+1}(e,f)\}.
\end{align}
Hence, we have that 
\begin{align}\nonumber 
\bM(e,f) &=\sum^{k-1}_{\ell =0 } \zeta^{\ell+1}\times  \sum_{w_0, \ldots, w_{\ell} \in \cM^{\ell+1}_{e,f}} 
 {\bf 1}\{w_0,  \ldots, w_{\ell}\in S_{\ell+1}(e,f)\}.
\end{align}
Combining the above with  \eqref{eq:WNBkUVExplicit},  we get that
\begin{align}\label{eq:WnbkVsBMef}
\Pweight_{\nbk,\bpsi}(u,  v) & = {\bf 1}\{u=v\}+ \sum_{\substack{e=(x,y)\in \OrntEdges :\\ x=u}} \ \  
\sum_{\substack{f=(r,s)\in \OrntEdges:\\ s=v} }\bM(e,f)
\end{align}
Furthermore,  the definition of $\bD$, $\vtooedge$ and $\oedgetov$, implies that
\begin{align}
\textstyle  \bD(u,v) & = \ \left( \bI+\vtooedge \cdot \bM   \cdot \oedgetov\right)(u,v) 
\ =\ {\bf 1}\{u=v\}+\sum_{(x,z)} \sum_{(y,q)} \vtooedge_{u,(x,z)} \bM_{(x,z), (y,q)} \oedgetov_{(y,q),v} \nonumber \\
&={\bf 1}\{u=v\}+\sum_{(x,z)} \sum_{(y,q)} {\bf 1}\{u=x\}\times {\bf 1}\{q=v\}  \times \bM_{(x,z), (y,q)} \nonumber \\
&={\bf 1}\{u=v\}+\sum_{\substack{e=(x,y)\in \OrntEdges :\\ x=u}} \sum_{\substack{f=(r,s)\in \OrntEdges:\\ s=v}}  \bM_{e, f}. \label{eq:AppropriateNBSumOfEntries} 
\end{align}
From \eqref{eq:WnbkVsBMef} and \eqref{eq:AppropriateNBSumOfEntries} we conclude that 
\eqref{eq:Target4lemma:WalkMatrixVsHashimoto} is true.  The lemma follows.

\subsection{Proof of Lemma \ref{lemma:InfluenceMatrixIsWalkMatrix}}\label{sec:lemma:InfluenceMatrixIsWalkMatrix}

Consider $\Pweight_{\saw,\infweight}$ where $\infweight$ is defined in \eqref{def:OfInfluenceWeights}. 
From the definition of the matrix we have that $\Pweight_{\saw,\infweight}$ is indexed by the vertices in $V$.
Also, we have  that both  $\infmatrix^{\Lambda,\tau}_{G}$ and $\Pweight_{\cP,\bphi}$
are matrices  indexed by the vertices in $V\setminus \Lambda$.

The lemma will follow by showing the following two relations:
 for any  $w,v\in V\setminus \Lambda$, we have that
\begin{align}\label{eq:Base4lemma:InfluenceMatrixIsWalkMatrixA}
\infmatrix^{\Lambda,\tau}_{G}(w,v) &= \Pweight_{\saw,\infweight}(w,v),
\end{align}
while, for the same $w,v$ we also have that
\begin{align}\label{eq:BaseB4lemma:InfluenceMatrixIsWalkMatrixA}
\Pweight_{\LamSAW,\bphi}(w,v) &= \Pweight_{\saw,\infweight}(w,v).
\end{align}
Note that the lemma follows by  combining the above with \eqref{eq:Base4lemma:InfluenceMatrixIsWalkMatrixA}.

We start by proving that \eqref{eq:Base4lemma:InfluenceMatrixIsWalkMatrixA} is true.
From the definition of walk-matrix, for $\infweight\in \mathbb{R}^{\WEdges_{\saw}}$, we have that 
$\Pweight_{\saw,\infweight}\in \mathbb{R}^{V\times V}$,  while   
\begin{align}\label{eq:Base4lemma:InfluenceMatrixIsWalkMatrixB}
\Pweight_{\saw,\infweight}(w,v)={\textstyle \sum_{M} \prod_{e\in M} }\infweight(e),
\end{align}
where $M$ varies over the paths in $\Tsaw(w)$ from the root to the set $\cp_{\saw,w}(v)$,
while, as mentioned above,  $\infweight$ is specified in \eqref{def:OfInfluenceWeights}.
In light of the above, it is immediate that \eqref{eq:Base4lemma:InfluenceMatrixIsWalkMatrixA}
is true  by showing  formally that \eqref{eq:InfluenceEntryAsWeightedSum} is true.

Consider $T=\Tsaw(w)$ and the Gibbs distribution $\mu_T(\cdot\ |\ \Gamma, \sigma)$, where $\Gamma$
and $\sigma\in \{\pm 1\}^{\Gamma}$ are  defined in Section \ref{sec:InfluenceVsWalkMatrix}. 
With respect to the aforementioned Gibbs distribution, let  $\infmatrix^{\Gamma,\sigma}_{T}$ be the 
corresponding pairwise influence matrix.  
The following result from \cite{VigodaSpectralInd} shows a useful relationship between the two 
influences matrices $\infmatrix^{\Lambda,\tau}_{G}$ and $\infmatrix^{\Gamma,\sigma}_{T}$.

\begin{lemma}\label{lemma:Redaux2Tree}
Let vertex $r$ be the root of the tree  $\Tsaw(w)$. We have that 
\begin{align}\label{eq:lemma:Redaux2Tree}
\infmatrix^{\Lambda,\tau}_{G}(w,v)=\sum_{ u \in {\cp}(v)} \infmatrix^{\Gamma,\sigma}_{T}(r,u),
\end{align}
where for $s\in V$,  $\cp(s)$ is the set of  copies of vertex $s$ in $\Tsaw(w)$ which do not belong in $\Gamma$. 
\end{lemma}

Combining results from \cite{OptMCMCIS}  and \cite{VigodaSpectralInd},  we have the following lemma, 
whose proof appears in Section \ref{sec:lemma:RootInfluenceProduct}.

\begin{lemma}\label{lemma:RootInfluenceProduct}
Let $r$ be the root in $\Tsaw(w)$. For any vertex $u$ in $\Tsaw(w)$, different than $r$,   the following holds:
letting $z_0, \ldots, z_{\ell}$ be the path in $\Tsaw(w)$ that connects the root of the tree with  $u$, 
i.e.,  $z_0=r$ and $z_{\ell}=u$,  we have that 
\begin{align}\label{eq:DefOfXi}
\infmatrix^{\Gamma,\sigma}_{T}(r, u)  & = \textstyle  \prod^{\ell}_{i=1} \infweight(\{z_{i-1}, z_i\}),
\end{align}
where $\infweight \in \mathbb{R}^{\WEdges_{\saw}}$ is specified in \eqref{def:OfInfluenceWeights}.
Furthermore, we have   $\infmatrix^{\Gamma,\sigma}_{T}(r,r)=1$.  
\end{lemma}

In light of Lemmas \ref{lemma:RootInfluenceProduct} and \ref{lemma:Redaux2Tree},   \eqref{eq:InfluenceEntryAsWeightedSum} follows 
as a simple corollary. We conclude that \eqref{eq:Base4lemma:InfluenceMatrixIsWalkMatrixA} is true. 

Now, consider the matrix $\Pweight_{\LamSAW,\bphi}$.   Similarly to \eqref{eq:Base4lemma:InfluenceMatrixIsWalkMatrixB}, 
we have that
\begin{align}\label{eq:BaseB4lemma:InfluenceMatrixIsWalkMatrixB}
\Pweight_{\LamSAW,\bpsi}(w,v)={\textstyle \sum_{M}\prod_{e\in M}} \bpsi(e),
\end{align}
where $M$ varies over the paths in $\walkT_{\LamSAW}(w)$ from the root to the set $\cp_{\LamSAW,w}(v)$.

Comparing the trees $\walkT_{\LamSAW}(w)$ and $\walkT_{\saw}(w)$, we have that the first one 
is a subtree of the second one.   One  obtains  $\walkT_{\LamSAW}(w)$  be removing the subtrees 
of $\walkT_{\saw}(w)$ which are rooted to vertices in $\Gamma$.  Due to \eqref{def:OfRestInfluenceWeights}, 
we have that $\Pweight_{\LamSAW,\bpsi}(w,v)$ and $\Pweight_{\saw,\infweight}(w,v)$ differ only on the sum 
of the weight of the paths that appear in $\walkT_{\saw}(w)$ but do not appear in  $\walkT_{\LamSAW}(w)$.  
However,  it is immediate  that the weight of these paths  is equal to zero. Note that  each one of these
paths involves at least one vertex in $\Gamma$, while  all the edges  $e$ incident to such vertex have 
$\infweight(e)=0$, i.e.,  this is due to \eqref{def:OfInfluenceWeights}. Hence,  
\eqref{eq:BaseB4lemma:InfluenceMatrixIsWalkMatrixA} is true.

All the above, conclude the proof of   Lemma \ref{lemma:InfluenceMatrixIsWalkMatrix}. \hfill $\Box$

\subsubsection{Proof of Lemma \ref{lemma:RootInfluenceProduct}}\label{sec:lemma:RootInfluenceProduct}
In order to prove Lemma \ref{lemma:RootInfluenceProduct}, we only need to use the following two results from \cite{OptMCMCIS} 
and \cite{VigodaSpectralInd}, respectively.

\begin{lemma}[\cite{OptMCMCIS}]\label{lemma:ProductFactorizationInlfuence}
Suppose that $x,y,z$ are three distinct vertices in $T=\Tsaw(w)$ such  that $y $  is on the unique path from $x$ to $z$. Then
\begin{align}\nonumber
 \infmatrix^{\Gamma,\sigma}_{T}(x,z)=  \infmatrix^{\Gamma,\sigma}_{T}(x, y)\times \infmatrix^{\Gamma,\sigma}_{T}(y, z).
\end{align}
\end{lemma}

\begin{lemma}[\cite{VigodaSpectralInd}]\label{lemma:1StepInfluencePotential}
Let $v,u$ be two vertices in $T=\Tsaw(w)$, while suppose $v\notin \Gamma$ and   $u$ is a child of $v$. Then we have 
\begin{align}\nonumber
 \infmatrix^{\Lambda,\tau}_{T}(v,u) &=
 \left \{
 \begin{array}{lcl}
 h(\log \gratio^{\Lambda,\tau}_u) &\quad& \textrm{if $u\notin \Gamma$} \\
0 &\quad& \textrm{otherwise}. \\
 \end{array}
 \right .
 \end{align}
 
\end{lemma}

\begin{proof}[Proof of Lemma \ref{lemma:RootInfluenceProduct}]
Recall that  $z_0, z_1, \ldots, z_{\ell}$ is the path in $T$ that
starts from the root  to  the vertex $u$.   

If $\ell=0$, then we immediately have $\infmatrix^{\Gamma,\sigma}_{T}(r,r)=1$.
For what follows, we focus on the case $\ell\geq 1$.
 Lemma \ref{lemma:ProductFactorizationInlfuence} implies  that
\begin{align}\nonumber
\infmatrix^{\Gamma,\sigma}_{T}(r,u) &=\textstyle \prod^{\ell}_{i=1} \infmatrix^{\Lambda,\tau}_{T}(z_{i-1},z_i).
\end{align}
Furthermore, Lemma \ref{lemma:1StepInfluencePotential} implies the following:
If there is $z_i\in \Gamma$, then
\begin{align}\nonumber
\infmatrix^{\Gamma,\sigma}_{T}(r,u) =0,
\end{align}
while if all $z_i$'s are outside $\Gamma$, then we have
\begin{align}\nonumber
\infmatrix^{\Gamma,\sigma}_{T}(z_{i-1},z_i)=h(\log \gratio^{\Lambda, \tau}_{z_i}).
\end{align}
From \eqref{def:OfInfluenceWeights}, it is immediate that 
\begin{align}\nonumber
\infmatrix^{\Gamma,\sigma}_{T}(r,u)  & = \textstyle  \prod^{\ell}_{i=1} \bxi(\{z_{i-1}, z_i\}).
\end{align}
The above concludes the proof. 
\end{proof}

\section{Proofs for results from Section \ref{sec:SPComparisonBasedOnEntries}}\label{sec:Proofs4EntryBasedTech}

\subsection{Proof of Theorem \ref{thrm:Monotonicity4WeiightedMatrix}}\label{sec:thrm:Monotonicity4WeiightedMatrix}

W.l.o.g. assume that the sets $\cP$ and $\cQ$ are non-empty. 
For  $u,v\in V_{\cQ}\subseteq V_{\cP}$ consider the trees $\walkT_{\cQ}(u)$ and $\walkT_{\cP}(u)$.
Let the sets  $\cp_{\cQ}(v)$ and $\cp_{\cP}(v)$ be the set of copies  of the vertex $v$ in these two trees, 
respectively.

Let $\cM_{\cQ}$ be the set of  paths in  $\walkT_{\cQ}(u)$  that  connect the root with the vertex set 
$\cp_{\cQ}(v)$.  Similarly,  for the tree $\walkT_{\cP}(u)$  we  define the set of paths $\cM_{\cP}$.

For each $M\in \cM_{\cQ}$ we consider the weight ${\tt weight}_{\cQ}(M)$ which is as follows:
\begin{align} \nonumber %\label{def:WeightPathCQ}
{\tt weight}_{\cQ}(M) &= {\textstyle \prod_{e\in M} } \bxi_1(e).
\end{align}
%while the weights of the edges are specified by $\bxi_1$.  
Similarly, for each $M\in \cM_{\cP}$ we consider  the weight ${\tt weight}_{\cP}(M)$ and  the weights of the 
edges are specified by $\bxi_2$. 

First we focus on proving \eqref{eq:thrm:Monotonicity4WeiightedMatrixA}. For this we use the following claim. 

\begin{claim}\label{claim:Injection2Edges}
There is an injective map $H:\cM_{\cQ}\to \cM_{\cP}$ such that for every $M\in \cM_{\cQ}$, we have that
\begin{align}\label{eq:claim:Injection2Edges}
{\tt weight}_{\cQ}(M) \leq {\tt weight}_{\cP}(H(M)).
\end{align}
\end{claim}
\begin{proof}
From Lemma \ref{lemma:StrongSubtreeRelation} we have the following: Since $\cQ\subseteq \cP$, 
the tree $\walkT_{\cQ}(u)$ is a subtree of  $\walkT_{\cP}(u)$.
We construct  the injective map $H:\cM_{\cQ}\to \cM_{\cP}$ such that it  maps every path $M$ in the 
tree $\walkT_{\cQ}(u)$ to the same path in the tree $\walkT_{\cP}(u)$.  The fact that $H(\cdot)$ is injective follows
from that $\walkT_{\cQ}(u)$ is a subtree of  $\walkT_{\cP}(u)$.

As far as \eqref{claim:Injection2Edges} is concerned, recall that 
\begin{align}
{\tt weight}_{\cQ}(M) &= \textstyle{\prod_{e\in M}}  \bxi_1(e)  & \textrm{and} &&
{\tt weight}_{\cQ}(H(M)) &= {\textstyle \prod_{e\in H(M)} }  \bxi_2(e). \nonumber
\end{align}
However,  from the definition of $H(\cdot)$ we have that the set of edges in $M$ is exactly the same as
the set of edges in $H(M)$.  Then, we get \eqref{eq:claim:Injection2Edges} by recalling that for any edge 
$e$ that appears in both $\walkT_{\cQ}(u), \walkT_{\cP}(u)$ we have that $\bxi_1(e)\leq \bxi_2(e)$.
The claim follows. 
\end{proof}

Let $\cS\subseteq \cM_{\cP}$ be such that $\cS=\cM_{\cP}\setminus \cM_{\cQ}$, that is  $\cS$ contains 
every $M\in \cM_{\cP}$ such that there is no $M'\in \cM_{\cQ}$ 
for which  $H(M')=M$.  Using  the above claim, we have that 
\begin{align}
\Pweight_{\cP,\bxi_2}(w,v) &\textstyle =\sum_{M\in \cM_{\cP}}{\tt weight}(M)  \ = \ 
\sum_{M\in \cM_{\cQ}}{\tt weight}(H(M))+ \sum_{M\in \cS }{\tt weight}(M)   \nonumber \\
&\textstyle \geq \sum_{M\in \cM_{\cQ}}{\tt weight}(M)+ \sum_{M\in \cS}{\tt weight}(M)   \qquad \qquad\qquad \mbox{[from Claim \ref{claim:Injection2Edges}]} \nonumber \\
&\textstyle \geq \sum_{M\in \cM_{\cQ}}{\tt weight}(M)\ =\  \Pweight_{\cQ,\bxi_1} (w,v). \nonumber
\end{align}
The inequality in the last line follows by noting that for every $M\in \cS$ we have that ${\tt weight}_{\cP}(M)\geq 0$, i.e., 
since we have  $\bxi_2\in \mathbb{R}^{\WEdges_{\cP}}_{\geq 0}$.  
The above proves that \eqref{eq:thrm:Monotonicity4WeiightedMatrixA} is true.

As far as \eqref{eq:thrm:Monotonicity4WeiightedMatrixB} is concerned, let us first prove it by using  
the assumption that $V_{\cP}=V_{\cQ}$. In this case, note that the two matrices  $\Pweight_{\cQ,\bxi_1}$,  $\Pweight_{\cP,\bxi_2}$
are indexed by the same set of vertices, and hence, they are of the same dimension.
Furthermore, from the assumption that $\bxi_2\in \mathbb{R}^{\WEdges_{\cP}}_{\geq 0}$ we have that
$\Pweight_{\cP,\bxi_2}$ is non-negative matrix, while  \eqref{eq:thrm:Monotonicity4WeiightedMatrixA} implies that
\begin{align}\nonumber 
|\Pweight_{\cQ,\bxi_1}|\leq \Pweight_{\cP,\bxi_2},
\end{align}
where recall that for the matrices    $\bA, \bB, \bC \in \mathbb{R}^{ N \times N}$,  we let $|\bA|$ denote the matrix having entries $|\bA_{i,j}|$.
while  we defined  $\bB\leq \bC$ to mean that $\bB_{i,j}\leq \bC_{i,j}$ for each $i$ and $j$.

Then, \eqref{eq:thrm:Monotonicity4WeiightedMatrixB} follows from the above by  using Lemma \ref{lemma:MonotoneVsSRad}.

We now proceed with \eqref{eq:thrm:Monotonicity4WeiightedMatrixB} and assuming  that $|V_{\cP}| > |V_{\cQ}|$, 
and $\Pweight_{\cP,\bxi_2}$ is symmetric. Consider the matrix $\Pweight^{V_{\cQ}}_{\cP,\bxi_2}$, 
this is the principle submatrix of $\Pweight_{\cP,\bxi_2}$  obtained by removing the rows and columns  
that correspond to the vertices in $V_{\cP}\setminus V_{\cQ}$. 

Note, now, that $\Pweight^{V_{\cQ}}_{\cP,\bxi_2}$ and  $\Pweight_{\cQ,\bxi_1}$ are indexed by the same set of vertices. 
Additionally, we have that $|\Pweight_{\cQ,\bxi_1}|\leq \Pweight^{V_{\cQ}}_{\cP,\bxi_2}$, which, together with
Lemma \ref{lemma:MonotoneVsSRad},  implies that 
\begin{align}\nonumber
\spradius\left(\Pweight_{\cQ,\bxi_1}\right) &\leq  \spradius\left(\Pweight^{V_{\cQ}}_{\cP,\bxi_2}\right).
\end{align}
However, since we assumed that $\Pweight_{\cP,\bxi_2}$ is symmetric and 
$\Pweight^{V_{\cQ}}_{\cP,\bxi_2}$ is a principal submatrix of $\Pweight_{\cP,\bxi_2}$, 
we also have that
\begin{align}\nonumber
\spradius\left(\Pweight^{V_{\cQ}}_{\cP,\bxi_2} \right)\leq \spradius\left(\Pweight_{\cP,\bxi_2}\right).
\end{align}
The above is a consequence of the well-known  Cauchy’s interlacing theorem, e.g. see \cite{MatrixAnalysis}.
Combining the two inequalities above it is immediate to get \eqref{eq:thrm:Monotonicity4WeiightedMatrixB}. 
The theorem follows.  \hfill $\Box$

\section{Proof  of results from Section  \ref{sec:SPComparisonBasedOnNorms}}\label{sec:Proofs4NormBasedTech}

\subsection{Proof of Theorem \ref{thrm:MySpectralMatrixNorm}}\label{sec:thrm:MySpectralMatrixNorm}
We use Theorems \ref{thrm:NormRedaux2WalkVector}, \ref{thrm:MonotoneWVector} and Corollary \ref{cor:MaxElement2}
to prove the theorem. Particularly, using these results, we obtain  the following:

For the walk-vector $\DBounded_{\KWalks}=\DBounded_{\KWalks} \left( \KWalks, {\SpGMatrix}^{\circ ({1}/{s})}, s, \delta, c\right) $, 
where  the $V\times V$ matrix $\SpGMatrix$ is defined in \eqref{def:AdjMatrixEigenvals}, we have that
\begin{align}\label{eq:Base4thrm:MySpectralMatrixNorm}
\lnorm \Pweight_{\cP,\bxi}  \rnorm_{\SpGMatrix_{\cP}, \frac{1}{s}, \infty}  &  \leq  \lnorm \DBounded_{\KWalks} \rnorm_{\infty}. 
\end{align}
Recall that the diagonal matrix $\SpGMatrix_{\cP}$ at the index of the matrix norm, is obtained from $\SpGMatrix$ by removing 
the rows and columns that correspond to vertices outside $V_{\cP}$. 
In light of \eqref{eq:Base4thrm:MySpectralMatrixNorm}, the theorem follows   showing  that 
\begin{align} \label{eq:Target4thrm:MySpectralMatrixNorm} %\label{eq:thrm:PowerMatrixTopologicalNorm}
\lnorm  \DBounded_{\KWalks}  \rnorm_{\infty}  &\leq  1+c \cdot \left(\maxDeg\right)^{1-\frac{1}{s}}\cdot
\left( \spradius_G \right)^{\frac{1}{s}} \cdot  {\textstyle \sum^{k-1}_{\ell=0} }  (\delta \cdot \spradius_G)^{\frac{\ell}{s}}.
\end{align}
%
%Recall that $ \DBounded_{\KWalks} \in \mathbb{R}^V_{>0}$. 
According to    Definition \ref{def:WalkVector},  for every $r\in V$ the entry $\DBounded_{\KWalks}(r)$ satisfies
\begin{align} \nonumber %
\DBounded_{\KWalks}(r) & \textstyle = 1+ \frac{c}{\left(\SpGMatrix(r,r)\right)^{\frac{1}{s}}}  \times 
\sum^d_{i=1}  \sum^{k-1}_{\ell= 0}
\left( \delta^{\ell} \cdot \sum_{w\in V } |\cp_{i,\ell}(w)| \cdot \SpGMatrix(w,w)
\right)^{\frac{1}{s}}.
\end{align}
In order to study the above quantity, note that the walk-tree of interest is $\walkT_{\KWalks}(r)$.

Let us recall the quantities in the above expression.  The quantity   $d$ is the degree of the root of the 
walk-tree $\walkT_{\KWalks}(r)$. Letting $T_i$ be the subtree  that is induced
by the $i$-th child of the root of $\walkT_{\KWalks}(r)$ and its decedents, 
$\cp_{i,\ell}(w)$ is the set of copies of vertex $w$ in the subtree $T_i$ that are 
at distance $\ell$ from the  root of $T_i$.

Using that  $\SpGMatrix(w,w)=\maxeigenv(w)$ for every $w\in V$, we have that
\begin{align}\label{eq:Base4Proofthrm:PowerMatrixTopologicalNorm}
\DBounded_{\KWalks}(r) & \textstyle = 1+ \frac{c}{\left(\maxeigenv(r)\right)^{\frac{1}{s}}} \times \sum^d_{i=1}\sum^{k-1}_{\ell= 0}
\left(\delta^{\ell} \cdot \sum_{w\in V} |\cp_{i,\ell}(w)|   \cdot  \maxeigenv(w)\right)^{\frac{1}{s}}.
\end{align}

\noindent
For $\ell\geq 1$, for every $x, w \in V$, let  $\cp_{i, x, \ell}(w) \subseteq \cp_{i,\ell}(w)$ be the set which contains  
all vertices $u$ in $T_i$,  copies of  $w$,  such that the parent of $u$  is in $\cp_{i, (\ell-1)}(x)$.

Since we assumed that the graph $G$ is simple, it is straightforward that for all  $w\in V$,  there are no
two copies of $w$ in $\walkT_{\KWalks}(r)$ that have the same parent.   This implies that  $|\cp_{i, (\ell-1)}(x)|$ is 
equal to $|\cp_{i, x, \ell}(w)|$, for any  $w$ neighbour of $x$ in $G$.  Using this observation, we have that
\begin{align}\nonumber
\sum_{w\in V} |\cp_{i,\ell}(w)|   \cdot  \maxeigenv(w) 
&= \sum_{w\in V} \sum_{x\in V}  |\cp_{i,x, \ell}(w)|   \cdot  \maxeigenv(w) \\ & =\ 
\sum_{x\in V} \sum_{w\in V} |\cp_{i,x, \ell}(w)|   \cdot  \maxeigenv(w) \\
&= \sum_{x\in V}   |\cp_{i, \ell-1}(x)|   \sum_{w\in V : \{w,x\}\in E}   \maxeigenv(w),  \nonumber
\end{align}
where in the second equation changed order of summation. Using the definition of $\maxeigenv$, 
note that the last summation is equal $ \eigenval_{\rm max}(\simpleadj_G)\maxeigenv(x)$.
Hence, we have that
\begin{align}\nonumber
{\textstyle \sum_{w\in V} } |\cp_{i,\ell}(w)|   \cdot  \maxeigenv(w) 
&=  \eigenval_{\rm max}(\simpleadj_G) \cdot 
{\textstyle \sum_{x\in V} }   |\cp_{i, \ell-1}(x)|  \cdot  \maxeigenv(x) % \label{eq:UseThetaMaxFirst4Eigenvector} 
\ =\  \spradius_G\cdot 
{\textstyle \sum_{x\in V}}   |\cp_{i, \ell-1}(x)|  \cdot  \maxeigenv(x). \nonumber 
\end{align}
For the last equality we use that $\eigenval_{\rm max}(\simpleadj_G)=\spradius(\simpleadj_G)$. This equality 
is a standard application of the Perron-Frobenius Theorem
\footnote{Note that the assumption that  $G$  is connected, implies that  $\simpleadj_G$ is a non-negative, irreducible matrix.}. 
Repeating the above  $\ell$ times in total, we get that
\begin{align}\nonumber
{\textstyle \sum_{w\in V} } |\cp_{i,\ell}(w)|   \cdot  \maxeigenv(w)
& =   \left( \spradius_G\right)^{\ell} \cdot \maxeigenv(k_i),
\end{align}
where $k_i\in V$  is such that   the root of $T_i$ belongs in  $\cp(k_i)$. 
Note that the above applies for $\ell\geq 1$. 
For $\ell=0$, it is immediate that   $\sum_{w\in V} |\cp_{i,\ell}(w)|   \cdot  \maxeigenv(w)= \maxeigenv(k_i)$.

Plugging the above into \eqref{eq:Base4Proofthrm:PowerMatrixTopologicalNorm} and rearranging, we get that
\begin{align}\label{eq:fthrm:PowerMatrixTopologicalNormTempA}
\DBounded_{\KWalks}(r) & =1+ c \cdot \sum^{k-1}_{\ell=0} \left(\delta \cdot \spradius_G \right)^{\frac{\ell}{s}} 
\cdot \sum^d_{i=1}\left( \frac{\maxeigenv(k_i)}{\maxeigenv(r)} \right)^{\frac{1}{s}}.
\end{align}
Note that the vertices $k_1, \ldots, k_d$ are the neighbours  of $r$  in the graph $G$. 

We need to bound the rightmost sum in the equation above. Recall 
that we have $\eigenval_{\rm max}(\simpleadj_G)=\spradius(\simpleadj_G)$, which
implies that  $\sum^d_{i=1}\maxeigenv(k_i)=\spradius_G\cdot \maxeigenv(r)$. 
Using this observation, we get that 
\begin{align}\label{eq:fthrm:PowerMatrixTopologicalNormTempB}
\sum^d_{i=1}\left( \frac{\maxeigenv(k_i)}{ \maxeigenv(r)} \right)^{\frac{1}{s}} &
\leq \max_{\substack{z_1, \ldots, z_d\ \in\  (0,\spradius_G):\\  \sum_i z_i=\spradius_G}} \sum^d_{i=1}(z_i)^\frac{1}{s} \ 
\leq \ \sum^d_{i=1}\left(\frac{\spradius_G}{d} \right)^\frac{1}{s} \ =\ d^{1-\frac{1}{s}} \left(\spradius_G \right)^{\frac{1}{s} }.
\end{align}
In the above series of inequalities, we use the following observations: Since we assumed that $s\geq 1$,  
it is elementary to show that for $z_1,\ldots, z_d>0$,  the function $f(z_1,\ldots, z_d)=\sum^d_{i=1}(z_i)^\frac{1}{s}$ is concave.
In the interval specified by the restrictions $z_1, \ldots, z_d \in\  (0,\spradius_G)$ and $\sum_i z_i=\spradius_G$, due to 
concavity, the function  $f(z_1,\ldots, z_d)$ attains its maximum when all $z_i$'s are equal with each other, i.e., 
$z_i=\frac{\spradius_G}{d}$, for $i=1, \ldots, d$. 

Plugging \eqref{eq:fthrm:PowerMatrixTopologicalNormTempB} into \eqref{eq:fthrm:PowerMatrixTopologicalNormTempA} 
we get that 
\begin{align}
\DBounded_{\KWalks}(r) &\leq 1+ c \cdot d^{1-\frac{1}{s}} \cdot 
\left(\spradius_G \right)^{\frac{1}{s} } \cdot {\textstyle \sum^{k-1}_{\ell=0}} \left(\delta \cdot \spradius_G \right)^{\frac{\ell}{s}} \ 
\leq \  1+  c \cdot \maxDeg^{1-\frac{1}{s}} \cdot  \left(\spradius_G \right)^{\frac{1}{s} } \cdot
{\textstyle \sum^{k-1}_{\ell=0} } \left(\delta \cdot \spradius_G \right)^{\frac{\ell}{s}}. \nonumber 
\end{align}
For the last inequality we use that $d\leq \maxDeg$.

Noting that the above bound holds for any $r\in V$,  it is immediate  to get   \eqref{eq:Target4thrm:MySpectralMatrixNorm}. 
The theorem follows.
\hfill $\Box$

\subsection{Proof of Theorem \ref{thrm:NormRedaux2WalkVector}}\label{sec:thrm:NormRedaux2WalkVector}

Let $\bC=(\bD^{\circ \frac{1}{s}})^{-1}  \cdot \Pweight_{\cP,\bxi} \cdot \bD^{\circ \frac{1}{s}}$. 
The theorem follows by showing that 
\begin{align}\label{eq:thrm:WeightNonLinearRelation}
{\textstyle \sum_{u\in V_{\cP}}} \left| \bC(r,u)  \right|  & \leq \DBounded(r) & \forall r\in V_{\cP},
\end{align}
for the walk-vector $\DBounded=\DBounded(\cP, \bD^{\circ \frac{1}{s}}, s, \delta, c)$  specified  in the statement of 
Theorem  \ref{thrm:NormRedaux2WalkVector}.

Before showing that \eqref{eq:thrm:WeightNonLinearRelation} is indeed true, let us show how we can use it to 
prove the theorem. 
From the definition of the norm $\lnorm \cdot \rnorm_{\bD, \frac{1}{s}, \infty}$, it is immediate that
\begin{align}\label{eq:proof:NormRedaux2WalkVectorA}
\lnorm \Pweight_{\cP,\bxi} \rnorm_{\bD, \frac{1}{s}, \infty} =  \lnorm \bC \rnorm_{\infty}\ =\ \max_{r} \left \{ {\textstyle \sum_{w\in V_{\cP}} } |\bC(r,w)| \right\}. 
\end{align}
Recalling that $\DBounded$ is a strictly positive vector, using  \eqref{eq:proof:NormRedaux2WalkVectorA} and \eqref{eq:thrm:WeightNonLinearRelation}
we conclude that 
\begin{align}\nonumber 
\lnorm \Pweight_{\cP,\bxi} \rnorm_{\bD, \frac{1}{s}, \infty} & \leq \max_{r}\DBounded(r) \ =\  \lnorm \DBounded \rnorm_{\infty}.
\end{align}
It remains to show that \eqref{eq:thrm:WeightNonLinearRelation} is true.
For  $r\in V_{\cP}$,  consider   the walk-tree $T=\walkT_{\cP}(r)$.   Also, consider a path $M$ of length $\ell$ in $T$ 
that emanates from the root, while let $e_1, e_2, \ldots, e_{\ell}$ be the edges on this path. 
%In  \eqref{def:WeightPath}, 
In order to define $\Pweight_{\cP,\bxi} $,  we  specify that  $M$ has weight such that
\begin{align}\nonumber %\label{eq:WeightMAgainA}
{\tt weight}(M)&=  \textstyle  \prod^{\ell}_{i=1}\bxi(e_i), 
\end{align}
where the weights $\bxi$ are specified in the statement of Theorem \ref{thrm:NormRedaux2WalkVector}.

We  use a simple telescopic trick to write the weight of $M$ slightly differently than what
we have above, i.e., involve the potential vector $\bgamma$. 

\begin{claim}\label{claim:AddPotential}
We have that
\begin{align} \nonumber
{\tt weight}(M)&=\bgamma(e_{\ell}) \frac{\bxi(e_1)}{\bgamma(e_1)} \prod^{\ell}_{i=2}\frac{\bgamma(e_{i-1})}{\bgamma(e_i)}\bxi(e_i).
\end{align}
\end{claim} 
\begin{proof}
Since, for every $e_i\in \WEdges_{\cP, r}$ we have that $\bgamma(e_i)>0$, it holds that
\begin{align}
{\tt weight}(M)&= \textstyle \prod^{\ell}_{i=1}\frac{\bgamma(e_i)}{\bgamma(e_i)}\bxi(e_i) \ 
 = \ \bgamma(e_{\ell}) \frac{\bxi(e_1)}{\bgamma(e_1)} \prod^{\ell}_{i=2}\frac{\bgamma(e_{i-1})}{\bgamma(e_i)}\bxi(e_i). \nonumber
\end{align}
The claim follows.
\end{proof}

For any $u\in V_{\cP}$ and any integer $\ell\geq 0$,  let ${\cM}(\ell, u)$ be the set of paths of length $\ell$ in $T$
that connect the root of the tree with a vertex $v\in \cp(u)$. 

From the definition of $\Pweight_{\cP, \bxi}$, recall that
\begin{align}\nonumber
\Pweight_{\cP, \bxi}(r,u)=\sum_{\ell\geq 0} \sum_{M\in \cM(\ell,u)}{\tt weight}(M).
\end{align}
Since  $\bC=(\bD^{\circ \frac{1}{s}})^{-1}\cdot \Pweight_{\cP, \bxi} \cdot \bD^{\circ \frac{1}{s}}$ and $\bD^{\circ \frac{1}{s}}$ 
is diagonal,  for any $u\in V_{\cP}$ we have that 
\begin{align}
\bC(r,u) &\textstyle = \frac{\bD^{\circ \frac{1}{s}}(u,u)}{\bD^{\circ \frac{1}{s}}(r,r)} \cdot \Pweight_{\cP, \bxi}(r,u) %\nonumber \\ 
\ =\ \frac{\bD^{\circ \frac{1}{s}}(u,u)}{\bD^{\circ \frac{1}{s}}(r,r)}\sum_{\ell\geq 0 } \sum_{M\in \cM(\ell,u)}{\tt weight}(M). \label{eq:DescOfbCru}
\end{align}
For every integer $\ell\geq 0$,  we let 
\begin{align}\nonumber
\bC^{(\ell)} &\textstyle =\frac{1}{\bD^{\circ \frac{1}{s}}(r,r)}   \left| \sum_{u\in V_{\cP}}  \sum_{M\in \cM(\ell,u)}
 {\tt weight}(M)\cdot  \bD^{\circ \frac{1}{s}}(u,u) \right |.
\end{align}
It is easy to see that $\bC^{(0)}=1$.
From the definition  of $\bC^{(\ell)}$  and \eqref{eq:DescOfbCru}, it is immediate that 
\begin{align}\label{eq:bCVsbCELL}
{\textstyle \sum_{u\in V_{\cP}} } \left| \bC(r,u) \right|  & \leq   {\textstyle \sum_{\ell\geq 0}}  \bC^{(\ell)} \ =\  1+ \textstyle{\sum_{\ell\geq 1} } \bC^{(\ell)}.
\end{align}

\noindent
Given some fixed $\ell\geq 1$,   for   $m=0, \ldots, \ell$ consider the vertex $z$ at distance  $m$ from the root of  $T$. 
Suppose  that $d_z$ is the number of children of $z$ in $T_{z}$, while let $x_1, x_2, \ldots, x_{d_z}$ 
be these children.  Recall that $T_z$ is the subtree  that is  induced by  $z$ and all its descendant in $T$.

With respect to vertex $z$, we define the quantity  $\subcont_{z}(\ell-m)$ as follows: 
For $m=\ell$, we have that
\begin{align} \label{eq:SRecurBaseWeights}
\subcont_{z}(0) &\textstyle =\sum_{u\in V_{\cP}} \mathbf{1}\{z\in \cp (u)\}\times \bD^{\circ \frac{1}{s}}(u,u).
\end{align}
For $0<m<\ell$, the quantity $\subcont_{z}(\ell-m)$ satisfies the following recursive relation:  
\begin{align}  \nonumber  
\subcont_{z}(\ell-m) &=  \textstyle
\bgamma(e_z)\sum^{d_z}_{j=1}
\frac{|\bxi(e_j)|}{\bgamma(e_{j})}
\times  \subcont_{x_j}(\ell-m-1),
\end{align}
where $e_{z}$ is the edge that connect $z$ with its parent, while $e_j$ is the edge that connects $z$
with  its child $x_j$.  Note that since we assumed that $z$ is at level $m>0$ it has a parent, i.e., 
$z$ is not the root of $T$.

Finally,  for $m=0$, i.e., $z$ and the root of $T$ are identical,  we have that   
\begin{align}\label{eq:SRecurTopWeights}
\subcont_{z}(\ell) &\textstyle = \frac{1}{\bD^{\circ \frac{1}{s}}(r,r) } \max_{e_1,e_2\in \WEdges_{\cP, r}}
\left\{ \bgamma(e_1) \cdot \frac{| \bxi(e_2)|}{\bgamma(e_2)} \right\}   
\sum^{d_z}_{j=1}\left|  \subcont_{x_j}( \ell-1) \right |.
\end{align}
Claim \ref{claim:AddPotential} and an  elementary induction imply that for any $\ell\geq 1$, we have
\begin{align}  \label{eq:CEllVsCalD}
\bC^{(\ell)}  &\leq   \subcont_{z}(\ell).
\end{align}
Furthermore, the assumption that  $\bgamma\in \mathbb{R}^{\WEdges}$ is  a $(s,\delta,c)$-potential vector
with respect to $\bxi$ and $\cP$,  together with \eqref{eq:SRecurTopWeights} 
imply that
\begin{align}\label{eq:SRecurTopWeightsB}
\subcont_{z}(\ell) & \leq \textstyle \frac{c}{\bD^{\circ \frac{1}{s}}(r,r) } \sum^{d_z}_{j=1}  \subcont_{x_j}( \ell-1) .
\end{align}
The same assumption about $\bgamma$  implies that for any $0<m<\ell$ we have that
\begin{align}\label{eq:SubContVsDeltaKWeights}
\left[ \subcont_{z}( \ell-m) \right] ^s & \leq \textstyle 
\delta \times   \sum^{d_z}_{j=1} \left[ \subcont_{x_j}(\ell-m-1) \right]^s.
\end{align}
Suppose that $k_i$ is  the $i$-th child of the root, while the  subtree $T_i$ includes vertex $k_i$ and all of its decedents.
Then, from \eqref{eq:SubContVsDeltaKWeights}  and  \eqref{eq:SRecurBaseWeights} it is elementary  to get the that
\begin{align}
 \left [ \subcont_{k_i}( \ell-1) \right]^s &\textstyle \leq  \left( \delta \right)^{\ell-1}\times \sum_{u\in V_{\cP}}| \cp_{i,\ell-1}(u)|  \cdot 
\left [ \bD^{\circ \frac{1}{s}}(u,u)\right]^s, \nonumber
\end{align}
where, recall that,  $\cp_{i,\ell-1}(u)$ is the set of copies of vertex $u$ in the subtree $T_i$ that are 
at distance $\ell-1$ from the  root of $T_i$.  
Plugging the above into \eqref{eq:SRecurTopWeightsB} yields
\begin{align}\nonumber
 \subcont_{z}(\ell)   &\textstyle \leq  \frac{c}{\bD^{\circ \frac{1}{s}}(r,r) } \sum^{d_z}_{i=1} \left( \delta^{\ell-1} \sum_{u\in V_{\cP}}| \cp_{i,\ell-1}(u)|  \cdot 
\left [  \bD^{\circ \frac{1}{s}}(u,u)\right ] ^s \right )^{\frac{1}{s}}.
\end{align}
Combining the above with \eqref{eq:CEllVsCalD} and \eqref{eq:bCVsbCELL} we get that
\begin{align}
\sum_{u\in V_{\cP}} \left| \bC(r,u) \right| &\textstyle 
\leq 1+ \frac{c}{\bD^{\circ \frac{1}{s}}(r,r) } \sum^{d_z}_{i =1} \sum_{\ell \geq 1} \left( \delta^{\ell-1} \sum_{u\in V_{\cP}}| \cp_{i,\ell-1}(u)|  \cdot 
 \left( \bD^{\circ \frac{1}{s}}(u,u)\right)^s \right)^{\frac{1}{s}}  \nonumber \\
&\textstyle = 1+  \frac{c}{\bD^{\circ \frac{1}{s}}(r,r) } \sum^{d_z}_{i =1} \sum_{\ell \geq 0} \left( \delta^{\ell} \sum_{u\in V_{\cP}}| \cp_{i,\ell}(u)|  \cdot 
\left( \bD^{\circ \frac{1}{s}}(u,u)\right)^s \right)^{\frac{1}{s}}, \nonumber  
\end{align}
where, in the last equality we  change variable.  
From the above and the definition of walk-vector $\DBounded=\DBounded(\cP, \bD^{\circ \frac{1}{s}}, s, \delta, c)$, it is immediate that
\eqref{eq:thrm:WeightNonLinearRelation} is true. 
The theorem follows. \hfill $\Box$

 \subsection{Proof of Theorem \ref{thrm:MonotoneWVector}}\label{sec:thrm:MonotoneWVector}

Note that \eqref{eq:thrm:MonotoneWVectorBB} follows immediately from \eqref{eq:thrm:MonotoneWVector}
since, by definition,   both  $\DBounded_{\cS}$ and $\DBounded_{\cP}$ have positive entries.
The rest of the proof focuses on proving \eqref{eq:thrm:MonotoneWVector}.

Firstly, we note for any $r\in V_{\cS}\subseteq V_{\cP}$, we have that
that $\walkT_{\cS}(r)$ is a  subtree of  $\walkT_{\cP}(r)$.  This follows from
 Lemma \ref{lemma:StrongSubtreeRelation}.

Suppose that the  root of $\walkT_{\cS}(r)$ has degree $d_{\cS}$, while let $T_{\cS, i}$ be the subtree
that is induced by the $i$-th child  of the root of $\walkT_{\cS}(r)$ and its descendants. Also, for $\ell\geq 0$ and any vertex $v\in V_{\cS}$,
let $\cp_{\cS, i,\ell}(v)$ be the set of copies of vertex $v$ in the subtree $T_{\cS,i}$ which is at distance $\ell$
from the root of $T_{\cS,i}$.

Following the definition of walk-vector, i.e.,  Definition \ref{def:WalkVector},  we have that
\begin{align}\label{eq:WeightNonLinearRelation4S}
\DBounded_{\cS}(r) &  \textstyle = 1+ \frac{c}{\bD(r,r)}\times \sum^{d_{\cS}}_{i=1} \sum_{\ell \geq 0} 
\left( \delta^{\ell} \cdot \sum_{w\in V_{\cS} }  |\cp_{\cS, i,\ell}(w)| \cdot (\bD(w,w))^{s}
\right)^{\frac{1}{s}}.
\end{align}
Suppose that the root of  $\walkT_{\cP}(r)$ has degree $d_{\cP}$.  In the same way as above, we define 
the subtrees $T_{\cP,i}$  and the set of copies $\cp_{\cP, i,\ell}(v)$ for  every $i\in [d_{\cP}]$ and $\ell\geq 0$. 
We also have that
\begin{align}\label{eq:WeightNonLinearRelation4P}
\DBounded_{\cP}(r) & \textstyle =1+  \frac{c}{\bD(r,r)}\times \sum^{d_{\cP}}_{i=1} \sum_{\ell \geq 0} 
\left( \delta^{\ell} \cdot \sum_{w\in V_{\cP}} |\cp_{\cP, i,\ell}(w)| \cdot (\bD(w,w))^{s}
\right)^{\frac{1}{s}}.
\end{align}
Note that, since $\delta, \bD(w,w)>0$,  the summands in both \eqref{eq:WeightNonLinearRelation4S}
and \eqref{eq:WeightNonLinearRelation4P} are non-negative quantities.  

The  theorem follows by observing   that $d_{\cS}\leq d_{\cP}$,  $V_{\cS}\subseteq V_{\cP}$ and $|\cp_{\cS, i,\ell}(w)| \leq |\cp_{\cP, i,\ell}(w)|$, 
for any $w\in V_{\cS}$, $i\in [d_{\cS}]$ and $\ell\geq 0$.  This is due to the  fact that $\walkT_{\cS}(r)$ is a subtree of  $\walkT_{\cP}(r)$.  
The theorem follows.  \hfill $\Box$

\appendix
\section{Basics in Algebra}

\subsection{Matrix and Vector Norms}
For what follows, we let  $N, M$ be  positive integers. Furthermore, we denote with $\mathbb{R}$ the set of real numbers, while
$\mathbb{C}$ is the set of complex numbers. 

For $p\geq 1$, the $p$-norm of  the vector ${\bf x}\in \mathbb{C}^N$, denoted as $\lnorm {\bf x} \rnorm_{p}$, is defined 
such that 
\begin{align}\nonumber
\lnorm {\bf x} \rnorm_{p}&=\textstyle \left( \sum^N_{i=1} |{\bf x}_i |^{p} \right)^{1/p}.
\end{align}
A matrix norms is a function $\lnorm \cdot \rnorm$  from the set of all complex matrices (of all finite orders) into $\mathbb{R}$ that satisfies
the following properties:
\begin{description}
\item[P.1] $\lnorm {\bf A} \rnorm\geq 0$, while $\lnorm {\bf A} \rnorm= 0\Leftrightarrow {\bf A}={\bf 0}$.
\item[P.2] $\lnorm \alpha  {\bf A} \rnorm = |\alpha| \lnorm {\bf A} \rnorm$, for any scalar $\alpha$.
\item[P.3] $\lnorm  {\bf A}+{\bf B} \rnorm \leq \lnorm {\bf A} \rnorm+\lnorm {\bf B} \rnorm$, for matrices of the same size.
\item[P.4] $\lnorm  {\bf A}{\bf B} \rnorm \leq \lnorm {\bf A} \rnorm \cdot \lnorm {\bf B} \rnorm$, for all conformable matrices.
\end{description}
For  each norm $\lnorm \cdot \rnorm$ on $\mathbb{R}^r$, where $r\in \{N, M \}$ there is a matrix norm that is ``induced" by 
$\lnorm \cdot \rnorm$ on $\mathbb{R}^{N\times M}$ by setting
\begin{align}\nonumber 
\lnorm {\bf A}\rnorm &=\max_{\lnorm {\bf x} \rnorm=1}\lnorm {\bf A} {\bf x}\rnorm & 
\textrm{for ${\bf A}\in \mathbb{R}^{N\times M}$ and ${\bf x}\in \mathbb{R}^{M}$.}
\end{align}
It is standard that $\lnorm {\bf A}\rnorm_{\infty}$ corresponds to the maximum absolute row sum in ${\bf A}$, 
while $\lnorm {\bf A}\rnorm_1$ corresponds to the maximum absolute column sum, i.e., 
\begin{align}\nonumber 
\lnorm {\bf A}\rnorm_{\infty}&=\max_{i}\{{\textstyle \sum_{j}|{\bf A}_{i,j}|}\} & \textrm{and} &&
\lnorm {\bf A}\rnorm_{1}&=\max_{j}\{{\textstyle \sum_{i}|{\bf A}_{i,j}|}\}.
\end{align}

\subsection{Perron-Frobenius Theorem}\label{sec:PerronFrobeniusThrm}
Let the matrix $\bM \in \mathbb{R}^{N\times N}$ be  non-negative. That is,  every entry $\bM_{i,j}\geq 0$.
We say that $\bM$ is irreducible if and only if $(\bI+\bM)^{N-1}$ is a positive matrix, i.e., all its entries are positive
numbers. 

 We can associate $\bM$ with the directed graph $G_{\bM}$ on the vertex set $[N]$, while the edge $(i,j)$
is in $G_{\bM}$ iff $\bM_{i,j}>0$. Then, $\bM$ is irreducible,  if the resulting graph $G_{\bM}$ is strongly 
connected.

In this work, it is common to use the so-called Perron-Frobenius Theorem. For the sake of keeping this
paper  self-contained,  we state this theorem below. 
\begin{theorem}[Perron-Frobenius Theorem]
Let $\bA\in \mathbb{R}^{N\times N}$  be irreducible and non-negative matrix and suppose that $N\geq 2$. Then, 
\begin{enumerate}
\item $\spradius(\bA)>0$
\item $\spradius(\bA)$ is an algebraically simple eigenvalue of $\bA$
\item there is a {\em unique} real vector ${\bf x}=({\bf x}_1, \ldots, {\bf x}_N)$ such that   $\bA\cdot {\bf x}=\spradius(\bA){\bf x}$
and ${\bf x}_1+\cdots {\bf x}_N=1$, while ${\bf x}_j>0$ for all $j\in {N}$
\item there is a {\em unique} real vector ${\bf y}=({\bf y}_1, \ldots, {\bf y}_N)$ such that   ${\bf y}^T\bA=\spradius(\bA){\bf y}^T$
and ${\bf x}_1{\bf y}_1+\cdots {\bf x}_N{\bf y}_N=1$, while ${\bf y}_j>0$ for all $j\in {N}$.
\end{enumerate}
\end{theorem}

\section{Basic Proporties of  $\Pweight_{\nbk, \bpsi}$}\label{sec:PropOfNB}

Consider the graph $G=(V,E)$.  
 For  integer $k\geq 1$ and   $\zeta\in \mathbb{R}_{\geq 0}$,  consider the walk matrix $\Pweight_{\nbk, \bpsi}$, where
$\bpsi  \in  \mathbb{R}^{\WEdges_{\nbk}}_{\geq 0}$ is such that $\bpsi(e)=\zeta$,  for all  $e\in \WEdges_{\nbk}$. 
That is, $\Pweight_{\nbk, \bpsi}$ is induced by all the non-backpacking  walks in $G$ which are of length  
$\leq k$, while for every walk-tree $\walkT_{\nbk}(r)$ every edge $e$ has weight $\bpsi(e)=\zeta$. 
Recall from Lemma \ref{lemma:WalkMatrixVsHashimoto} that we have  
\begin{align}\label{eq:NBKVsHashimotoAgain}
\Pweight_{\nbk, \bpsi} &= \textstyle \bI + \vtooedge \cdot \left( \sum^{k-1}_{\ell=0} \zeta^{\ell+1} \cdot \NBMatrix^{\ell}_G\right)\cdot \oedgetov,
\end{align} 
where $\NBMatrix_G$ is the non-backtracking matrix of $G$, while 
$\vtooedge$, $\oedgetov$ are   $V\times \OrntEdges$ and $\OrntEdges\times V$  matrices, respectively,   such that 
for any $v\in V$ and for any $(x,z)\in \OrntEdges$ we have 
\begin{align}\nonumber %\label{def:HasVertex2EdgeMatricesAgain}
\vtooedge(v,(x,z)) &= {\bf 1}\{ x=v\} 
& \textrm{and} &&
\oedgetov((x,z), v) &=  {\bf 1}\{ z=v\}
\end{align}

In what follows, we present we show that   $\Pweight_{\nbk, \bpsi} $ is symmetric.

\begin{lemma}\label{lemma:WNBKSymmetric}
For any $k\geq 0$ and any $\zeta>0$, the matrix $\Pweight_{\nbk, \bpsi} $ is  symmetric. 
\end{lemma}
\begin{proof}
From the definition of $\Pweight_{\nbk, \bpsi}$  in \eqref{eq:NBKVsHashimotoAgain}, 
we note that, since the identity matrix $\bI$  is symmetric, the lemma follows by showing 
that for any $\ell\geq 0$  the matrix 
\begin{align}\nonumber
\bX &=\textstyle \vtooedge \cdot  \NBMatrix^{\ell}_G \cdot \oedgetov,
\end{align}
is symmetric.  Note that $\bX$ is a $V\times V$ matrix. 

Elementary calculations, imply that for any $u,w\in V$, different with each other,  we have that
\begin{align}\nonumber
\bX(u,w) &= {\textstyle \sum_{z\in V} \sum_{y\in V} }  \NBMatrix^{\ell }_{(u,z), (y,w)}.
\end{align}
From the  PT invariance of $\NBMatrix^{\ell}$, i.e.,  \eqref{def:PTInvariance}, we get the following: 
\begin{align}
\bX(u,w)& =  {\textstyle \sum_{z} \sum_{y}}  \NBMatrix^{\ell }_{(u,z), (y,w)}
 = {\textstyle  \sum_{z} \sum_{y} }  \NBMatrix^{\ell }_{(w,y), (z,u)} 
 =\bX(w,u). \nonumber 
\end{align}
It is in the second equality that we use the PT invariance of $(\NBMatrix_G)^{\ell}$. 

The above implies that $\bX$ is a symmetric matrix.  The lemma follows. 
\end{proof}

\section{Results for Unbounded spectral radius}\label{sec:UnboundedRadius}

In this section we present results which are analogous to  Theorems \ref{thrm:Ising4SPRadius}
and \ref{thrm:HC4SPRadius} for graphs whose adjacency matrix $\simpleadj_G$
has unbounded spectral radius, i.e., $\spradius(\simpleadj_G)$ grows with the size of the graph $n$.
We start with the Ising model.

\begin{theorem}\label{thrm:Ising4SPRadiusUnbounded}
For any $\delta\in (0,1)$ and for $\spradius_G>1$,     consider  the graph  $G=(V,E)$  whose 
adjacency matrix $\simpleadj_G$ has  spectral radius $\spradius_{G}$. Also,   let $\mu_G$  be 
the Ising model on $G$ with zero external field and parameter $\beta\in \UnIsing(\spradius_G,\delta)$.  

If $\spradius_G$ is unbounded,  there is a constant  $C=C( \delta)$ such that the mixing time of 
the  Glauber dynamics  on $\mu_G$  is  at most $C n^{1+\frac{1}{\delta}}$.
\end{theorem}

Note that there is  a discrepancy between the two mixing times we get from Theorems \ref{thrm:Ising4SPRadius}
and \ref{thrm:Ising4SPRadiusUnbounded}.
 This  has to do with how  we use  spectral independence, i.e.,  once this has been  established,  to  bound  the mixing 
 time of Glauber  dynamics. In that respect,  we make a direct use of results from  \cite{VigodaSpectralInd,VigodaSpectralIndB}
and hence, the discrepancy comes from the fact that  these two papers derive different bounds on the mixing time. 
In the related literature there is also the work in  \cite{YitongAllDegreeOptMix} which can be considered 
for  the unbounded case. However, it seems that the result in  \cite{YitongAllDegreeOptMix}
requires from the Gibbs distribution  to satisfy a strong condition called Gibbs uniqueness. The range of the parameters 
we consider for the Gibbs distribution in  Theorems \ref{thrm:Ising4SPRadiusUnbounded} and
 \ref{thrm:Ising4SPRadius} is beyond this region.

Theorem \ref{thrm:Ising4SPRadius}   improves on results in \cite{Hayes06} for the Ising model 
by allowing a wider rage for $\beta$.
For $\spradius_G=\omega(1)$, the range of $\beta$  is wider than that in \cite{Hayes06} 
but the increase   is with regard to smaller order terms.  Note that the  mixing time there is $O(n\log n)$.

\begin{proof}[Proof of Theorem \ref{thrm:Ising4SPRadiusUnbounded}] 

The proof of Theorem \ref{thrm:Ising4SPRadiusUnbounded} is almost identical to that
of Theorem \ref{thrm:Ising4SPRadius}. 
Specifically, using Lemma \ref{lemma:IsingInfNormBound}  and 
Theorem \ref{thrm:Recurrence4InfluenceEigenBound} in the same way as in the proof of 
Theorem \ref{thrm:Ising4SPRadius} we get the following:
for   $\Lambda\subseteq V$ and $\tau\in \{\pm 1\}^{\Lambda}$,  the pairwise influence 
matrix $\infmatrix^{\Lambda,\tau}_{G}$,  induced by $\mu_G$,  satisfies that  
\begin{align}\label{eq:thrm:Ising4SPRadiusBasisUnbounded}
\spradius(\infmatrix^{\Lambda,\tau}_{G})\leq \delta^{-1}.
\end{align}
The theorem follows as a corollary from \eqref{eq:thrm:Ising4SPRadiusBasisUnbounded},  
Theorem \ref{thrm:SPCT-INDClosed}   and \eqref{eq:MivingTimeVsSPGap}.
\end{proof}

We proceed with the Hard-core model on a graph with unbounded spectral radius. We prove the following result.

\begin{theorem}\label{thrm:HC4SPRadiusUnbounded}
For any $\epsilon \in (0,1)$,  $\maxDeg\geq 2$ and  $\spradius_G\geq 2$,  consider the graph $G=(V,E)$  
of maxim degree $\maxDeg$,  whose adjacency matrix $\simpleadj_G$ has  spectral radius $\spradius_{G}$. Also,   let $\mu_G$ be 
the Hard-core model on $G$ with  fugacity $\lambda\leq (1-\epsilon)\lcritical(\spradius_G)$.

If $\spradius_G$ is unbounded, there are constants  $C=C(\epsilon)$ and $C'=C'(\epsilon)$ such that the mixing 
time of the   Glauber dynamics on $\mu_G$  is at most  $C n^{2+C'\sqrt{\maxDeg/\spradius_G}}$.
\end{theorem}

Note that   the above result implies a polynomial bound for the mixing time only for the case
where the ratio $\sqrt{\maxDeg/\spradius_G}$ is bounded. 

The above improves on \cite{Hayes06} by allowing a wider rage for $\lambda$.  The  improvement is 
the same as in the bounded case, i.e., get an extra factor $e$.  Note, though,  the  extra  condition that
the ratio of maximum degree $\maxDeg$ over   $\spradius_G(\simpleadj_G)$ must be bounded.  

\begin{proof}[Proof of Theorem \ref{thrm:HC4SPRadiusUnbounded}]
The proof of Theorem \ref{thrm:HC4SPRadiusUnbounded} is almost identical to that of Theorem \ref{thrm:HC4SPRadius}.
Specifically, using  Theorem \ref{thrm:SPIndependenceHC} we get the following:
There is a constant $z>0$  which depend on $\epsilon$ such that for any $\Lambda\subseteq V$ and 
$\tau\in \{\pm 1\}^{\Lambda}$ we have that
\begin{align} \nonumber  
\textstyle \spradius  \left( \infmatrix^{\Lambda,\tau}_{G} \right) &\leq 
1+  e^3   \left( {\maxDeg}/{\spradius_G} \right)^{1/2}  z^{-1}.
\end{align}
The theorem follows from the above inequality,  Theorem   \ref{thrm:SPCT-INDClosed}   and \eqref{eq:MivingTimeVsSPGap}.
\end{proof}

\section{Rapid mixing bound for General Gibbs Distributions}\label{sec:GeneralGibbs}

Note that  Theorems \ref{thrm:Recurrence4InfluenceEigenBound} and \ref{thrm:Recurrence4InfluenceEigenBoundNonLinear}
do not provide bounds for the spectral radius of the influence matrix  only for the Ising model and the Hard-core model.
They are general results and apply to any Gibbs distribution on $G$. In that respect, it might be interesting to write 
the corresponding bounds   we get from these two theorems on the mixing time  of  Glauber dynamics on a general two-spin
Gibbs distribution.
%  In what follows, we focus on the cases where the spectral radius of $\simpleadj_G$ is bounded.

%\subsection{Bounded spectral radius - $O(n\log n)$ mixing}
%

Consider a graph $G=(V,E)$ and assume that $\spradius(\simpleadj_G)$ is bounded. Recall that this assumption
implies that the maximum degree $\maxDeg$ is also bounded. 

For what follows, we need to introduce few useful concepts from \cite{VigodaSpectralIndB}. 
For  $S\subset V$, let the Hamming graph $\cH_S$ be the graph whose vertices correspond to 
 the configurations $\{\pm 1\}^S$, while two configurations are adjacent iff they differ at the assignment
 of a single vertex, i.e., their Hamming distance  is  one. 
 Similarly, any subset $\Omega_0\subseteq   \{\pm 1\}^S$ is considered to be connected if the subgraph
 induced by $\Omega_0$ is connected.   
 
 A distribution $\mu$ over $\{\pm 1\}^V$ is considered to be {\em totally
 connected} if for every nonempty $\Lambda \subset V$ and every boundary condition $\tau$ at $\Lambda$
 the set of configurations in the support of $\mu(\cdot \ |\ \Lambda, \tau)$ is connected. 

We need to remark here that all Gibbs distribution with soft-constraints such as the Ising model are 
totally connected in a trivial way.  The same holds for the Hard-core model and this follows with standard arguments.  

\begin{definition}
For some number $b\geq 0$,  we say that a distribution $\mu$ over $\{\pm 1\}^V$ is $b$-marginally
bounded if for every $\Lambda\subset V$ and any configuration $\tau$ at $\Lambda$ we have the following:
for any $V\setminus \Lambda$ and for any $x\in\{\pm 1\} $ which is in the support of $\mu_u(\cdot \ |\ \Lambda, \tau)$,
we have that
\begin{align}\nonumber 
\mu_u(x\ |\ \Lambda, \tau)\geq b. 
\end{align}
\end{definition}

The following result is a part of Theorem 1.9 from \cite{VigodaSpectralIndB} (arxiv version). This is one of the results in the
literature that improves on Theorem \ref{thrm:SPCT-IND}. 

\begin{theorem}[\cite{VigodaSpectralIndB}]\label{thrm:SPINLOGN} Let the integer $\maxDeg\geq 3$ and  $b,\eta\in \mathbb{R}_{>0}$. 
Consider  $G=(V,E)$ a graph with $n$ vertices and maximum degree  $\maxDeg$. Also, let $\mu$ be a totally connected Gibbs
distribution on $\{\pm 1\}^V$. 

If $\mu$ is both $b$-marginally bounded and $\eta$-spectrally independent, then there are constants $C_1, C_2>0$  such
the Glauber dynamics for $\mu$ exhibits mixing time
\begin{align}\nonumber 
T_{\rm mix} &\leq  \left( \frac{\maxDeg}{b} \right)^{C_1\left ( \frac{\eta}{b^2}+1 \right)} \times C_2 \left( n \log n \right).
\end{align} 
\end{theorem}

We combine Theorems \ref{thrm:Recurrence4InfluenceEigenBound} and \ref{thrm:Recurrence4InfluenceEigenBoundNonLinear} 
with the above to get the following rapid mixing results for the Glauber dynamics on ageneral Gibbs distributions on
$G$ with bounded $\spradius(\simpleadj_G)$.

\begin{theorem}[$\delta$-contraction]\label{thrm:DeltaConMixing}
Let   the integer $\maxDeg\geq 3$, $b,\epsilon \in (0,1]$,   $\spradius_G>1$ and  $\beta,\gamma, \lambda \in  \mathbb{R}$  
such that  $0\leq \beta\leq  \gamma$,  $\gamma >0$ and $\lambda>0$.

Consider  $G=(V,E)$ a graph with $n$ vertices and maximum degree at most $\maxDeg$, 
while   $ \simpleadj_G$ is of spectral radius $\spradius_G$.
 Also, let $\mu$ be a totally connected Gibbs distribution on $\{\pm 1\}^V$ specified by the parameters $\beta,\gamma$
 and $\lambda$ as in \eqref{def:GibbDistr}. 

Setting  $\delta =\frac{1-\epsilon}{\spradius_G}$, suppose that $\mu$ is  $b$-marginally bounded while 
the set  of functions $\{ \logtrecur_d\}_{d\in [\maxDeg]}$  specified  with respect to  $(\beta,\gamma,\lambda)$ 
exhibits   $\delta$-contraction. 

There are  constants $C_1, C_2>0$  such the Glauber dynamics on $\mu$ exhibits mixing time $T_{\rm mix}$
such that 
\begin{align}\nonumber 
T_{\rm mix} &= \left( \frac{\maxDeg}{b} \right)^{C_1\left ( \frac{1}{\epsilon b^2}+1 \right)} \times C_2 \left( n \log n \right).
\end{align} 
\end{theorem}

 Theorem \ref{thrm:DeltaConMixing} is straightforward  from Theorems \ref{thrm:SPINLOGN} and \ref{thrm:Recurrence4InfluenceEigenBound}. 
 For this reason, we omit its proof. 

\begin{theorem}[$(s,\delta,c)$-potential  function]\label{thrm:GoodPotFMixing}
Let   the integer $\maxDeg\geq 3$, $b,\epsilon \in (0,1]$,   $\zeta>0$,  $\spradius_G>1$  and  $\beta,\gamma, \lambda \in  \mathbb{R}$  
such that  $0\leq \beta\leq  \gamma$,  $\gamma >0$ and $\lambda>0$.

Consider  $G=(V,E)$ a graph with $n$ vertices and maximum degree at most $\maxDeg$, 
while   $ \simpleadj_G$ is of spectral radius $\spradius_G$.
 Also, let $\mu$ be a totally connected Gibbs distribution on $\{\pm 1\}^V$ specified by the parameters $\beta,\gamma$
 and $\lambda$ as in \eqref{def:GibbDistr}. 

Setting  $\delta= \frac{1-\epsilon}{\spradius_{G}}$ and $c=\frac{\zeta}{\spradius_G}$, suppose 
that $\mu$ is  $b$-marginally bounded, while there is a  $(s,\delta,c)$-potential  function $\potF$ with 
respect to  $(\beta,\gamma,\lambda)$.

There is  a constant $C>0$  such the Glauber dynamics on $\mu$ exhibits mixing time $T_{\rm mix}$
such that 
\begin{align}\nonumber 
T_{\rm mix} &= \left( \frac{\maxDeg}{b} \right)^{C \left ( \frac{\eta}{ b^2}+1 \right)} \times \left( n \log n \right), 
\end{align} 
where $\eta=1+ \zeta \cdot (1-(1-\epsilon)^s)^{-1} \cdot \left( {\maxDeg}/{\spradius_G} \right)^{1-\frac{1}{s}}$.
\end{theorem}

Theorem \ref{thrm:GoodPotFMixing} is straightforward  from Theorems \ref{thrm:SPINLOGN} and 
\ref{thrm:Recurrence4InfluenceEigenBoundNonLinear}.   For this reason, we omit its proof. 

\end{document}